\documentclass[12pt]{amsart}

\usepackage{amssymb,amsfonts,amsmath}
\usepackage{mathtools,bm}
\usepackage{enumerate}
\usepackage[letterpaper,ignoreall,margin=1in]{geometry}
\usepackage[colorlinks]{hyperref}
\usepackage{booktabs}

\usepackage[mathscr]{eucal}

\newtheorem*{thm*}{Theorem}
\newtheorem{thm}{Theorem}[section]
\newtheorem{theorem}[thm]{Theorem}

\newtheorem{corollary}[thm]{Corollary}
\newtheorem{prop}[thm]{Proposition}
\newtheorem{lemma}[thm]{Lemma}

\theoremstyle{remark}
\newtheorem{remark}[thm]{Remark}

\theoremstyle{definition}
\newtheorem{example}[thm]{Example}
\newtheorem{defn}[thm]{Definition}
\newtheorem{notation}[thm]{Notation}
\newtheorem{conjecture}[thm]{Conjecture}

\newcommand{\ba}{\mathbf{a}}
\newcommand{\bb}{\mathbf{b}}

\newcommand{\bd}{\mathbf{d}}
\newcommand{\be}{\mathbf{e}}

\newcommand{\bk}{\mathbf{k}}
\newcommand{\bn}{\mathbf{n}}

\newcommand{\bzero}{\mathbf{0}}
\newcommand{\bdelta}{\bm{\delta}}

\renewcommand{\O}{\mathcal{O}}

\newcommand{\PP}{\mathbb{P}}

\newcommand{\cA}{\mathcal{A}}
\newcommand{\cB}{\mathcal{B}}
\newcommand{\cC}{\mathcal{C}}
\newcommand{\cE}{\mathcal{E}}
\newcommand{\cL}{\mathcal{L}}
\newcommand{\cO}{\mathcal{O}}

\renewcommand{\r}{r}  
\renewcommand{\d}{\partial}

\DeclareMathOperator{\Sing}{Sing}
\DeclareMathOperator{\Span}{span}

\DeclareMathOperator{\Hom}{Hom}
\DeclareMathOperator{\cHom}{\mathcal{H}\mathit{om}}
\DeclareMathOperator{\rank}{rank}
\DeclareMathOperator{\codim}{codim}
\DeclareMathOperator{\img}{img}
\DeclareMathOperator{\per}{per}
\DeclareMathOperator{\pr}{pr} 
\DeclareMathOperator{\dett}{det}
\DeclareMathOperator{\Seg}{Seg}
\DeclareMathOperator{\Sub}{Sub}
\DeclareMathOperator{\Spec}{Spec}
\DeclareMathOperator{\Tr}{Tr}
\DeclareMathOperator{\mult}{mult}

\newcommand{\defining}[1]{\textbf{#1}}

\newcommand{\rad}[1]{\operatorname{supp}(#1)}

\title{Geometric lower bounds for generalized ranks}
\author{Zach Teitler}
\email{zteitler@boisestate.edu}
\address{Department of Mathematics \\
1910 University Drive \\
Boise State University \\
Boise, ID 83725-1555 \\
USA}
\date{\today}
\subjclass[2010]{15A21, 15A69, 14N15}
\keywords{Waring rank, secant varieties}

\begin{document}

\bibliographystyle{amsalpha}       

\begin{abstract}
We revisit a geometric lower bound for Waring rank of polynomials
(symmetric rank of symmetric tensors)
of \cite{Landsberg:2009yq}
and generalize it to a lower bound for rank with respect to arbitrary varieties,
improving the bound given by the ``non-Abelian'' catalecticants recently introduced
by Landsberg and Ottaviani.
This is applied to give lower bounds for ranks of multihomogeneous polynomials
(partially symmetric tensors);
a special case is the simultaneous Waring decomposition problem
for a linear system of polynomials.
We generalize the classical Apolarity Lemma to multihomogeneous polynomials
and give some more general statements.
Finally we revisit the lower bound of \cite{MR2842085},
and again generalize it to multihomogeneous polynomials and some more general settings.
\end{abstract}

\maketitle

\section{Introduction}\label{section: introduction}

Let $F$ be a homogeneous polynomial of degree $d$ in several variables.
A \defining{power sum decomposition} of $F$ is
an expression $F = c_1 \ell_1^d + \dotsb + c_r \ell_r^d$ in which
the $\ell_i$ are linear forms and the $c_i$ are scalars.
The \defining{length} of a power sum decomposition is the number $r$ of terms.
The \defining{Waring rank} of $F$, denoted $r(F)$,
is the least length $r$ of a power sum decomposition of $F$.
A \defining{Waring decomposition} of $F$ is a power sum decomposition of minimal length.
For example,
\begin{gather}
  xy = \frac{1}{4} \Big( (x+y)^2 - (x-y)^2 \Big) , \\
  \label{eq: xyz}
  xyz = \frac{1}{24} \Big( (x+y+z)^3 - (x+y-z)^3 - (x-y+z)^3 + (x-y-z)^3 \Big),
\end{gather}
so $\r(xy) \leq 2$ and $\r(xyz) \leq 4$.

In fact, both of these inequalities are actually equalities.
Several lower bounds for Waring rank have been developed.
The earliest lower bound, and the basis for all the rest, involves \textit{catalecticants},
which were introduced by Sylvester in 1851.
We review catalecticants in Section~\ref{section: catalecticants}.
It turns out that the catalecticant lower bound gives $\r(xy) \geq 2$ and $\r(xyz) \geq 3$.
Clearly, an improvement is desirable.
One such improvement was given in \cite{Landsberg:2009yq},
which showed that the catalecticant bound for rank could be improved by
adding the dimension of a certain set of singularities.
We review this in Section~\ref{section: improved lower bounds}; it gives $\r(xyz) \geq 4$.
Another improvement given in \cite{MR2842085} also yields $\r(xyz) \geq 4$;
it is reviewed in Section~\ref{section: ranestad-schreyer bounds}.

Here we are interested in more general notions of rank
for which, as we will see, there are well-known (generalized) catalecticant lower bounds.
We develop improvements to these lower bounds
analogous to the improvements in \cite{Landsberg:2009yq}.
In some cases we are also able to develop improvements analogous to the one in \cite{MR2842085}.

Here is an example in which we are able to determine a generalized rank.
Fix $a, b > 0$ and consider $F = x_1 \dotsm x_a y_1 \dotsm y_b$, a homogeneous form of degree $a+b$ in $a+b$ variables,
and also a bihomogeneous form of bidegree $(a,b)$.
Generalizing Waring rank, we consider expressions of $F$ as a sum of terms of the form $\ell(x)^a m(y)^b$:
\[
  x_1 \dotsm x_a y_1 \dotsm y_b = \sum_{i=1}^r \ell_i(x_1,\dotsc,x_a)^a m_i(y_1,\dotsc,y_b)^b,
\]
each $\ell_i$ and $m_i$ a linear form.
We are able to show that the least number of terms in such an expression is $2^{a+b-2}$,
see Example~\ref{example: bihomogeneous product} and Example~\ref{example: bihomogeneous product ranestad schreyer}.

And here is an example which we are not able to resolve.
Fix $s, t > 0$ and $F = x_1 \dotsm x_{st}$.
Consider expressions of $F$ as a sum of $s$-th powers of homogeneous forms of degree $t$:
\[
  x_1 \dotsm x_{st} = \sum_{i=1}^r G_i(x_1,\dotsc,x_{st})^s,
\]
each $\deg G_i = t$.
It is easy to see that there is such an expression with $r = 2^{s-1}$,
but we are not able to determine whether or not this is the shortest possible expression.

As the reader will see, this is just one of many open questions.

\bigskip

In the remainder of this introduction we describe, in steps of increasing generality,
the notions of rank in which we are interested.
In Section~\ref{section: catalecticants} we review catalecticants for classical Waring rank (as above)
and in the generalized settings.
In Section~\ref{section: improved lower bounds} we introduce our improvements to the catalecticant lower bounds
generalizing \cite{Landsberg:2009yq}.
We review Apolarity Lemmas in Section~\ref{section: apolarity lemmas}.
Finally in Section~\ref{section: ranestad-schreyer bounds} we develop, at least for some cases,
improvements to the catalecticant lower bounds generalizing \cite{MR2842085}.


We work over an algebraically closed field $\Bbbk$ of characteristic $0$.

\subsection{Classical Waring rank}

Recall that a homogeneous polynomial of degree $d$ in $n$ variables is called an \defining{$n$-ary $d$-form}
or an \defining{$n$-ary $d$-ic}; thus, for example, a homogeneous polynomial of degree $5$ in $2$ variables
is a $2$-ary $5$-form, or binary quintic.

If $F$ is a quadratic form, the Waring rank of $F$ is equal to its rank as a quadratic form.
Also, ranks of binary forms are understood,
thanks to 19th century work by Sylvester
\cite{Sylvester:1851kx}, \cite{Sylvester:1851wd}, \cite{Sylvester:1886uq},
Gundelfinger \cite{Gundelfinger:1886fk}, and others, see \cite[Ch.~XI]{MR2850282};
for more recent treatments see for example \cite{MR859177}, \cite{MR2754189}, \cite{Reznick:2013uq}.
Ranks of ternary cubics are well known,
see for example \cite{comonmour96}, \cite[\textsection 8]{Landsberg:2009yq}.
Also the theorem of
Alexander and Hirschowitz \cite{MR1311347}, \cite{MR2387598}, \cite{MR1813598}, \cite{MR1943904}, \cite{MR2886162}
gives the ranks of \emph{general} forms, meaning those in a dense open subset of the space of forms.
Namely, a general $d$-form in $n$ variables has rank
\[
  \left \lceil \frac{1}{n} \binom{n+d-1}{n-1} \right \rceil,
\]
with a short list of exceptions: when $d=2$, the general rank is $n$ (instead of $(n+1)/2$);
when $(n,d) = (3,4), (4,4), (5,4), (5,3)$ the general rank is respectively $6, 10, 15, 8$ (instead of $5, 9, 14, 7$).

However for an arbitrary given form,
it is surprisingly nontrivial to determine $\r(F)$.
There is no known effective way to determine if a given $F$ is general,
so that it has the rank given by Alexander--Hirschowitz.
Nevertheless, ranks have been worked out in a number of cases.
Typically one can express $F$ as a sum of $d$th powers of linear forms,
giving an upper bound on $\r(F)$.
For example, the power sum decomposition
\[
  x_1 \dotsm x_n
    = \frac{1}{2^{n-1} n!}
      \sum_{(\epsilon_2,\dotsc,\epsilon_n) \in \{\pm1\}^{n-1}}
        \epsilon_2 \dotsm \epsilon_n (x_1 + \epsilon_2 x_2 + \dotsb + \epsilon_n x_n)^n
\]
shows $\r(x_1 \dotsm x_n) \leq 2^{n-1}$.
Computational methods to find power sum decompositions have been developed \cite{MR2736103}, \cite{Oeding2013}.
Finding good upper bounds is an interesting challenge,
see \cite{MR2383331}, \cite[\textsection 5]{Landsberg:2009yq}, \cite{Jelisiejew:2013fk},
\cite{Ballico:2013sf}, \cite{Blekherman:2014eq}.

The problem we are concerned with here, however, is to give a lower bound.
We will review some ideas for lower bounds in the following sections;
for now, I will simply list all the determinations of Waring ranks of which I am aware.

That $\r(x_1 \dotsm x_n) = 2^{n-1}$ was shown for $n=4$ in \cite{Landsberg:2009yq} (2010),
and for all $n$ in \cite{MR2842085} (2011),
which showed more generally that $\r((x_1 \dotsm x_n)^d) = (d+1)^{n-1}$.
The ranks of arbitrary monomials and sums of pairwise coprime monomials
were determined in \cite{Carlini20125} (2012).
Plane quartics have been studied \cite{Kleppe:1999fk}, \cite{Bernardi201134}, \cite{Paris:2013fk} in great detail.
In \cite{Teitler:2013uq}, Waring ranks are determined for the defining equations
of hyperplane arrangements which are mirror arrangements for complex reflection groups
satisfying a certain hypothesis.
A few isolated examples have been computed: $\r(x(y_1^2+\dotsb+y_n^2)) = \r(x(y_1^2+\dotsb+y_n^2+x^2)) = 2n$
and $\r(x_1 y_1 z_1 + \dotsb + x_n y_n z_n) = 4n$
\cite[\textsection7]{Landsberg:2009yq}, \cite{Ventura:2013bh};
$\r(x_0^2 y_0 - (x_0+x_1)^2 y_1 + x_1^2 y_2) = 9$ \cite{Buczynska:2013zh}.
See \cite{MR1096187} for the forms $(x_1^2+\dotsb+x_n^2)^{d/2}$.
Carlini, Catalisano, and Chiantini have shown very recently \cite{Carlini:2014rr}
that $\r(F(x_1,\dotsc,x_n) + y_1^d + \dotsb + y_s^d) = \r(F) + s$
and $\r(F(x_1,x_2) + G(y_1,y_2)) = \r(F) + \r(G)$
(a conjecture of Strassen asserts that this should hold for all $F$ and $G$ involving any number of variables).
See also \cite{Woo:2014ix}.
Together with the previously mentioned quadratic and binary forms, ternary cubics, and general forms,
this is, as far as I know, a complete list of all forms whose Waring ranks have been determined.

The recency of these results is somewhat surprising given the long history
and widespread interest in questions about Waring rank,
going back at least to Sylvester and 19th century investigations of apolarity and canonical forms.
A wide range of applications has emerged in other areas of mathematics, statistics, engineering, and sciences.
See for example \cite{comonmour96}, \cite[Chapter 4]{MR2723140}, \cite{MR2447451},
\cite{MR2736103}, \cite{MR2895192}, \cite{MR2865915}.
For comprehensive introductions to this subject and its history and applications
see \cite{MR1735271}, \cite{MR2865915}.
We digress to briefly describe some of the applications.
The rank of a polynomial may be considered a measure of its complexity, as in the field of
\emph{geometric complexity theory} \cite{Landsberg:2013ys}.
The linear functional on the space of $d$-forms corresponding to the inner product with $\ell^d = (a_1 x_1 + \dotsb + a_n x_n)^d$
is given by evaluation at the point $(a_1,\dotsc,a_n)$ (up to factorial factors),
so a power sum decomposition of a polynomial corresponds to a
\emph{decomposition of a linear functional into a combination of atomic measures}.
See \cite{MR1096187} for applications of this idea to number theory, functional analysis,
numerical analysis (quadrature problems), and spherical designs.
In statistics, a power sum decomposition of a polynomial corresponds to
\emph{mixture model of joint distributions of independent identically distributed random variables}.
As an extremely simple example of this, a random variable $X$ on a finite set with $P(X=i) = p_i$ for $1 \leq i \leq n$
may be described by the linear form $\ell_X = p_1 x_1 + \dotsb + p_n x_n$;
then the coefficients of $\ell_X^d$ give the joint distribution of $d$ independent identically distributed copies of $X$.
For a given random variable $Y$ encoded in the coefficients of a $d$-form $F$, a power sum decomposition of $F$
corresponds to an expression of $Y$ as a mixture model of joint distributions
of independent identically distributed random variables.
(In this context one considers power sum decompositions with the extra constraints that all coefficients must be
nonnegative real numbers summing to $1$.)
This is related to the \emph{PARAFAC/CANDECOMP} decomposition.
This basic idea plays a role in applications such as \emph{blind source separation} in signal processing,
where an observed signal must be decomposed into simple single sources.
Much more discussion and detail may be found in the references above.

Despite this long history and widespread interest we are still left with the simple question: given $F$, what is $\r(F)$?

\begin{example}\label{example: determinant}
Let $\det_n$ be the determinant of an $n \times n$ generic matrix,
\[
  \det\nolimits_n
    = \det
      \begin{pmatrix}
        x_{1,1} & \dots & x_{1,n} \\
         \vdots &       &  \vdots \\
        x_{n,1} & \dots & x_{n,n}
      \end{pmatrix} ,
\]
a polynomial of degree $n$.
Since $\det_n$ is a sum of $n!$ monomials each with rank $2^{n-1}$, $\r(\det_n) \leq 2^{n-1} n!$;
in particular $\r(\det_3) \leq 24$.

There is a remarkable improvement of this, recently discovered by Derksen \cite{Derksen:2013sf}:
\begin{equation}\label{eq: derksen formula}
\begin{split}
 \dett_3 = \frac{1}{2} \Big( & (x_{13}+x_{12})(x_{21}-x_{22})(x_{31}+x_{32}) \\
   & + (x_{11}+x_{12})(x_{22}-x_{23})(x_{32}+x_{33}) \\
   & + 2 x_{12}(x_{23}-x_{21})(x_{33}+x_{31}) \\
   & + (x_{13}-x_{12})(x_{22}+x_{21})(x_{32}-x_{31}) \\
   & + (x_{11}-x_{12})(x_{23}+x_{22})(x_{33}-x_{32})
   \Big) .
\end{split}
\end{equation}
Each of the $5$ terms is a product of $3$ linear forms, $\ell_1 \ell_2 \ell_3$.
This has rank $4$ by substitution in \eqref{eq: xyz}.
Thus $\r(\det_3) \leq 20$.
As noted by Derksen, an improved upper bound for larger determinants
follows by Laplace expansion by complementary minors in the first $3$ rows.
We will see in Example~\ref{example: derksen} that
\begin{equation}\label{eq: derksen bound}
  \r(\dett_n) \leq \left(\frac{5}{6}\right)^{\lfloor n/3 \rfloor} 2^{n-1} n! .
\end{equation}

What about lower bounds for $\r(\det_n)$?
We will see in Example~\ref{example: determinant catalecticant}
that Sylvester's catalecticant lower bound gives
$\r(\det_n) \geq \binom{n}{\lfloor n/2 \rfloor}^2$
(so $\r(\det_3) \geq 9$).
A different lower bound introduced by Ranestad and Schreyer \cite{MR2842085},
together with a result of Masoumeh Sepideh~Shafiei \cite{Shafiei:ud},
gives $\r(\det_n) \geq \frac{1}{2} \binom{2n}{n}$ (so $\r(\det_3) \geq 10$),
see Example~\ref{example: ranestad-schreyer-shafiei determinant bound}.
The lower bound of \cite{Landsberg:2009yq} gives
$\r(\det_n) \geq \binom{n}{\lfloor n/2 \rfloor}^2 + n^2 - (\lfloor n/2 \rfloor + 1)^2$
(so $\r(\det_3) \geq 14$), see Example~\ref{example: determinant singularity improvement}.
Among these lower bounds for $\r(\det_n)$, the Ranestad--Schreyer--Shafiei lower bound grows most quickly,
giving the best result for $n \geq 5$.
However all three of these lower bounds grow exponentially in $n$, much more slowly than the factorial upper bound.

It would be quite interesting to determine $\r(\det_n)$, or even to give better bounds.
See \cite{Shafiei:ud}, \cite{Shafiei:2013fk} for further discussion of determinants along with
permanents (see Example~\ref{example: matrix linear series}), Pfaffians, symmetric determinants, etc.

We will revisit the generic determinant
in Examples~\ref{example: determinant bihomogeneous rank},
\ref{example: determinant catalecticant},
\ref{example: determinant bihomogeneous catalecticant},
\ref{example: determinant singularity improvement}.
\end{example}

\begin{remark}
The rank of a general form of degree $n$ in $n^2$ variables is $\lceil \frac{1}{n^2} \binom{n^2+n-1}{n} \rceil$.
This is greater than $2^{n-1} n!$ for $n \geq 4$, so $\det_n$ has less than the general rank for $n \geq 4$.
However the general rank of a cubic in $9$ variables is $19$ while $14 \leq \r(\det_3) \leq 20$.
Even to determine whether $\det_3$ has greater than, equal to, or less than the general rank is an interesting challenge.
\end{remark}

\begin{remark}
Waring decompositions are typically not unique.
Of course one may replace a term $c_i \ell_i^d$ with $(c_i/\lambda^d)(\lambda \ell_i)^d$, and one may reorder the terms.
Even ignoring these trivial changes---say, by considering the unordered set $\{[\ell_1],\dotsc,[\ell_r]\}$,
with each linear form considered just up to scalar multiple---uniqueness may still fail.
For example,
\[
  xy = \frac{1}{4}(x+y)^2 - \frac{1}{4}(x-y)^2 = \frac{1}{12}(x+3y)^2 - \frac{1}{12}(x-3y)^2 .
\]
Uniqueness and non-uniqueness have been studied, see for example
\cite{MR2238925}, \cite{MR2439429}, \cite{MR2945601}, \cite{MR3056286}.
\end{remark}

\begin{remark}
Let $F$ be a $d$-form.
We have just remarked that the set of linear forms $\{\ell_1,\dotsc,\ell_r\}$
appearing in a Waring decomposition of $F$ is typically not uniquely determined.
What if we fix the $\ell_i$ and ask for uniqueness of the scalar coefficients $c_i$?

Suppose $F$ is a $d$-form and some linear forms $\ell_1,\dotsc,\ell_r$ are fixed.
Even if there exist scalars $c_1,\dotsc,c_r$ such that $F = \sum c_i \ell_i^d$,
i.e., even if $F$ is in the linear span of the $\ell_i^d$,
it may happen that the scalars $c_i$ are not necessarily uniquely determined,
as the $\ell_i^d$ may be linearly dependent.

However in the case of a Waring decomposition, the scalars $c_i$ are uniquely determined,
once $F$ and the linear forms $\ell_i$ are chosen.
Indeed, if $\{\ell_1,\dotsc,\ell_r\}$ are the linear forms appearing in a Waring decomposition of $F$
then the $\ell_i^d$ must be linearly independent,
or else the number of terms could be reduced by replacing
one of the $\ell_i^d$ by a linear combination of the others.
Since the $\ell_i^d$ are linearly independent, the scalars $c_i = c_i(F,\{\ell_1,\dotsc,\ell_r\})$
appearing in the Waring decomposition $F = \sum c_i \ell_i^d$ are uniquely determined.
\end{remark}

\subsection{Simultaneous Waring rank}

As a first step toward full generality,
let $W$ be a linear series of homogeneous forms of degree $d$.
A \defining{simultaneous power sum decomposition of $W$} of length $r$ is a collection of linear forms
$\ell_1,\dotsc,\ell_r$ such that for every $F \in W$ there exist scalars $c_1,\dotsc,c_r$
yielding a power sum decomposition $F = c_1 \ell_1^d + \dotsb + c_r \ell_r^d$.
That is, $W \subseteq \Span\{\ell_1^d,\dotsc,\ell_r^d\}$.
See \cite{Bronowski:1933fk}, \cite{MR1929406}, \cite{MR1998391}, \cite{MR2996361}.
The \defining{simultaneous Waring rank} $r(W)$ is the least length of
a simultaneous power sum decomposition of $W$;
a \defining{simultaneous Waring decomposition} is a simultaneous power sum decomposition of minimal length.
Clearly, $r(W) \geq r(F)$ for all $F \in W$, and $r(W) \geq \dim W$.
Also clearly, $r(W) \leq \sum r(F_i)$ for a basis $F_1,\dotsc,F_n$ of $W$.
Typically it is a difficult problem to determine the maximum and minimum of $r(F)$ for $F \in W$,
the set of $F$ on which the maximum is attained, etc.

\begin{example}\label{example: matrix linear series}
Recall that the \defining{permanent} of a $k \times k$ matrix $A = (a_{i,j})$
is
\[
  \per A = \sum_{\pi \in S_k} \prod_{i=1}^k a_{i,\pi(i)} ,
\]
that is, the (un-signed) sum of products
with one entry from each row and column of $A$.

Let $X = (x_{i,j})$, $1 \leq i \leq m$, $1 \leq j \leq n$ be a generic $m \times n$ matrix.
Let $D_k$ be the linear series spanned by the $k$-minors of $X$,
let $P_k$ be spanned by the permanents of $k \times k$ submatrices of $X$,
and let $R_k$ be spanned by the degree $k$ \defining{rook-free} products in $X$,
that is, products of $k$ distinct entries of $X$ with no two in the same row or column
(so that chess rooks placed in those positions would be pairwise non-attacking).

Note that $D_k, P_k \subset R_k$, so $\r(D_k), \r(P_k) \leq \r(R_k)$.
Each degree $k$ product of linearly independent factors has rank $2^{k-1}$,
so
\begin{equation}\label{eq: upper bound rank rookfree linear series}
  \r(R_k) \leq 2^{k-1} \dim(R_k) = 2^{k-1} \binom{m}{k} \binom{n}{k} k! .
\end{equation}
And
\[
  \r(D_k) \leq \left( \frac{5}{6} \right)^{\lfloor k/3 \rfloor} 2^{k-1} k! \dim(D_k)
    = \left( \frac{5}{6} \right)^{\lfloor k/3 \rfloor} 2^{k-1} \binom{m}{k} \binom{n}{k} k!  .
\]
As far as I know these are the best known upper bounds for the ranks of $R_k$ and $D_k$,
including in the case $m=n=k$ (where $D_k$ reduces to a single, square determinant).
While \eqref{eq: upper bound rank rookfree linear series} is also an upper bound for $r(P_k)$, we can do better.
First, the Ryser identity \cite{MR0150048} (for a square matrix) is the following:
\[
  \per_k = \per (x_{i,j})_{1\leq i,j\leq k}
  = \sum_{S \subseteq \{1,\dotsc,k\}} (-1)^{k - |S|} \prod_{i=1}^k \sum_{j \in S} x_{i,j} .
\]
For example,
\[
\begin{split}
  \per_3 &= (x_{1,1} + x_{1,2} + x_{1,3})(x_{2,1} + x_{2,2} + x_{2,3})(x_{3,1} + x_{3,2} + x_{3,3}) \\
    & \quad - (x_{1,1} + x_{1,2})(x_{2,1} + x_{2,2})(x_{3,1} + x_{3,2}) \\
    & \quad - (x_{1,2} + x_{1,3})(x_{2,2} + x_{2,3})(x_{3,2} + x_{3,3}) \\
    & \quad - (x_{1,1} + x_{1,3})(x_{2,1} + x_{2,3})(x_{3,1} + x_{3,3}) \\
    & \quad + x_{1,1}x_{2,1}x_{3,1} + x_{1,2}x_{2,2}x_{3,2} + x_{1,3}x_{2,3}x_{3,3} .
\end{split}
\]
This expresses the $k \times k$ permanent as a sum of $2^{k}-1$ terms, each of rank $2^{k-1}$.
Applying this to each basis element gives $\r(\per_k) \leq 2^{2k-1}-2^{k-1}$
and $\r(P_k) \leq (2^{2k-1}-2^{k-1}) \binom{m}{k} \binom{n}{k}$.
But better, Glynn gives a similar identity \cite{MR2673027}:
\begin{equation}\label{eq: Glynn}
  \per (x_{i,j})_{1\leq i,j\leq k}
  = \sum_{\substack{\epsilon \in \{\pm1\}^k \\ \epsilon_1 = 1}} \prod_{i=1}^k \sum_{j=1}^k \epsilon_i \epsilon_j x_{i,j} .
\end{equation}
For example,
\[
\begin{split}
  \per_3 &= (x_{1,1} + x_{1,2} + x_{1,3})(x_{2,1} + x_{2,2} + x_{2,3})(x_{3,1} + x_{3,2} + x_{3,3}) \\
    & \quad - (x_{1,1} + x_{1,2} - x_{1,3})(x_{2,1} + x_{2,2} - x_{2,3})(x_{3,1} + x_{3,2} - x_{3,3}) \\
    & \quad - (x_{1,1} - x_{1,2} + x_{1,3})(x_{2,1} - x_{2,2} + x_{2,3})(x_{3,1} - x_{3,2} + x_{3,3}) \\
    & \quad + (x_{1,1} - x_{1,2} - x_{1,3})(x_{2,1} - x_{2,2} - x_{2,3})(x_{3,1} - x_{3,2} - x_{3,3}) .
\end{split}
\]
This expresses the $k \times k$ permanent as a sum of $2^{k-1}$ terms, each of rank $2^{k-1}$.
Therefore $\r(\per_k) \leq 2^{2k-2}$ and $\r(P_k) \leq 2^{2k-2} \binom{m}{k} \binom{n}{k}$.

It would be interesting to determine if these natural linear series, especially $D_k$,
admit any simultaneous Waring decomposition shorter than simply decomposing separately
each member of a basis for the linear series,
or at least if any cleverly chosen basis can do better than the ``obvious'' defining basis
consisting of minors for $D_k$, permanents for $P_k$, and products for $R_k$.

We will revisit these linear series in Examples~\ref{example: matrix linear series catalecticants},
\ref{example: matrix linear series singularity improvement},
\ref{example: rookfree linear series singularity improvement}.
\end{example}


\begin{remark}\label{remark: coefficient uniqueness in simultaneous decomposition}
Note that if $\{\ell_1,\dotsc,\ell_r\}$ is a simultaneous Waring decomposition of $W$,
then the $\ell_i^d$ must be linearly independent, or else they would not be a minimal spanning set.
Thus if $\{\ell_1,\dotsc,\ell_r\}$ is a simultaneous Waring decomposition of $W$
then for each $F \in W$ the scalar coefficients $c_i = c_i(F,\{\ell_1,\dotsc,\ell_r\})$
in the power sum decomposition $F = \sum c_i \ell_i^d$
are uniquely determined, even though it is not necessarily a Waring decomposition of each $F \in W$.
\end{remark}

\subsection{Multihomogeneous polynomials}\label{section: multihomogeneous polynomials}

More generally, we consider ranks of multihomogeneous polynomials.
Fix $s > 0$, positive integers $n_1,\dotsc,n_s$,
and $s$ sets of doubly-indexed variables $x_{i,j}$, $1 \leq i \leq s$, $1 \leq j \leq n_i$.
A polynomial $F$ in the $x_{i,j}$ is multihomogeneous of multidegree $(d_1,\dotsc,d_s)$
if for each $i$, each monomial appearing in $F$ has degree $d_i$
in the $i$th set of variables, that is, $x_{i,1}, \dotsc, x_{i,n_i}$.

A \defining{multihomogeneous power sum decomposition of $F$} of length $r$
is an expression $F = \sum_{k=1}^r c_k \ell_{1,k}^{d_1} \dotsm \ell_{s,k}^{d_s}$
where each $\ell_{i,k}$ is a linear form in the $i$th set of variables.
As before, the \defining{multihomogeneous Waring rank} of $F$
is the least number of terms in a multihomogeneous power sum decomposition of $F$
and a \defining{multihomogeneous Waring decomposition} of $F$
is a multihomogeneous power sum decomposition of $F$ of minimal length.
See \cite{MR2246903} (focusing on uniqueness of decompositions).

Ranks of multihomogeneous polynomials generalize several familiar notions.
Classical Waring rank is the case $s=1$.
The case when the multidegree is $(1,\dotsc,1)$ is the usual tensor rank.
Tensor rank is very well studied, with applications far too numerous to mention;
see \cite{MR2535056}, \cite{MR2865915}.
If $W$ is a linear series of degree $d$ forms,
then the simultaneous Waring rank of $W$
is the rank of a single bihomogeneous polynomial of bidegree $(1,d)$.
Namely, let $F_1(x_1,\dotsc,x_n), \dotsc, F_s(x_1,\dotsc,x_n)$ be a basis for $W$;
then $\r(W)$ is equal to the rank of the multihomogeneous polynomial $M = \sum t_i F_i$ with multidegree $(1,d)$.

We make this last observation explicit.
Let $W, F_i, M$ be as above.
For each $j = 1,\dotsc,s$, let $e_j = (0,\dotsc,1,\dotsc,0)$,
the $s$-tuple with $1$ in the $j$th position and all other entries zero.
First, if $M = M(t,x) = \sum_{i=1}^r \ell_i(t) m_i(x)^d$ then for each $j = 1,\dotsc,s$,
$F_j = M(e_j,x) = \sum_{i=1}^r \ell_i(e_j) m_i(x)^d$,
so the $m_i$ give a simultaneous power sum decomposition of $W$ of length $r$.
Conversely, if $m_1,\dotsc,m_r$ give a simultaneous power sum decomposition of $W$,
write $F_j = \sum_{i=1}^r c_{i,j} m_i^d$ for each $j$.
The $c_{i,j}$ are uniquely determined by Remark~\ref{remark: coefficient uniqueness in simultaneous decomposition}.
For each $1 \leq i \leq r$, let $\ell_i$ be the linear form with $\ell_i(e_j) = c_{i,j}$ for $1 \leq j \leq s$.
Then $M = \sum_{i=1}^r \ell_i m_i^d$, giving a multihomogeneous power sum decomposition of $M$ of length $r$.

In order to distinguish between the classical Waring rank of $F$ as a homogeneous polynomial
and the rank of $F$ as a multihomogeneous polynomial,
we reserve $\r(F)$ for the former
and write $\r_{MH}(F)$ for the latter,
or $\r_{MH(d_1,\dotsc,d_s)}(F)$ if we wish to specify how $F$ is considered to be multihomogeneous.

\begin{example}\label{example: bihomogeneous product}
$F = x_1 \dotsm x_a y_1 \dotsm y_b$ is homogeneous of degree $d=a+b$
and has Waring rank $2^{a+b-1}$,
but it is also bihomogeneous of bidegree $(a,b)$ in the $x$ and $y$ variables,
and the bihomogeneous rank of $F$ is at most $2^{a+b-2}$.
Indeed, a decomposition of this length is given by multiplying decompositions of the separate parts.
Let $x_1 \dotsm x_a = \sum_{i=1}^{2^{a-1}} \ell_i^a$
and $y_1 \dotsm y_b = \sum_{j=1}^{2^{b-1}} m_j^b$ be Waring decompositions.
Then $F = \sum_{i,j} \ell_i^a m_j^b$ is a decomposition of $F$ as a bihomogeneous form
using $2^{a+b-2}$ terms.

This is actually the rank when $a=1$ or $b=1$---if, say, $a=1$,
then setting $x_1=1$ in any decomposition yields a (classical) Waring decomposition of $y_1\dotsm y_b$,
which must involve at least $2^{b-1}$ terms, so as a bihomogeneous form
$\r_{MH(1,b)}(F) \geq 2^{b-1} = 2^{a+b-2}$.

We will revisit these bihomogeneous products of variables
in Example~\ref{example: product of variables bihomogeneous catalecticant},
in Example~\ref{example: product of variables bihomogeneous singularity improvement},
and in Example~\ref{example: bihomogeneous product ranestad schreyer}
where we show that in fact $\r_{MH}(x_1 \dotsm x_a y_1 \dotsm y_b) = 2^{a+b-2}$.
\end{example}

\begin{example}
More generally, $r_{MH}(f(X)g(Y)) \leq r(f)r(g)$,
with equality if $r(f)=1$ or $r(g)=1$.
\end{example}

\begin{example}\label{example: determinant bihomogeneous rank}
The generic determinant $\det_n$ (see Example~\ref{example: determinant}) is bihomogeneous of bidegree $(a,n-a)$
in the sets of variables appearing in the first $a$ rows or last $n-a$ rows of the matrix.
The rank $\r_{MH(a,n-a)}(\det_n)$ is less than or equal to the Waring rank $\r(\det_n)$.
Also $\r_{MH(a,n-a)}(\det_n) \leq 2^{n-2} n!$, since $\det_n$ is a sum of $n!$ monomials
each with (multihomogeneous) rank at most $2^{n-2}$.
Better, Derksen's formula \eqref{eq: derksen formula} is multihomogeneous in the rows of the matrix,
so $\r_{MH(a,n-a)}(\det_n) \leq (\frac{5}{6})^{\lfloor n/3 \rfloor} 2^{n-2} n!$,
since $\det_n$ is a sum of $(5/6)^{\lfloor n/3 \rfloor} n!$ products of linear forms each with
(multihomogeneous) rank $2^{n-2}$.
We can do at least as well by expanding $\det_n$ as an alternating sum of products of maximal minors of the first $a$ rows
of the matrix with their complementary minors from the last $n-a$ rows,
yielding
\[
\begin{split}
  \r_{MH(a,n-a)}(\det\nolimits_n)
    & \leq \binom{n}{a} \r(\det\nolimits_a) \r(\det\nolimits_{n-a}) \\
    & \leq \binom{n}{a} \left( \frac{5}{6} \right)^{\lfloor a/3 \rfloor + \lfloor (n-a)/3 \rfloor} 2^{n-2} a! (n-a)! \\
    & = \left(\frac{5}{6}\right)^{\lfloor a/3 \rfloor + \lfloor (n-a)/3 \rfloor} 2^{n-2} n! \\
    & \leq \left(\frac{5}{6}\right)^{\lfloor n/3 \rfloor} 2^{n-2} n! .
\end{split}
\]
See also Examples~\ref{example: determinant bihomogeneous catalecticant},
\ref{example: determinant bihomogeneous singularity improvement}.
\end{example}

\subsection{Generalized rank}\label{section: generalized rank}

Even more generally, one can define rank with respect to any projective variety.
Let the variety $X \subset \PP^n$ be nondegenerate, that is, not contained in any hyperplane.
For an affine point $q \neq 0$, the \defining{rank of $q$ with respect to $X$}, denoted $r_X(q)$,
is the least $r$ such that there exist some $r$ distinct, reduced affine points $x_1, \dotsc, x_r$
such that $[x_1], \dotsc, [x_r] \in X$
and their linear span contains $q$: that is, $q = c_1 x_1 + \dotsb + c_r x_r$ for some scalars $c_i$.
Then the classical Waring rank is rank with respect to a Veronese variety;
tensor rank is rank with respect to a Segre variety;
the rank of a multihomogeneous polynomial is rank with respect to a Segre-Veronese variety,
a product of projective spaces $\PP^{n_1-1} \times \dotsb \times \PP^{n_s-1}$
embedded by the line bundle
\[
  \cO_{\PP^{n_1-1} \times \dotsb \times \PP^{n_s-1}}(d_1,\dotsc,d_s)
    = \pr_1^* \cO_{\PP^{n_1-1}}(d_1) \otimes \dotsb \otimes \pr_s^* \cO_{\PP^{n_s-1}}(d_s)
\]
where $(d_1,\dotsc,d_s)$ is the multidegree of the polynomial
and $\pr_i : \PP^{n_1-1} \times \dotsb \times \PP^{n_s-1} \to \PP^{n_i-1}$ is the projection onto the $i$th factor.

So this notion includes all the previous ideas, and more.
For example, the rank of an alternating tensor --- the least length of an expression
as a sum of simple wedges --- is its rank with respect to a Grassmannian in its Pl\"ucker embedding.
Ranks with respect to an elliptic normal curve have been studied in \cite[Thm.~28]{Bernardi201134}.

\begin{example}\label{example: carlini codimension one decompositions}
Carlini considered ``codimension one decompositions'' in \cite{MR2202247}.
Given a $d$-form $F \in S = \Bbbk[x_1,\dotsc,x_n]$,
such a decomposition is an expression $F = G_1 + \dotsc + G_r$,
where each $G_i$ is a $d$-form in a subring generated by $n-1$ linear forms:
$G_i \in \Bbbk[\ell_1,\dotsc,\ell_{n-1}]$.
In \cite{MR2202247} Carlini determines the number of summands in a codimension one decomposition of a general form.
The least number of terms in such a decomposition is given by rank with respect to
the variety of forms that depend on (at most) $n-1$ variables, called a \defining{subspace variety}.
We will define this more precisely in terms of catalecticants, introduced in the next section.
See Definition~\ref{definition: subspace variety}.
\end{example}

\begin{example}\label{example: carlini binary decompositions}
Similarly, Carlini considered ``binary decompositions'' in \cite{MR2184818},
expressions for $F$ as a sum of binary forms, i.e., forms lying in a subring generated by two linear forms.
Again the least number of terms in such an expression is given by rank with respect
to a subspace variety, this time parametrizing forms that depend on at most two variables,
see Definition~\ref{definition: subspace variety}.
\end{example}

Note that classical Waring rank corresponds to decompositions into forms depending on one variable,
i.e., homogeneous polynomials in a single linear form.
In fact the subspace variety whose points are forms depending on one variable is just the Veronese variety.

\begin{example}\label{example: split rank}
It is interesting to write a form $F$ as a sum of products of linear forms.
For example, the determinant and permanent of an $n \times n$ matrix can be written as
sums of $n!$ products of linear forms.
Derksen's formula improves this, see Example~\ref{example: derksen} below.
The least number of terms in such an expression is rank with respect to the variety parametrizing forms
which completely factor as products of linear forms, called the \defining{split variety}
or the Chow variety of zero-cycles,
see for example \cite{MR2729216}, \cite{MR2997451}, \cite{Torrance:2013qf}, \cite{MR3166068}.

Let us consider lower bounds for this rank, which we will call \defining{split rank}
and denote $\r_{\text{split}}$.
Since $\r(x_1\dotsm x_d) = 2^{d-1}$, we have for any $d$-form $F$ the relations
\[
  \r_{\text{split}}(F) \leq \r(F) \leq 2^{d-1} \r_{\text{split}}(F) .
\]
Therefore
\[
  \r_{\text{split}}(F) \geq 2^{1-d} \r(F) ,
\]
and any lower bound for Waring rank leads to a lower bound for split rank.
In Example~\ref{example: determinant} we saw $\r(\det_n) \geq \frac{1}{2} \binom{2n}{n}$,
so $\r_{\text{split}}(\det_n) \geq 2^{-n} \binom{2n}{n}$.
From this we get $\r_{\text{split}}(\det_3) \geq 3$.
The lower bound of \cite{Landsberg:2009yq} gives $\r(\det_3) \geq 14$,
so in fact $\r_{\text{split}}(\det_3) \geq 4$.

Similarly, the generic $n \times n$ permanent $\per_n$ has Waring rank
$\frac{1}{2} \binom{2n}{n} \leq \r(\per_n) \leq 2^{2n-2}$ (see \cite{Shafiei:ud})
so $2^{-n} \binom{2n}{n} \leq \r_{\text{split}}(\per_n)$.
And the Glynn identity \eqref{eq: Glynn} gives $\r_{\text{split}}(\per_n) \leq 2^{n-1}$.
Thus for instance $3 \leq \r_{\text{split}}(\per_3) \leq 4$.
Nathan Ilten has shown that $\r_{\text{split}}(\per_3) = 4$ (private communication).
\end{example}

\begin{example}\label{example: kth powers of t-forms}
If $F$ is a $d$-form and $d = kt$ then $F$ may be written as a sum of $k$th powers of $t$-forms.
The least number of terms in such an expression is the rank with respect to the variety of $d$-forms
which are $k$th powers.
Classical Waring rank is the case $t=1$.
Sum-of-squares decompositions correspond to $k=2$
(although in that setting one is usually interested in working
over a real field and finding decompositions with nonnegative coefficients,
whereas here we work over a closed field).
See \cite{Carlini:2013vl}, \cite{MR2935563}, \cite{Reznick:2013yq}.
\end{example}

\begin{example}\label{example: derksen}
We consider again the remarkable formula in Example~\ref{example: determinant}
discovered by Derksen \cite{Derksen:2013sf}.
Consider the generic determinant as a multilinear function of the columns,
i.e., $\det_n \in (\Bbbk^n)^{\otimes n}$.
For this paragraph we denote tensor rank by $\r_{\otimes}$.
Naively we have $\r_{\otimes}(\det_n) \leq n!$; for example,
\[
\begin{split}
  \dett_3 &= e_1 \otimes e_2 \otimes e_3 + e_2 \otimes e_3 \otimes e_1 + e_3 \otimes e_1 \otimes e_2 \\
    & \quad - e_1 \otimes e_3 \otimes e_2 - e_2 \otimes e_1 \otimes e_3 - e_3 \otimes e_2 \otimes e_1 ,
\end{split}
\]
where $\{e_1,e_2,e_3\}$ is a basis for $\Bbbk^3$.
Derksen's formula improves this:
\begin{equation}\label{eq: derksen formula tensor}
\begin{split}
  \dett_3 = \frac{1}{2} \Big( & (e_1+e_2) \otimes (e_2-e_3) \otimes (e_2+e_3) \\
    & + (e_1+e_2) \otimes (e_2-e_3) \otimes (e_2+e_3) \\
    & + 2 e_2 \otimes (e_3-e_1) \otimes (e_3+e_1) \\
    & + (e_3-e_2) \otimes (e_2+e_1) \otimes (e_2-e_1) \\
    & + (e_1-e_2) \otimes (e_3+e_2) \otimes (e_3-e_2) \Big) .
\end{split}
\end{equation}
This shows $\r_{\otimes}(\det_3) \leq 5$.
(One can show $\r_{\otimes}(\det_3) \geq 4$, see below.)
As noted by Derksen, Laplace expansion by complementary minors in the first $3$ rows gives
an improved upper bound for larger determinants.
\[
  \r_{\otimes}(\dett_n) \leq \binom{n}{3} \r_{\otimes}(\dett_3) \r_{\otimes}(\dett_{n-3})
    = \frac{5 \cdot n!}{6 \cdot (n-3)!} \r_{\otimes}(\dett_{n-3}),
\]
so by induction
\[
  \r_{\otimes}(\dett_n) \leq \left(\frac{5}{6}\right)^{\lfloor n/3 \rfloor} n! .
\]

Let us return to considering the determinant as a function of the entries, rather than the columns.
We get \eqref{eq: derksen formula}, as in Example~\ref{example: split rank},
from \eqref{eq: derksen formula tensor} by the substitution
that replaces $e_i \otimes e_j \otimes e_k$ with $x_{1,i} x_{2,j} x_{3,k}$:
\[
\begin{split}
 \dett_3 = \frac{1}{2} \Big( & (x_{13}+x_{12})(x_{21}-x_{22})(x_{31}+x_{32}) \\
   & + (x_{11}+x_{12})(x_{22}-x_{23})(x_{32}+x_{33}) \\
   & + 2 x_{12}(x_{23}-x_{21})(x_{33}+x_{31}) \\
   & + (x_{13}-x_{12})(x_{22}+x_{21})(x_{32}-x_{31}) \\
   & + (x_{11}-x_{12})(x_{23}+x_{22})(x_{33}-x_{32})
   \Big).
\end{split}
\]
This shows that $\r_{\text{split}}(\det_3) \leq 5$.
As in Example~\ref{example: split rank}, Laplace expansion shows that
\[
  \r_{\text{split}}(\dett_n) \leq \left(\frac{5}{6}\right)^{\lfloor n/3 \rfloor} n! .
\]
It follows that Waring rank satisfies
\[
  \r(\dett_n) \leq 2^{n-1} \left(\frac{5}{6}\right)^{\lfloor n/3 \rfloor} n! ,
\]
for example $\r(\det_3) \leq 20$.

Note that these formulas are multihomogeneous: each term involves one factor from each row of the matrix.
So Derksen's formula yields similarly improved upper bounds for $\r_{MH}(\det_n)$.

Analogously to the substitution that takes
\eqref{eq: derksen formula tensor} to \eqref{eq: derksen formula},
a similar substitution takes any tensor decomposition of $\det_n$
to a similar expression as a sum of products of linear forms.
Thus
\[
  \r_{\otimes}(\dett_3) \geq \r_{\text{split}}(\dett_3) \geq \frac{1}{4} \r(\dett_3) \geq \frac{14}{4} > 3,
\]
which shows $\r_{\otimes}(\dett_3) \geq 4$.
\end{example}

The problem of determining rank with respect to an arbitrary variety is
at least as hard as the already difficult problems of Waring rank and tensor rank.
Nevertheless, a great deal of progress has been made in understanding related questions
involving secant varieties, including determining dimensions and equations of secant varieties.
Recently, Landsberg and Ottaviani \cite{LO-apolar-catalecticant}
introduced a method to generate new equations of secant varieties.
They define a notion of catalecticants with respect to a given variety,
generalizing Sylvester's catalecticants beyond the case of Veronese varieties
as well as the symmetric flattenings or generalized Hankel matrices used for studying
secant varieties of Segre varieties (i.e., tensor rank).
Their motivation was to find equations for secant varieties, but here we are interested in
the lower bounds for rank given by their generalized catalecticants.

\section{Catalecticants}\label{section: catalecticants}

For each of the generalizations of Waring rank defined above,
there is a reasonably well-known notion of a \emph{catalecticant}.
In each case this is a linear map whose rank is a lower bound for (generalized) Waring rank.

For a discussion of the name ``catalecticant'', see
\cite[pg.~49--50]{MR1096187}, \cite[Lecture~11]{Geramita}, \cite{Miller:2013lo}.

\subsection{Classical Waring rank}\label{section: catalecticants - classical Waring rank}

If $V$ is a finite dimensional vector space,
the $d$th symmetric power $S^d V$ is the space of symmetric tensors or equivalently
homogeneous polynomials of degree $d$ in $V$.
That is, if $V$ has basis $x_1,\dotsc,x_n$, $S^d V$ is the $d$th graded piece of the polynomial ring
$\Bbbk[V] = \Bbbk[x_1,\dotsc,x_n]$.
We regard $S^a V^*$ as the space of differential operators of order $a$ with constant coefficients,
equivalently homogeneous polynomials of degree $a$ in the partial differentiation operators
$\d_1 = \frac{\d}{\d x_1} , \dotsc, \d_n = \frac{\d}{\d x_n}$.
That is, $S^a V^*$ is the $a$th graded piece of the polynomial ring
$\Bbbk[V^*] = \Bbbk[\d_1,\dotsc,\d_n]$.

An element of $S^{d-a} V^*$ can be regarded both as a polynomial function on $V$
and as a differential operator on $\Bbbk[V]$.
Say $p \in S^{d-a} V^*$ is regarded as a polynomial with corresponding differential operator $D_p$.
For a linear form $\ell = c_1 x_1 + \dotsb + c_n x_n \in V$ we write $p(\ell) = p(c_1,\dotsc,c_n)$.
We will make frequent use, especially in Section~\ref{section: improved lower bounds},
of the following well-known fact:
If $p \in S^{d-a} V^*$ and $\ell \in V$ then $D_p(\ell^d) = \frac{d!}{a!} \ell^a p(\ell)$.
See for example \cite[(1.1.10)]{MR1735271}.
In particular $D_p(\ell^d) = 0$ if and only if $p(\ell) = 0$.

For $0 \leq a \leq d$ the natural map
$S^d V \otimes S^{d-a} V^* \to S^a V$
coincides (at least in characteristic zero) with the usual differentiation action,
$F \otimes D \mapsto DF$ for a polynomial $F$ of degree $d$ and differential operator $D$ of order $d-a$.

\begin{defn}
Let $F \in S^d V$ and $0 \leq a \leq d$.
The \defining{$a$th catalecticant of $F$}, denoted $C^a_F$,
is the linear map $C^a_F : S^a V^{*} \to S^{d-a} V$,
$D \mapsto DF$.
\end{defn}
This is also called a symmetric flattening or generalized Hankel matrix.
(If $F$ is a binary polynomial, i.e.\ $\dim V = 2$, of degree $d = 2a$, then
modulo some binomial factors
$C^a_F$ is a Hankel matrix
when written in the usual monomial basis for $S^a V$ and $S^a V^*$.)

The inequality
\begin{equation}\label{eq: catalecticant bound single homogeneous form}
  \r(F) \geq \rank C^a_F
\end{equation}
is well known.
See \cite{Sylvester:1851kx}, in which Sylvester introduced the catalecticant matrix in 1851.
For the reader who is new to this area, we mention a surprisingly quick proof
(at least, it seemed surprisingly quick to me when I learned it!):
If $F = c_1 \ell_1^d + \dotsb + c_r \ell_r^d$,
then the image of $C^a_F$ is spanned by the $DF$ for $D \in S^a V^*$,
and each $DF$ is contained in the span of $D \ell_1^d, \dotsc, D \ell_r^d$.
Each $D \ell_i^d = c \ell_i^{d-a}$ for some constant $c$.
So the image of $C^a_F$ is contained in the span of $\ell_1^{d-a}, \dotsc, \ell_r^{d-a}$.
Thus the rank of $C^a_F$ is at most $r = \r(F)$.

This simple dimension-counting idea carries through in the generalizations as well.

\begin{example}\label{example: product of variables catalecticant}
Each derivative of $x_1 \dotsm x_n$ is a product of a subset of $\{x_1,\dotsc,x_n\}$
(or a linear combination of such products).
For each $a$, the products of degree $d-a$ are linearly independent and there are $\binom{n}{a}$ of them.
So $\r(x_1 \dotsm x_n) \geq \rank(C^a_{x_1\dotsm x_n}) = \binom{n}{a}$.
This lower bound is maximized when $a = \lfloor n/2 \rfloor$.

For a completely explicit example, $\r(xyz) \geq \rank(C^2_{xyz}) = 3$,
since the image of $C^2_{xyz}$ is spanned by
$\{x = \partial_y \partial_z xyz, y = \partial_x \partial_z xyz, z = \partial_x \partial_y xyz\}$.
Compare this with the upper bound $\r(xyz) \leq 4$
obtained from the explicit power sum decomposition in the introduction.
The non-sharpness of the catalecticant bound in such an undemanding example
spurs us to look for improved lower bounds.

See Example~\ref{example: product of variables singularity improvement}.
\end{example}

\begin{example}\label{example: determinant catalecticant}
Consider the generic determinant $\det_n$ as in Example~\ref{example: determinant}.
By the Laplace expansion, $\partial_{i,j} \det_n$ is the complementary $(n-1)$-minor of the matrix.
By induction, the image of the $a$th catalecticant of $\det_n$ is spanned by the $(n-a)$-minors of the matrix.
There are $\binom{n}{a}^2$ of these and they are linearly independent
(in fact, no two $(n-a)$-minors have any monomials in common,
since every monomial appearing in a minor determines the
set of rows and columns of the minor).
Thus $\r(\det_n) \geq \rank(C^a_{\det_n}) = \binom{n}{a}^2$.
Again, this lower bound is maximized when $a = \lfloor n/2 \rfloor$.

See Example~\ref{example: determinant singularity improvement}.
\end{example}

\begin{remark}
Note that $C^a_F$ and $C^{d-a}_F$ are transposes of one another.
Certainly they map between the right spaces to be transposes of one another,
as $C^a_F : S^a V^* \to S^{d-a} V$ and $C^{d-a}_F : S^{d-a} V^* \to S^a V$.
It is easy to see that these maps are actually transposes.
Here is a completely elementary argument.
Let $A \in S^a V^*$ and $B \in S^{d-a} V^*$.
The pairing $\langle C^a_F A, B \rangle$ is given by applying the differentiation operator $B$
to the polynomial $C^a_F A = AF$, yielding $BAF$.
Similarly, the pairing $\langle A, C^{d-a}_F B \rangle$ is given by $ABF$.
These are equal, $BAF = ABF$, and thus $C^a_F$ and $C^{d-a}_F$ are transposes of one another.
That is, $C^a_F$ and $C^{d-a}_F$ are transpose because partial differentiation operators commute on polynomials.

Another way to say the same thing is
that for any bilinear functional $W_1 \times W_2 \to \Bbbk$,
the induced maps $W_1 \to W_2^*$ and $W_2 \to W_1^*$ are transposes.
The catalecticant maps $C^a_F$ and $C^{d-a}_F$
arise in this way from the bilinear map $S^a V^* \times S^{d-a} V^* \to S^d V^* \to \Bbbk$
sending $(A,B) \mapsto AB \mapsto ABF$.
\end{remark}

\begin{remark}
In Sylvester's \cite{Sylvester:1851kx} the term ``catalecticant'' refers to the determinant
of the middle (square) catalecticant $C^a_F$ when $F$ is a form of even degree $d=2a$---additionally,
Sylvester limits his discussion in that paper to binary forms.
This usage seems to have persisted for some time; see for example \cite[pg.~232]{MR2850282}, \cite{MR859177}.

I do not know at what point ``catalecticant'' came to refer to the linear maps as above,
but this has been common usage for quite some time.
For example Definition 9.19 of \cite{Geramita}
introduces the $(i,j)$-catalecticant matrix $\operatorname{Cat}_F(i;j:n)$ of $F$ for any $i+j=\deg(F)$;
Definitions 1.2 and 1.3 of \cite{MR1735271}
introduce the catalecticant homomorphism $C_f(u,v)$
and catalecticant matrix $\operatorname{Cat}_f(u,v;r)$ for any $u+v=\deg(f)$.

Our $C_F^a$ would be $\operatorname{Cat}_F(d-a;a:n)$ in the notation of \cite{Geramita},
or $C_F(d-a,a)$ in the notation of \cite{MR1735271}.
\end{remark}

\begin{remark}\label{remark: nonunimodality of catalecticant ranks}
If $n \leq 3$ then the sequence $\rank(C^a_F)$, $a = 0,1,\dotsc,d$, is unimodal,
so the maximum occurs for $a = \lfloor d/2 \rfloor$ \cite{MR0485835}.
For $n \geq 5$ the sequence is not necessarily unimodal \cite{MR1172667}, \cite{MR1227512}, \cite{MR1311776}
and it is not at all clear where the maximum occurs.
It is still an open question whether this sequence is unimodal when $n=4$,
see for example \cite{MR2968919}.
It is known to be unimodal when $n=5$ and $d \leq 15$ \cite{MR3003310}
but surprisingly (to me, at least) this appears to be still open when $n=4$.
(The example of Bernstein--Iarrobino in \cite{MR1172667} has $n=5$ and $d=16$.)
\end{remark}

\begin{example}\label{example: stanley nonunimodal}
We briefly present Stanley's nonunimodal example \cite[Example 4.3]{MR0485835},
a form of degree $4$ in $13$ variables, which we denote $x,y,z,t_1,\dotsc,t_{10}$.
In these variables let $F = x^3 t_1 + x^2 y t_2 + \dotsb + xyz t_{10}$, so that the coefficients of the $t_i$
are precisely the monomials of degree $3$ in $x,y,z$.
Then $\rank C^0_F = \rank C^4_F = 1$ and $\rank C^1_F = \rank C^3_F = 13$ while $\rank C^2_F = 12$.
Thus $\r(F) \geq 13$.
See Example~\ref{example: stanley nonunimodal improvement}.
\end{example}

\begin{example}\label{example: bernstein-iarrobino nonunimodal}
We briefly present the nonunimodal example of Bernstein--Iarrobino \cite{MR1172667},
a form of degree $16$ in $5$ variables, which we denote $x,y,z,s,t$.
Let $F = Gs + Ht$ where $G,H$ are general forms of degree $15$ in $x,y,z$.
The ranks of catalecticants of $F$ are as follows:
\[
\begin{array}{l rrrrrrrrr rrrrrrrr}
  \toprule
  a           & 0 & 1 &  2 &  3 &  4 &  5 &  6 &  7 &  8 &  9 & 10 & 11 & 12 & 13 & 14 & 15 & 16 \\
  \midrule
  \rank C^a_F & 1 & 5 & 12 & 22 & 35 & 51 & 70 & 91 & 90 & 91 & 70 & 51 & 35 & 22 & 12 &  5 &  1 \\
  \bottomrule
\end{array}
\]
Thus $\r(F) \geq 91$.
See Example~\ref{example: bernstein-iarrobino nonunimodal improvement}.
\end{example}

In fact, the examples of Stanley and Bernstein--Iarrobino
are merely the first two members of a much larger family.
See \cite{MR1227512}.

\begin{defn}\label{definition: concise}
A polynomial $F \in S^d V$ is called \defining{concise} (with respect to $V$)
if it satisfies the following equivalent conditions:
\begin{enumerate}
\item $F$ cannot be written as a polynomial in a smaller number of variables;
that is, if $F \in S^d V'$ for some $V' \subseteq V$, then $V' = V$.
\item The projective hypersurface $V(F)$ is not a cone.
\item $C^{d-1}_F$ is surjective, $C^1_F$ is injective.
\end{enumerate}
\end{defn}

It would be interesting to have a similar geometric characterization
of the condition that $C^{d-k}_F$ is surjective for $k \geq 2$.

Following \cite{MR2279854}, let the \defining{span} of $F$, denoted $\langle F \rangle \subseteq V$, be the image of $C^{d-1}_F$.
Then $F \in S^d \langle F \rangle$.
Elements of $\langle F \rangle$ are called \defining{essential variables} of $F$, and $F$ is said to
\defining{depend essentially} on $k$ variables if $k = \dim \langle F \rangle$.
\begin{defn}\label{definition: subspace variety}
For each $1 \leq k < n$, the \defining{subspace variety} $\Sub_k \subset \PP S^d V$ is the locus of forms depending
essentially on $k$ or fewer variables:
\[
  \Sub_k = \{ [f] \mid \dim \langle f \rangle \leq k \} = \{ [f] \mid \rank C^{d-1}_F \leq k \} .
\]
Note that, upon choosing bases for $S^{d-1} V^*$ and $S^1 V \cong V$,
$\Sub_k$ is the zero locus of the $(k+1)$-minors of the matrix $C^{d-1}_F$,
whose entries are polynomials in the coefficients of $F$.
So $\Sub_k$ is a projective variety.
(In fact the entries are just the coefficients, up to some factorial factors.)
\end{defn}
We have $\nu_d(\PP V) = \Sub_1 \subset \Sub_2 \subset \dotsb$.
Each $\Sub_k$ contains, but may be strictly larger than, the secant variety $\sigma_k(\nu_d(\PP V))$.
$\Sub_{n-1}$ is precisely the locus of non-concise forms.
See Examples~\ref{example: carlini codimension one decompositions}, \ref{example: carlini binary decompositions}.

\begin{remark}
If $F$ lies in the span of $\ell_1^d, \dotsc, \ell_r^d$
then the image of $C^a_F$ is contained in the span of $\ell_1^{d-a}, \dotsc, \ell_r^{d-a}$.
But the converse does not hold.
Fix $0 < a < d$, let $r = \dim S^{d-a} V = \binom{n-1+d-a}{d-a}$, and let each $\ell_i$ be a general linear form.
Then the $\ell_i^{d-a}$ span $S^{d-a} V$, but the $\ell_i^d$ do not span $S^d V$;
so any $F \in S_d$ outside of the span of the $\ell_i^d$ furnishes a counterexample.

Concretely: the image of $C^1_{xy}$ is spanned by $\{x,y\}$, but $xy$ does not lie in the span of $\{x^2,y^2\}$.
\end{remark}

\subsection{Simultaneous Waring rank}

Let $W \subseteq S^d V$ be a linear series of degree $d$ forms.
Just as the simultaneous Waring rank of $W$ is a special case of multihomogeneous Waring rank,
the catalecticants of $W$ are special cases of catalecticants of multihomogeneous forms,
which we discuss next.
Nevertheless we pause to examine this special case before we go on.

We define two apparently different types of catalecticants (but they will be unified below,
see Example~\ref{example: linear series catalecticant equals multihomogeneous catalecticant}).
The $(0,a)$th catalecticant $C^{(0,a)}_W$ is the linear map
$S^a V^* \to W^* \otimes S^{d-a} V = \Hom(W,S^{d-a} V)$
sending $D \mapsto (F \mapsto DF)$.
The $(1,a)$th catalecticant $C^{(1,a)}_W$ is the linear map
$W \otimes S^a V^* \to S^{d-a} V$
sending $F \otimes D \mapsto DF$.
Thus the image of $C^{(1,a)}_W$ is the subspace spanned by the images of the $C^a_F$ for $F \in W$.

We have $\r(W) \geq \rank(C^{(1,a)}_W)$, with essentially the same proof as for classical Waring rank:
If $W$ is contained in the span of $\ell_1^d,\dotsc,\ell_r^d$,
then every $a$th derivative of every element of $W$ is contained in the span
of $\ell_1^{d-a},\dotsc,\ell_r^{d-a}$.

Similarly $\r(W) \geq \rank(C^{(0,a)}_W)$:
If $W$ is contained in the span of $\ell_1^d, \dotsc, \ell_r^d$
then $C^{(0,a)}_W(D)$ is determined by $D \ell_1^d, \dotsc, D \ell_r^d$
which are in turn determined by $D \ell_1^a, \dotsc, D \ell_r^a \in \Bbbk$,
so the image of $C^{(0,a)}_W$ has dimension at most $r$.

As before, it is easy to see that $C^{(1,a)}_W$ and $C^{(0,d-a)}_W$ are transposes of each other.

\begin{remark}
For a deep investigation of the ranks of catalecticants of linear series,
see \cite{MR2292384}.
\end{remark}

\begin{example}\label{example: matrix linear series catalecticants}
If $D_k$ is the linear series spanned by $k$-minors of a generic $m \times n$ matrix,
as in Example~\ref{example: matrix linear series},
then the image of $C^{(1,a)}_{D_k}$ is spanned by $(k-a)$-minors.
That is, the image of $C^{(1,a)}_{D_k}$ is just $D_{k-a}$.
Similarly, when $P_k$ is the linear series spanned by permanents of $k \times k$ submatrices
and $R_k$ is the linear series spanned by degree $k$ rook-free products in a generic $m \times n$ matrix,
then the image of the $(1,a)$th catalecticant of $P_k$ is $P_{k-a}$
and the image of the $(1,a)$th catalecticant of $R_k$ is $R_{k-a}$.
Thus
\begin{gather*}
  \rank C^{(1,a)}_{D_k} = \rank C^{(1,a)}_{P_k} = \binom{m}{k-a}\binom{n}{k-a},
  \\
  \rank C^{(1,a)}_{R_k} = \binom{m}{k-a}\binom{n}{k-a} (k-a)!.
\end{gather*}
In particular
\[
  \r(D_k) \geq \max_{0 \leq a \leq k} \binom{m}{a}\binom{n}{a},
  \qquad
  \r(P_k) \geq \max_{0 \leq a \leq k} \binom{m}{a}\binom{n}{a},
  \qquad
  \r(R_k) \geq \max_{0 \leq a \leq k} \binom{m}{a}\binom{n}{a} a!.
\]
Note that $\binom{m}{a}\binom{n}{a}$ is maximized when $a = \left\lceil \frac{mn-1}{m+n+2} \right\rceil$.
When $m=n$ this is $\left\lceil \frac{n-1}{2} \right\rceil$.
Indeed,
\[
  \frac{ \binom{m}{a+1}\binom{n}{a+1} }{ \binom{m}{a}\binom{n}{a} } = \frac{(m-a)(n-a)}{(a+1)^2}
\]
and this $>1$ if and only if $mn-(m+n)a > 2a+1$, that is, $(m+n+2)a < mn-1$.
In particular the sequence $\{ \binom{m}{a}\binom{n}{a} \mid a = 0,\dotsc,\min(m,n) \}$ is unimodal.

Similarly, $\binom{m}{a}\binom{n}{a} a!$ is maximized when
\[
  a = \left\lceil \frac{m+n+1 - \sqrt{(m+n+1)^2 - 4(mn-1)}}{2} \right\rceil .
\]
When $m=n$ this is $\left\lceil \frac{2n+1-\sqrt{4n+2}}{2} \right\rceil$.

See Examples~\ref{example: matrix linear series singularity improvement},
\ref{example: rookfree linear series singularity improvement}.
\end{example}

\subsection{Multihomogeneous polynomials}

At this point it is convenient to introduce multinomial notation.

\begin{notation}
Fix $s \geq 1$ and vector spaces $V_1 \cong \Bbbk^{n_1}, \dotsc, V_s \cong \Bbbk^{n_s}$.
We denote $\bn = (n_1,\dotsc,n_s)$.
For $\bd = (d_1, \dotsc, d_s)$, where $0 \leq d_i$ for $1 \leq i \leq s$,
we define
\[
  S^{\bd} V = S^{d_1} V_1 \otimes \dotsb \otimes S^{d_s} V_s
\]
and the dual space $S^{\bd} V^*$ similarly.
Also, we let $\bzero = (0,\dotsc,0)$, with $s$ entries.

The sum and difference of $s$-tuples is elementwise.
We partially order $s$-tuples by elementwise comparison;
thus 
$\bzero \leq \ba \leq \bd$ means that for each $i$, $0 \leq a_i \leq d_i$.
\end{notation}

\begin{defn}
Fix $\bn$ and $\bd$ as above.
Let $M \in S^{\bd} V$ be a multihomogeneous polynomial of multidegree $\bd$.
%
For each $\bzero \leq \ba \leq \bd$ the $\ba$'th catalecticant of $M$ is the linear map
\[
  C^{\ba}_M : S^{\ba} V^*  \to  S^{\bd-\ba} V ,
\]
that is
\[
  C^{\ba}_M : \prod_{i=1}^s S^{a_i}(V_i^*) \to \prod_{i=1}^s S^{d_i - a_i}(V_i) ,
\]
given by the usual contraction or flattening.

In detail, and for the sake of concreteness,
suppose for each $i$ we fix a basis $\{x_{i,1}, \dotsc, x_{i,n_i}\}$ for $V_i$.
We denote the dual basis for $V_i^*$ by $\{\partial_{i,1}, \dotsc, \partial_{i,n_i}\}$.
A polynomial in the $\partial_{i,m}$ acts as a differential operator on $S^{\bd} V$
(where each $\partial_{i,m}$ acts as $\partial / \partial x_{i,m}$).
The catalecticant $C^{\ba}_M$ takes a differential operator $D$ to its evaluation $DM$ on $M$,
regarded as a multihomogeneous polynomial.
When $D$ is homogeneous of multidegree $\ba$, $DM$ has multidegree $\bd-\ba$.
\end{defn}

The bound
\begin{equation}\label{eq: multihomogeneous polynomial catalecticant bound}
  \r_{MH}(M) \geq \rank C^{\ba}_M
\end{equation}
is well-known.
The proof is essentially the same as for classical Waring rank and simultaneous Waring rank.
Suppose $M$ is written as a sum of $r$ terms each of the form $\ell_1^{d_1} \dotsm \ell_s^{d_s}$,
where each $\ell_i \in V_i$.
If $D$ is multihomogeneous of multidegree $\ba$ then
$D(\ell_1^{d_1} \dotsm \ell_s^{d_s}) = c \ell_1^{d_1-a_1} \dotsm \ell_s^{d_s-a_s}$
for some constant $c$.
So $DM$ is a linear combination of the $r$ terms of the form $\ell_1^{d_1-a_1} \dotsm \ell_s^{d_s-a_s}$,
which means the image of $C^{\ba}_M$ is contained in the linear span of these $r$ terms.
Thus $r = \r_{MH}(M) \geq \rank C^{\ba}_M$.

\begin{example}\label{example: linear series catalecticant equals multihomogeneous catalecticant}
Given a linear series $W \subseteq S^d V$,
we regard the associated
bihomogeneous form $M$, defined in \textsection\ref{section: multihomogeneous polynomials},
as an element of $S^1 W^* \otimes S^d V$.
Then the catalacticants of $W$ defined earlier agree with the catalecticants of $M$ defined here:
$C^{(0,a)}_W = C^{(0,a)}_M$ and $C^{(1,a)}_W = C^{(1,a)}_M$.
\end{example}

\begin{example}\label{example: product of variables bihomogeneous catalecticant}
Let $F = x_1 \dotsm x_a y_1 \dotsm y_b$, a bihomogeneous form of bidegree $\bd = (a,b)$ in the $x$ and $y$ variables,
as in Example~\ref{example: bihomogeneous product}.
Let $\ba = (p,q)$ with $0 \leq p \leq a$, $0 \leq q \leq b$.
Then the image of $C^{\ba}_F$ is spanned by monomials which are products of $a-p$ of the $x$ variables and $b-q$ of the $y$ variables.
So $r_{MH(a,b)}(F) \geq \rank C^{\ba}_F = \binom{a}{p} \binom{b}{q}$.

In particular, when $a=b=2$, $r_{MH(2,2)}(x_1 x_2 y_1 y_2) \geq \binom{2}{1}^2 = 4$.
We saw before $r_{MH(a,b)}(F) \leq 2^{a+b-2}$.
Thus $r_{MH(2,2)}(x_1 x_2 y_1 y_2) = 4$.

See Example~\ref{example: product of variables bihomogeneous singularity improvement}.
\end{example}

\begin{example}\label{example: determinant bihomogeneous catalecticant}
The generic determinant $\det_n$ is bihomogeneous of bidegree $\bd = (a,n-a)$
in the sets of variables in the first $a$ and last $n-a$ rows,
as in Example~\ref{example: determinant bihomogeneous rank}.
For $\ba = (p,q)$ with $0 \leq p \leq a$, $0 \leq q \leq n-a$,
the image of the catalecticant $C^{\ba}_{\det_n}$ is spanned by $(n-p-q)$-minors
with $a-p$ rows in the first $a$ rows of the matrix, $n-a-q$ rows in the last $n-a$ rows of the matrix,
and any $n-p-q$ columns.
Thus $r_{MH(a,n-a)}(\det_n) \geq \rank C^{\ba}_{\det_n} = \binom{a}{p} \binom{n-a}{q} \binom{n}{p+q}$.

See Example~\ref{example: determinant bihomogeneous singularity improvement}.
\end{example}

\subsection{Generalized catalecticants}

Recently Landsberg and Ottaviani \cite{LO-apolar-catalecticant} gave a generalization of Sylvester's catalecticants
to the setting of generalized rank.
We recall their definition here.

Let $X$ be a projective variety and $L$ a very ample line bundle on $X$.
Let $V = H^0(X,L)^*$, so that $L$ naturally embeds $X \subset \PP V$.
Fix $v \in V$, $v \neq 0$; we aim to determine, or bound, $r_X(v)$.
Let $E$ be a vector bundle on $X$ whose rank (fiber dimension) we denote by $\rank E$.
Landsberg and Ottaviani's generalized catalecticant of $v$ with respect to $E$ is a map
denoted $C^E_v$ (they use the letter $A$ but we prefer $C$ for \emph{catalecticant})
and defined as follows.
From the natural map $E \otimes E^* \to \cO_X$ we get $E \otimes (L \otimes E^*) \to L$
and then in turn a multiplication map on global sections, $H^0(E) \otimes H^0(L \otimes E^*) \to H^0(L) = V^*$.
We regard $v \in V$ as a map $V^* \to \Bbbk$,
so composing gives a bilinear map $H^0(E) \otimes H^0(L \otimes E^*) \to \Bbbk$.
The \defining{generalized catalecticant of $v$ with respect to $E$}
is the resulting map $C^E_v : H^0(E) \to H^0(L \otimes E^*)^*$.
That is, $C^E_v : s \mapsto (t \mapsto v(st))$.

The catalecticants $C^E_v$ and $C^{L \otimes E^*}_v$ are transposes of one another.

\begin{example}\label{example: realize classical catalecticants as generalized catalecticants}
The classical catalecticant $C^a_F$ for a homogeneous form $F$ corresponds to
$X = \PP V$, $L = \mathcal{O}_{\PP V}(d)$, and $E = \mathcal{O}_{\PP V}(a)$.
The embedding $X \subset \PP H^0(X,L)^* = \PP S^d V$ is the Veronese map, taking $X$ to $\nu_d(\PP V)$.

The multihomogeneous catalecticant $C^{\ba}_M$
for a multihomogeneous form $M$ of multidegree $\bd = (d_1,\dotsc,d_s)$
corresponds to the following.
Let $X = \PP V_1 \times \dotsb \times \PP V_s$
and for each $i$ let $\pr_i : X \to \PP V_i$ be the projection onto the $i$th factor.
Then $C^{\ba}_M$ corresponds to the line bundles
$L = \mathcal{O}_{\PP V_1 \times \dotsb \times \PP V_s}(d_1,\dotsc,d_s) = \pr_1^* \O_{\PP V_1}(d_1) \otimes \dotsb \otimes \pr_s^* \O_{\PP V_s}(d_s)$
and
$E = \mathcal{O}_{\PP V_1 \times \dotsb \times \PP V_s}(a_1,\dotsc,a_s)$.
The embedding of $X$ is the Segre--Veronese map, taking $X$ to
$\Seg(\nu_{d_1}(\PP V_1) \times \dotsm \times \nu_{d_s}(\PP V_s))$.
\end{example}

The following proposition is essentially just Propositions 5.1.1 and 5.4.1 in \cite{LO-apolar-catalecticant}.
(Their Proposition 5.4.1 is more general, working with multipoint restrictions rather than just one point.)
\begin{prop}
If $[v] = x \in X$, then $\rank C^E_v \leq \rank E$,
with equality if and only if both $E$ and $L \otimes E^*$ are globally spanned at $x$;
that is, the maps
$H^0(E) \to E_x$ and $H^0(L \otimes E) \to (L \otimes E^*)_x = L_x \otimes E^*_x$
are surjective.
For all $0 \neq v \in V$,
\begin{equation}\label{eq: LO bound}
  \r_X(v) \geq \frac{\rank C^E_v}{\rank E} .
\end{equation}
\end{prop}
\begin{proof}
If $[v] = x \in X$,
$C^E_v$ corresponds to the bilinear functional
$H^0(E) \otimes H^0(L \otimes E^*) \to H^0(L) \to \Bbbk$
where the first map is multiplication of global sections and the second map
is evaluation at $x \in X$.
This commutes with restriction to the fiber at $x$, namely
$E_x \otimes (L \otimes E^*)_x \to L_x \cong \Bbbk$.
Therefore the map $C^E_v$ is equal to the composition of the restriction $H^0(E) \to E_x$,
followed by the identification
$E_x \cong L^*_x \otimes E_x \cong (L \otimes E^*)^*_x$
(up to a choice of scalar, i.e., a basis for $L_x \cong \Bbbk^1$),
followed by the transpose of the restriction map, $(L \otimes E^*)^*_x \to H^0(L \otimes E^*)^*$.
That is, $C^E_v$ is exactly the natural map
\[
  H^0(E) \to E_x \cong L^*_x \otimes E_x \cong (L \otimes E^*)^*_x \to H^0(L \otimes E^*)^* .
\]
Since $\dim (L \otimes E^*)_x = \dim E^*_x = \rank E$,
the composition of these maps has rank at most $\rank E$,
with equality if and only if the first map is surjective and the last is injective.

If $v = v_1 + \dotsb + v_r$ then $C^E_v = \sum C^E_{v_i}$.
This gives \eqref{eq: LO bound}.
\end{proof}

In particular, when $\r_X(v) = r$ then the size $(re+1)$ minors of $C^E_v$ vanish, where $e = \rank E$.
Thus the size $(re+1)$ minors of $C^E_v$ vanish on a dense subset of
the $r$th secant variety $\sigma_r(X)$,
hence on all of $\sigma_r(X)$.
Therefore these minors give equations for $\sigma_r(X)$.
One can thus obtain equations for secant varieties by producing suitable bundles $E$,
which is pursued in \cite{LO-apolar-catalecticant} by using representation theory.



More generally, let $\cE$ be an $\cO_X$-module whose fibers have bounded dimension:
that is, there is a positive integer $e$ such that for each $x \in X$,
the fiber $\cE_x = \cE \otimes \Bbbk(x)$ (where $\Bbbk(x) = \cO_{X,x}/\mathfrak{m}_{x}$)
has dimension at most $e$, as a vector space over $\Bbbk(x) \cong \Bbbk$
(since we have assumed $\Bbbk$ is algebraically closed).
Let $\cE^\vee = \cHom_{\cO_X}(\cE,\cO_X)$.
Then each fiber of $\cE^\vee$ also has vector space dimension at most $e$.
As before, the natural multiplication map $\cE \otimes \cE^\vee \to \cO_X$
yields $\cE \otimes (L \otimes \cE^\vee) \to L$,
which on global sections yields in turn $H^0(X,\cE) \otimes H^0(X,L \otimes \cE^\vee) \to H^0(L) = V^*$.
Once again our fixed $v \in V$ gives a map $V^* \to \Bbbk$,
so we have a bilinear functional $H^0(X,\cE) \otimes H^0(X,L \otimes \cE^\vee) \to \Bbbk$.
We define the \defining{generalized catalecticant of $v$ with respect to $\cE$}
to be $C^{\cE}_v : H^0(X,\cE) \to H^0(X,L \otimes \cE^\vee)^*$.

If $\cE$ is reflexive then $C^{\cE}_v$ and $C^{L \otimes \cE^\vee}_v$ are transposes of one another.

Once again if $[v] = x \in X$ then $C^{\cE}_v$ factors through restriction of global sections to the fiber at $x$:
\[
  H^0(X,\cE) \to \cE_x \to ((L \otimes \cE^\vee)_x)^* \to H^0(X,L \otimes \cE^\vee)^* .
\]
Hence if $[v] = x \in X$ then $\rank C^{\cE}_v \leq \dim \cE_x = \dim (L \otimes \cE^\vee)_x \leq e$,
with equality holding in the first inequality if and only if both $\cE$ and $L \otimes \cE^\vee$ are globally generated at $x$.
Therefore
\[
  \r_X(v) \geq \frac{\rank C^{\cE}_v}{e} .
\]

\begin{example}\label{example: subspace variety sheaf}
Let $X = \Sub_k \subset \PP S^d V$, the subspace variety of forms depending essentially on $k$ or fewer variables.
For $[F] \in X$, the catalecticant $C^{d-1}_F$ can be written as a matrix whose entries depend linearly
on the coefficients of $F$.
So we have a map of $\cO_X$-modules
\[
  C^{d-1}: S^{d-1} V^* \otimes \cO_X(-1) \to V \otimes \cO_X
\]
given by $C^{d-1}_F : S^{d-1} V^* \to V$ at each $[F] \in X$.
The $\Sub_k$ are degeneracy loci of the map $C^{d-1}$.
Then the kernel and cokernel of $C^{d-1}$ are naturally associated sheaves,
which are vector bundles on each $\Sub_k \setminus \Sub_{k-1}$.
It would be interesting to work out the global sections and catalecticants corresponding to these sheaves.
Perhaps they might be related to the ``Apolarity Lemma'' discussed below, see Theorem~\ref{thm: apolarity lemma carlini}.
\end{example}

\section{Improved lower bounds}\label{section: improved lower bounds}

An improvement to the catalecticant lower bound for classical Waring rank
was found in \cite{Landsberg:2009yq}.
It raises the lower bound $\rank C^a_F$ by the dimension of a certain set of singularities
of the hypersurface $V(F)$.
As our goal is to give a similar improvement more generally
we first recall that result.

\subsection{Classical Waring rank}

\begin{defn}
Let $F \in S^d V$ and $0 \leq a < d$.
We define $\Sigma_a(F) \subseteq \PP V^{*}$
to be the common zero locus of all the $a$th mixed partial derivatives of $F$;
that is, the projective variety defined by the (ideal generated by the) image of the catalecticant $C^a_F$.
As a set, $\Sigma_a(F)$ is the set of points at which $F$ vanishes to order at least $a+1$.

We write $\widehat{\Sigma}_a(F)$ for the affine cone over $\Sigma_a(F)$.
Explicitly, $\widehat{\Sigma}_a(F)$ is the common zero locus in $V^*$ of the $a$th mixed partial derivatives of $F$;
in particular, for $a<d$, the origin $0 \in \widehat{\Sigma}_a(F)$ even if $\Sigma_a(F)$ is empty.
\end{defn}
Thus $\Sigma_0(F)$ is the projective hypersurface $V(F)$
and $\Sigma_1(F)$ is the singular locus of $V(F)$,
similarly $\widehat\Sigma_0(F)$ is the affine hypersurface defined by $F$
and $\widehat\Sigma_1(F)$ is its singular locus.

Note that $F$ is concise (Definition~\ref{definition: concise}) if and only if $\widehat{\Sigma}_{d-1}(F) = \{0\}$,
equivalently $\Sigma_{d-1}(F) = \varnothing$.

\begin{thm}[Theorem 1.3 of \cite{Landsberg:2009yq}] \label{thm: single homogeneous polynomial}
Let $F \in S^d V$ and let $0 \leq a < d$.
If $F$ is concise
then
\begin{equation}\label{eq: LT bound}
  \r(F) \geq \rank C^{d-a}_F + \dim \widehat{\Sigma}_a(F) .
\end{equation}
\end{thm}

In \cite{Landsberg:2009yq} this is stated as
\[
  \r(F) \geq \rank C^{d-a}_F + \dim \Sigma_a(F) + 1,
\]
with the understanding that $\dim \varnothing = -1$.
Although the affine cone version \eqref{eq: LT bound} is simpler
(and avoids conventions about negative dimensions),
this ``projective'' one is the version that will generalize.


We review the proof given in \cite{Landsberg:2009yq}.
First, recall that the differential operator $D_p$ associated to $p \in S^{d-a} V^*$
satisfies $D_p(\ell^d) = \left(\frac{d!}{a!}\right) \ell^a \, p(\ell)$,
so $D_p(\ell^d) = 0$ if and only if $p([\ell]) = 0$.
We begin with the following well-known statement (see, for example, \cite[Proposition 4.1]{MR1215329}).

\begin{prop}\label{prop: catalecticant kernel veronese intersection}
Let $h \in V^*$.
Then $h^{d-a} \in \ker C^{d-a}_F$ if and only if $h \in \widehat{\Sigma}_a(F)$.
\end{prop}
\begin{proof}
$h^{d-a} \in \ker C^{d-a}_F$ if and only if $h^{d-a} F = 0$,
if and only if $\Theta h^{d-a} F = 0$ for all $\Theta \in S^a V^*$,
if and only if $h^{d-a} \Theta F = 0$ for all $\Theta$.
Since $\deg \Theta F = d-a$, $h^{d-a} \Theta F$ is just the evaluation of $\Theta F$ at the point $h$,
$h^{d-a} \Theta F = \Theta F \big|_h$ (up to a factorial factor).
So $h^{d-a} \in \ker C^{d-a}_F$ if and only if $\Theta F \big|_h = 0$ for all $\Theta \in S^a V^*$,
if and only if $F$ vanishes to order at least $a+1$ at $h$,
if and only if $h \in \widehat{\Sigma}_a(F)$.
\end{proof}

\begin{remark}
Proposition~\ref{prop: catalecticant kernel veronese intersection} verges on tautology:
$h$ is in $\widehat{\Sigma}_a(F)$ if and only if the $a$th derivatives of $F$ vanish at $h$,
if and only if $h$ is a common zero of the forms in the image of $C^a_F$,
if and only if $h^{d-a} \in (\img C^a_F)^\perp = \ker C^{d-a}_F$.
\end{remark}

\begin{proof}[Proof of Theorem \ref{thm: single homogeneous polynomial}]
Suppose $F = \ell_1^d + \dotsb + \ell_r^d$.
Let $\mathcal{L} = \{ p \in S^{d-a} V^* : p([\ell_1]) = \dotsb = p([\ell_r]) = 0 \}$,
the linear series of degree $d-a$ forms vanishing at the points $[\ell_i] \in \PP V$.
Since $D_p(\ell_i^d) = 0$ if and only if $p([\ell_i]) = 0$,
we have $\mathcal{L} \subseteq \ker C^{d-a}_F$.
This shows that
\[
  r \geq \codim \mathcal{L} \geq \codim \ker C^{d-a}_F = \rank C^{d-a}_F ,
\]
the first inequality holding since $\mathcal{L}$ is defined by $r$ linear conditions.

Now we use the hypothesis that $F$ is concise.
Because of this, the $[\ell_i]$ can not lie on a hyperplane,
or else $F$ could be written in fewer variables.
Hence $\mathcal{L}$ can not contain any power $p = h^{d-a}$ of a linear form $h$;
for if $h^{d-a} \in \mathcal{L}$, then $h^{d-a}([\ell_i]) = 0$ for each $i$,
so $h([\ell_i]) = 0$ for each $i$, and the $[\ell_i]$ would lie on the hyperplane defined by $h$,
a contradiction.
Therefore the projectivization $\PP \mathcal{L}$ is disjoint from the
Veronese $v_{d-a}(\PP V^*)$.
Since $\mathcal{L} \subseteq \ker C^{d-a}_F$,
we disregard everything lying outside this kernel
and observe that $\PP \mathcal{L}$ is disjoint from $\PP \ker C^{d-a}_F \cap v_{d-a}(\PP V^*)$.
Thus
\[
  \dim(\PP \mathcal{L}) + \dim(\PP \ker C^{d-a}_F \cap v_{d-a}(\PP V^*)) < \dim \PP \ker C^{d-a}_F ,
\]
or else a nonempty intersection would be forced.
We subtract each side of this inequality from the ambient dimension $\dim \PP S^{d-a} V^*$
and rearrange to get
\[
  r \geq \codim \mathcal{L} = \codim(\PP \mathcal{L})
    > \rank C^{d-a}_F + \dim(\PP \ker C^{d-a}_F \cap v_{d-a}(\PP V^*)) .
\]
Finally $\PP \ker C^{d-a}_F \cap v_{d-a}(\PP V^*) \cong \Sigma_a(F)$
by Proposition~\ref{prop: catalecticant kernel veronese intersection}.
\end{proof}

Though we have not emphasized it here, reducedness plays a key role in the above proof,
see Remark~\ref{remark: reducedness used in proof}.
This will be important when we consider ``scheme'' statements
in Sections \ref{section: apolarity lemmas} and \ref{section: ranestad-schreyer bounds}.

\begin{remark}
Since $C^a_F$ and $C^{d-a}_F$ are transposes of one another,
we could of course say $\r(F) \geq \rank C^a_F + \dim \widehat\Sigma_a(F)$.
By Remark~\ref{remark: nonunimodality of catalecticant ranks},
it is impossible (or at least difficult) to say for which $a$ the maximum of $\rank C^a_F$ occurs.
On the other hand, since $\Sigma_0(F) \supseteq \Sigma_1(F) \supseteq \dotsb$,
the dimensions of $\widehat\Sigma_a(F)$ are clearly nonincreasing.
A priori the maximum value of $\rank C^a_F + \dim \widehat\Sigma_a(F)$ could occur at a different $a$
than the maximum value of $\rank C^a_F$,
but I do not know any example in which this happens.
\end{remark}


\begin{example}\label{example: product of variables singularity improvement}
We saw $r(xyz) \leq 4$ and $\rank C^1_{xyz} = 3$ in Example~\ref{example: product of variables catalecticant}.
We have $\Sigma_1(xyz) = \Sing V(xyz) = \{[1:0:0], [0:1:0], [0:0:1]\}$
so $\dim \Sigma_1(xyz) = 0$,
$\dim \widehat{\Sigma}_1(xyz) = 1$.
Therefore $r(xyz) \geq 3 + 1 = 4$.
\end{example}

\begin{example}\label{example: determinant singularity improvement}
We saw $\rank C^a_{\det_n} = \binom{n}{a}^2$ in Example~\ref{example: determinant catalecticant}.
And $\det_n$ vanishes to order $a+1$ at a point (matrix) $M$ if and only if every $(n-a)$-minor of $M$ vanishes,
that is, $M$ has rank $n-a-1$ or less.
The dimension of the locus of matrices of rank $n-a-1$ or less is $n^2 - 1 - (a+1)^2$.
So, for each $a$, $\r(\det_n) \geq \binom{n}{a}^2 + n^2 - (a+1)^2$.
This is maximized at $a = \lfloor n/2 \rfloor$,
so $\r(\det_n) \geq \binom{n}{\lfloor n/2 \rfloor}^2 + n^2 - (\lfloor n/2 \rfloor + 1)^2$.

For instance, $\r(\det_3) \geq 14$.
See \cite[\textsection9]{Landsberg:2009yq}.
For other bounds on $\r(\det_n)$ and ranks of permanents, Pfaffians, etc.,
see \cite{Shafiei:ud}, \cite{Shafiei:2013fk}.
\end{example}

\begin{example}\label{example: stanley nonunimodal improvement}
We consider Stanley's nonunimodal example, see Example~\ref{example: stanley nonunimodal}.
Since $F$ vanishes to order $3$ on the linear subspace $x=y=z=0$,
$\widehat\Sigma_2(F) \supseteq V(x,y,z)$, a $10$-dimensional subspace.
On the other hand, $\widehat\Sigma_2(F) \subseteq \widehat\Sigma_1(F) \subseteq V(x,y,z)$,
since $\d_{t_1} F = x^3$ so $\d_{t_1} F = 0$ implies $x=0$, and so on.
So $\widehat\Sigma_1(F) = \widehat\Sigma_2(F) = V(x,y,z)$.
Thus $\r(F) \geq \rank C^2_F + \dim \widehat\Sigma_2(F) = 22$,
and better, $\r(F) \geq \rank C^1_F + \dim \widehat\Sigma_1(F) = 23$.
%
\end{example}

\begin{example}\label{example: bernstein-iarrobino nonunimodal improvement}
We consider the nonunimodal example of Bernstein--Iarrobino, see Example~\ref{example: bernstein-iarrobino nonunimodal}.
Since $F$ vanishes to order $15$ on the subspace $x=y=z=0$,
$\dim \widehat\Sigma_a(F) \geq 2$ for $0 \leq a \leq 14$.

Now we examine $\widehat\Sigma_1(F)$.
It contains the $2$-plane $x=y=z=0$.
Suppose $(x,y,z) \neq (0,0,0)$ and $p = (x,y,z,s,t) \in \widehat\Sigma_1(F)$.
Since $\d_s F = \d_t F = 0$, the projective point $q = [x:y:z]$ must lie in the complete intersection $G=H=0$,
a finite set of $225$ points.
Furthermore the plane curve $sG + tH$ must be singular at $q$.
We have
\[
  (\d_x F, \d_y F, \d_z F) = s \nabla G + t \nabla H .
\]
Since $G, H$ are general, they intersect transversally at $q$, with independent gradients.
This forces $s=t=0$.
So $\widehat\Sigma_1(F)$ is the union of the $2$-plane $x=y=z=0$
and $225$ lines with $s=t=0$, spanned by the points in the complete intersection $G=H=0$.
Therefore $\dim \widehat\Sigma_a(F) = 2$ for $0 \leq a \leq 14$.

Hence $\r(F) \geq \rank C^7_F + \dim \widehat\Sigma_7(F) = 93$.
%
\end{example}

\begin{remark}
The last examples may be generalized to other members of the family described in \cite{MR1227512},
the family that includes both the Bernstein--Iarrobino and the Stanley nonunimodal examples.
We leave this generalization to the reader.
\end{remark}

\subsection{Simultaneous Waring rank}

\begin{defn}
Let $W \subseteq S^d V$ be a linear series of degree $d$ forms.
For $0 \leq a$, let $\Sigma_{(1,a)}(W) = \bigcap_{F \in W} \Sigma_a(F)$.
That is, $\Sigma_{(1,a)}(W)$ is the locus of points in $\PP V^*$
at which every member of the linear series $W$ vanishes to order at least $a+1$.
As a scheme, $\Sigma_{(1,a)}(W)$ is defined by the vanishing of all the $a$th partial derivatives
of all the members of $W$; that is, by the image of the map $C^{(1,a)}_W$.
\end{defn}

\begin{defn}
A linear series $W \subset S^d V$ is called \defining{concise} (with respect to $V$)
if it satisfies the following equivalent conditions:
\begin{enumerate}
\item $W$ cannot be written as a linear series in a smaller number of variables;
that is, if $W \subseteq S^d V'$ for some $V' \subseteq V$, then $V' = V$.
\item If $V' \subset V$ is such that for every $F \in W$, $V(F)$ is a cone over $\PP V'$,
then $V' = 0$.
\item The set $\widehat{\Sigma}_{(1,d-1)}(W) = \{0\}$.
\item A general $F \in W$ is concise.
\item $C^{(1,d-1)}_W$ is surjective, $C^{(0,1)}_W$ is injective.
\end{enumerate}
\end{defn}

\begin{thm}\label{thm: linear series improvement first version}
Let $W \subseteq S^d V$ and $0 \leq a < d$.
If $W$ is concise then
\[
  r(W) \geq \rank C^{(0,d-a)}_W + \dim \widehat{\Sigma}_{(1,a)}(W) .
\]
\end{thm}

\begin{proof}
Suppose $W$ is contained in the span of $\ell_1^d,\dotsc,\ell_r^d$.
Let $\mathcal{L}$ be the linear series of degree $d-a$ forms vanishing at the $[\ell_i]$.
Then $\mathcal{L} \subseteq \ker C^{(0,d-a)}_W$.
Indeed for $\Theta \in \cL$ and $F \in W$, say $F = \sum c_i \ell_i^d$,
we have $C^{(0,d-a)}_W(\Theta)(F) = \Theta F = \sum c_i \Theta \ell_i^d = 0$.
This already shows
$r \geq \codim \mathcal{L} \geq \codim \ker C^{(0,d-a)}_W = \rank C^{(0,d-a)}_W$.

By hypothesis the $[\ell_i]$ do not lie on any hyperplane,
so $\mathcal{L}$ does not contain any power of a hyperplane $h^{d-a}$.
Thus $\mathcal{L}$ is disjoint from the Veronese variety $\nu_{d-a}(\PP V^*)$.
As before,
$r \geq \codim \mathcal{L} > \codim \ker C^{(0,d-a)}_W + \dim (\PP \ker C^{(0,d-a)}_W \cap \nu_{d-a}(\PP V^*))$.

We claim that $\PP \ker C^{(0,d-a)}_W \cap \nu_{d-a}(\PP V^*) \cong \Sigma_{(1,a)}(W)$.
Let $h \in \PP V^*$.
Then $h^{d-a} \in \ker C^{(0,d-a)}_W$ if and only if $h^{d-a} F = 0$ for all $F \in W$,
if and only if for all $F \in W$, $F$ vanishes at $[h]$ to order at least $a+1$,
if and only if $[h] \in \bigcap_{F \in W} \Sigma_a(F) = \Sigma_{(1,a)}(W)$.
\end{proof}

%

\begin{example}\label{example: matrix linear series singularity improvement}
Fix a generic $m \times n$ matrix $X = (x_{i,j})$, $1 \leq i \leq m$, $1 \leq j \leq n$.
Let $D_k, P_k, R_k$ be as in Example~\ref{example: matrix linear series}.
The image of $C^{(1,a)}_{D_k}$ is $D_{k-a}$, and similarly for $P_k$ and $R_k$,
by Example~\ref{example: matrix linear series catalecticants},
which also contains the ranks of those catalecticants.

Thus $\Sigma_{(1,a)}(D_k)$ is the common zero locus of $D_{k-a}$,
the locus of matrices of rank $\leq k-a-1$.
This has dimension $mn-(m-(k-a-1))(n-(k-a-1)) = (k-a-1)(m+n-k+a+1)$.
Therefore $r(D_k) \geq \binom{m}{k-a}\binom{n}{k-a} + (k-a-1)(m+n-k+a+1)$.

$\Sigma_{(1,a)}(P_k)$ is the common zero locus of $P_{k-a}$,
the locus of matrices all of whose $(k-a) \times (k-a)$ submatrices have permanent zero.
This includes at least all matrices with only $k-a-1$ nonzero columns or $k-a-1$ nonzero rows,
so $\dim \Sigma_{(1,a)}(P_k) \geq \max(m,n)(k-a-1)$
and $r(P_k) \geq \binom{m}{k-a}\binom{n}{k-a} + \max(m,n)(k-a-1)$.
See \cite{MR1673948}, \cite{MR1777172}, \cite{MR2386244}, \cite[\textsection5.4]{MR1925796}
for more on the zero locus of permanental ideals.

We treat $R_k$ separately in the next example.
\end{example}

\begin{example}\label{example: rookfree linear series singularity improvement}
Fix a generic $m \times n$ matrix $X = (x_{i,j})$, $1 \leq i \leq m$, $1 \leq j \leq n$,
and let $R_k$ be as in Example~\ref{example: matrix linear series}.
As in the previous example, $\Sigma_{(1,a)}(R_k)$ is the common zero locus of $R_{k-a}$.
We claim that $\dim \Sigma_{(1,a)}(R_k) = \dim V(R_{k-a}) = \max(m,n)(k-a-1)$.

Note the ideal $\langle R_{t} \rangle$ is a squarefree monomial ideal,
so it can be written as the Stanley-Reisner ideal $I_\Delta$
of a simplicial complex $\Delta$ on the entries of $X$, i.e., the $x_{i,j}$, $1 \leq i \leq m$, $1 \leq j \leq n$.
That is, $\langle R_t \rangle$ is generated by the minimal nonfaces of $\Delta$;
see for example \cite[\textsection1.1]{MR2110098}.
We claim that $\langle R_t \rangle$ is the Stanley-Reisner ideal of the
simplicial complex $\Delta(t)$ whose facets are given by
unions of $t-1$ rows and columns of $X$, that is, each facet is
the union of $a$ rows and $b$ columns where $a+b = t-1$.

Granting that $\langle R_t \rangle = I_{\Delta(t)}$,
the largest facets of $\Delta(t)$ include $\max(m,n)(t-1)$ entries of $X$,
so their complements define linear subspaces of that dimension;
all other facets correspond to smaller components,
showing that $\dim V(R_t) = \max(m,n)(t-1)$, as desired.

Clearly $\langle R_1 \rangle = I_{\Delta(1)}$ (trivially).
It holds as well for $t=2$: a square-free product contains two rook-free elements if and only if it does not
lie on a ``union'' of one row or of one column.

A square-free monomial $M$ corresponds to a subset of the entries of the matrix,
or to the set of edges in a subgraph $G = G(M)$ of the complete bipartite graph $K_{m,n}$.
Recall K\"onig's Theorem (see, e.g., \cite[Thm.~2.1.1]{MR2744811}), which asserts
that the smallest size of a vertex cover of the bipartite graph $G$
is equal to the largest size of a matching.
A matching in $G$ (a set of pairwise non-adjacent edges) corresponds to a rook-free product contained in $M$.
A vertex cover (a set of vertices incident to every edge) corresponds to a set of rows and columns whose union contains $M$.
Thus a monomial $M$ lies in $R_t$ if and only if $M$ contains a rook-free product of degree $t$,
if and only if $G(M)$ has a matching of size at least $t$, if and only if $G(M)$ has no vertex cover of size less than $t$,
if and only if $M$ is not contained in any union of $t-1$ rows and columns.

In conclusion, $\dim V(R_t) = \max(m,n)(t-1)$,
\[
  \dim \Sigma_{(1,a)}(R_k) = \dim V(R_{k-a}) = \max(m,n)(k-a-1) ,
\]
and
\[
  \r(R_k) \geq \binom{m}{k-a}\binom{n}{k-a}(k-a)! + \max(m,n)(k-a-1) ,
\]
for $1 \leq a \leq k$.
\end{example}


\subsection{Simultaneous Waring rank, second version}

Now we describe an improvement to the bound $\r(W) \geq \rank C^{(0,a)}_W$.
For the first time, we need to consider the projective version of the singular set $\Sigma$
(rather than the affine version $\widehat\Sigma$).

\begin{defn}
Let $W \subseteq S^d V$ be a linear series of degree $d$ forms.
For $0 \leq a < d$, let $\Sigma_{(0,a)}(W) = \{([F],[p]) \in \PP W \times \PP V^* \mid [p] \in \Sigma_a(F)\}$.
\end{defn}

Thus $\Sigma_{(0,a)}(W)$ is the scheme defined by the image of the map $C^{(0,a)}_W$.
Indeed, the correspondence $\Hom(W,S^{d-a} V) \cong W^* \otimes S^{d-a} V$ takes a map $\Theta: W \to S^{d-a} V$ to
$f_\Theta \in W^* \otimes S^{d-a} V$,
the function on $W \times S^{d-a} V^*$ defined by $f_\Theta(F,\psi) = \langle \Theta F, \psi \rangle = \psi \Theta F \in \Bbbk$.
In particular for $h \in V^*$, $f_\Theta(F,h^{d-a}) = \Theta F \big|_h$,
the evaluation of $\Theta F$ at the point $h$ (up to a factorial factor).
So $f_\Theta$ vanishes at a point $(F,p) \in W \times V^* \hookrightarrow W \times S^{d-a} V^*$
if and only if $\Theta F$ vanishes at $p$.
In our case, a point $([F],[p])$ is a common zero of every element in the image of $C^{(0,a)}_W$
if and only if $\Theta F \big|_p = 0$ for every $\Theta \in S^a V^*$, equivalently $[p] \in \Sigma_a(F)$.

In order to improve $\r(W) \geq \rank C^{(1,a)}_W$, we needed to assume that the map $C^{(1,d-1)}_W$ was surjective,
equivalently that $W$ was concise.
Now we will have to assume that the map $C^{(0,d-1)}_W : S^{d-1} V^* \to \Hom(W,V) \cong W^* \otimes V$ is surjective.
This is considerably stronger.
If the map $C^{(0,d-1)}_W$ is surjective, then every nonzero $F \in W$ is concise
(for every nonzero linear form $\ell \in V$ there is a linear map $W \to V$ taking $F \mapsto \ell$,
and it is realized as $C^{(0,d-1)}_W \Theta$ for some $\Theta \in S^{d-1} V^*$, which means $\Theta F = \ell$;
so $C^{d-1}_F$ is surjective).
And of course if every nonzero $F \in W$ is concise (and $W$ is nontrivial) then $W$ is concise.

Clearly having $W$ concise does not imply that every $F \in W$ is concise:
the linear series in Example~\ref{example: matrix linear series} are concise
but are spanned by forms (determinant, permanent, rook-free product) that only depend on
the variables in a submatrix.
And having every nonzero $F \in W$ be concise does not imply that $C^{(0,d-1)}_W$ is surjective.
Let $n=2$ and let $W$ be the pencil spanned by $x^3 y^2, x^2 y^3$.
Every member of $W$ is of the form $x^2 y^2 (ax+by)$ and this is concise because it is not a perfect power.
But $C^{(0,4)}_W$ is not surjective.
If $\Theta = a_4 \d_x^4 + a_3 \d_x^3 \d_y + \dotsb + a_0 \d_y^4$ has $\Theta x^3 y^2 = 0$ then $a_3 = a_2 = 0$,
and then $\Theta x^2 y^3 = 6a_1 x$; so there is no $\Theta \in S^4 V^*$ such that $\Theta x^3 y^2 = 0$, $\Theta x^2 y^3 = y$.

So $C^{(0,d-1)}_W$ being surjective is a strong condition.
Nevertheless it can be met.
For example, let $n=2$ and let $W$ be the pencil spanned by $x^4 y^2$, $x^2 y^4$.
Then we have
\[
\begin{array}{ll}
  (\d_x^3 \d_y^2)(x^4 y^2) = 48 x, & (\d_x^3 \d_y^2)(x^2 y^4) = 0, \\
  (\d_x^4 \d_y)(x^4 y^2) = 48 y, & (\d_x^4 \d_y)(x^2 y^4) = 0, \\
  (\d_x \d_y^4)(x^4 y^2) = 0,  & (\d_x \d_y^4)(x^2 y^4) = 48 x, \\
  (\d_x^2 \d_y^3)(x^4 y^2) = 0, & (\d_x^2 \d_y^3)(x^2 y^4) = 48 y, \\
\end{array}
\]
which shows that $C^{(0,5)}_W$ is surjective onto $\Hom(W,V)$.

\begin{thm}\label{thm: linear series improvement second version}
Let $W \subseteq S^d V$ and $0 \leq a < d$.
If $C^{(0,d-1)}_W$ is surjective then
\[
  \r(W) \geq \rank C^{(1,d-a)}_W + \dim \Sigma_{(0,a)}(W) + 1,
\]
where $\dim \varnothing = -1$.
\end{thm}
\begin{proof}
Suppose $W$ is contained in the span of $\ell_1^d,\dotsc,\ell_r^d$
and these are linearly independent.
Every element of $W \otimes S^{d-a} V^*$ can be written in the form
\[
  \sum_{i=1}^r \ell_i^d \otimes \Theta_i,
\]
the $\Theta_i$ being uniquely determined by the linear independence of the $\ell_i^d$.
Let
\[
  \cL = \left\{ \sum_{i=1}^r \ell_i^d \otimes \Theta_i \in W \otimes S^{d-a} V^*
    \: \middle| \: \Theta_1([\ell_1]) = \dotsb = \Theta_r([\ell_r]) = 0 \right\} .
\]
Then $\cL \subseteq \ker C^{(1,d-a)}_W$ clearly.
Since $\Theta_i([\ell_i]) = 0$ imposes just one condition on the polynomial $\Theta_i$,
we have
\[
  r \geq \codim \cL \geq \codim \ker C^{(1,d-a)}_W = \rank C^{(1,d-a)}_W.
\]
Next we claim that $\cL$ does not contain any nonzero element of the form $F \otimes h^{d-a}$, $F \in W$, $h \in V^*$.
For if $F \otimes h^{d-a} \in \cL$, say $F = \sum c_i \ell_i^d$, then
\[
  F \otimes h^{d-a} = \sum \ell_i^d \otimes c_i h^{d-a} \in \cL,
\]
whence $c_1 h([\ell_1]) = \dotsb = c_r h([\ell_r]) = 0$.
If each $c_i \neq 0$ this implies the $[\ell_i]$ lie on the hyperplane defined by $h$, and $W$ is not concise.
In general (allowing that some $c_i$ may be zero) all we can say is that
\[
  C^{(1,1)}_W(F \otimes h) = \sum c_i h(\ell_i^d) = 0,
\]
which means $C^{(1,1)}_W$ is not injective and $C^{(0,d-1)}_W$ is not surjective, contradicting the hypothesis.

Finally we have $F \otimes h^{d-a} \in \ker C^{(1,d-a)}_W$ if and only if $h^{d-a}(F) = 0$.
Proposition~\ref{prop: catalecticant kernel veronese intersection}
implies that this happens if and only if $[h] \in \Sigma_a(F)$,
equivalently if and only if $([F],[h]) \in \Sigma_{(0,a)}(W)$.
The points $F \otimes h^{d-a}$ correspond precisely to points in the Segre embedding of
$\PP W \times \nu_{d-a}(\PP V^*)$ in $\PP(W \otimes S^{d-a} V^*)$.
We have shown that the intersection of this Segre variety with the kernel of the catalecticant satisfies
\[
  \PP \ker C^{(1,d-a)}_W \cap \Seg(\PP W \times \nu_{d-a}(\PP V^*)) \cong \Sigma_{(0,a)}(W).
\]

Putting this all together we have
\[
\begin{split}
  r &\geq \codim \cL \\
    &> \rank C^{(1,d-a)}_W + \dim \Big\{ \PP \ker C^{(1,d-a)}_W \cap \Seg(\PP W \times \nu_{d-a}(\PP V^*)) \Big\} \\
    &= \rank C^{(1,d-a)}_W + \dim \Sigma_{(0,a)}(W),
\end{split}
\]
as claimed.
\end{proof}

\begin{example}
Here is an example in which Theorem~\ref{thm: linear series improvement second version}
gives a better bound than Theorem~\ref{thm: linear series improvement first version}.

Let $n=2$ and let $W$ be the pencil spanned by $x^6 y^3$ and $x^4 y^5$.
We have $\r(W) \leq \r(x^6 y^3) + \r(x^4 y^5) = 7 + 5 = 12$.
Better, $\r(W) \leq \r(x^6 y^3 + x^4 y^5) + \r(x^6 y^3 - x^4 y^5) = 5 + 5 = 10$.

For lower bounds, we have
\[
\begin{array}{l lllll lllll}
  \toprule
  a                   & 0 & 1 & 2 & 3 & 4 & 5 & 6 & 7 & 8 & 9 \\
  \midrule
  \rank C^{(1,9-a)}_W & 1 & 2 & 3 & 4 & 5 & 6 & 6 & 5 & 4 & 2 \\
  \bottomrule
\end{array}
\]
In particular $C^{(1,1)}_W$ is injective.
Other than the spanning elements, elements of $W$ are of the form
\[
  b x^6 y^3 + c x^4 y^5 = x^4 y^3 (b x^2 + c y^2)
\]
which has a quadruple root at $x=0$, a triple root at $y=0$, and two simple roots when $bc \neq 0$.
Therefore
\[
  \Sigma_{(1,a)}(W) =
  \begin{cases}
    V(x) \cup V(y), & 0 \leq a \leq 2 \\
    V(x), & a = 3 \\
    \varnothing, & 4 \leq a
  \end{cases}
\]
and
\[
  \dim \Sigma_{(0,a)}(W) =
  \begin{cases}
    1, & 0 \leq a \leq 3 \\
    0, & 4 \leq a \leq 5 \\
    -1, & 6 \leq a
  \end{cases}
\]
We have
\[
\begin{array}{l lllll lllll}
  \toprule
  a                                & 0 & 1 & 2 & 3 & 4 & 5 & 6 & 7 & 8 & 9 \\
  \midrule
  \rank C^{(0,9-a)}_W              & 2 & 4 & 5 & 6 & 6 & 5 & 4 & 3 & 2 & 1 \\
  \dim \widehat{\Sigma}_{(1,a)}(W) & 1 & 1 & 1 & 0 & 0 & 0 & 0 & 0 & 0 & 0 \\
  \bottomrule
\end{array}
\]
so the lower bound given by Theorem~\ref{thm: linear series improvement first version} is $\r(W) \geq 6$.
And we have
\[
\begin{array}{l lllll lllll}
  \toprule
  a                      & 0 & 1 & 2 & 3 & 4 & 5 & 6 & 7 & 8 & 9 \\
  \midrule
  \rank C^{(1,9-a)}_W    & 1 & 2 & 3 & 4 & 5 & 6 & 6 & 5 & 4 & 2 \\
  \dim \Sigma_{(0,a)}(W) & 1 & 1 & 1 & 1 & 0 & 0 & -1 & -1 & -1 & -1 \\
  \bottomrule
\end{array}
\]
so the lower bound given by Theorem~\ref{thm: linear series improvement second version} is $\r(W) \geq 7$.
\end{example}

\subsection{Multihomogeneous polynomials}

$M \in S^{\bd} V$ is \defining{concise} if $M \in S^{\bd} V' = S^{d_1} V'_1 \otimes \dotsb \otimes S^{d_s} V'_s$
for $V'_1 \subseteq V_1, \dotsc, V'_s \subseteq V_s$
implies each $V'_i = V_i$.
This is equivalent to $C^{\bd-\be}_M$ being surjective,
or $C^{\be}_M$ injective,
for each $\be=(1,0,\dotsc,0), \dotsc, (0,\dotsc,0,1)$.

However, we will need a strengthening of conciseness,
involving $\be$ with possibly more than one nonzero entry.

\begin{notation}
For a tuple $\ba = (a_1,\dotsc,a_s)$,
we define $\rad{\ba} = (e_1, \dotsc, e_s)$
where $e_i = 1$ if $a_i > 0$, otherwise $e_i = 0$.
\end{notation}

Say that $M$ vanishes to multiorder $\bb = (b_1,\dotsc,b_s)$
at a point $P \in \prod \PP V_i^*$ if, for every differential operator $D$ of multidegree $\bb$,
$D(M)$ vanishes at $P$.
Equivalently, choosing local coordinates centered at $P = (p_1,\dotsc,p_s)$ by choosing local coordinates in each $\PP V_i$
centered at $p_i$ and writing $M$ as a (non-homogeneous) polynomial in these coordinates,
no monomial appearing in $M$ has multidegree less than or equal to $\bb$.

We define $\Sigma_{\ba} = \Sigma_{\ba}(M) \subset \prod \PP V_i^*$
to be the subvariety defined by the image of $C^{\ba}_M$, regarded as a set of multihomogeneous polynomials.
Equivalently, $\Sigma_{\ba}(M)$ is the locus of points at which the multiorder of vanishing of $M$
is not less than or equal to $\ba$.

The bound for ranks of multihomogeneous polynomials is the following:
\begin{thm} \label{thm: multihomogeneous polynomial improvement}
Let $M$ be a multihomogeneous polynomial as above.
If $C^{\rad{\bd-\ba}}_M$ is injective 
then
$\r_{MH}(M) \geq \rank(C^{\bd-\ba}_M) + \dim \Sigma_{\ba}(M) + 1$,
where $\dim \varnothing = -1$.
\end{thm}

Conciseness would not be enough; it only gives injectivity of
$C^{\rad{\bd-\ba}}_M$ when $\bd-\ba$ has a single nonzero entry.

\begin{proof}
Suppose
\[
    M = \sum_{j=1}^{r}  \ell_{j,1}^{d_1} \dotsb \ell_{j,s}^{d_s} .
\]
For each $j$, let $P_j = ([\ell_{j,1}], \dotsc, [\ell_{j,s}]) \in \prod \PP V_i$.
Let $\cL = \{ p \in S^{\bd - \ba} V^* \mid p(P_1) = \dotsb = p(P_r) = 0 \}$.
Denote by $D_p$ the differential operator associated to $p$.
Then
\[
  D_p( \ell_{j,1}^{d_1} \dotsm \ell_{j,s}^{d_s} )
    = \frac{\bd !}{\ba !} \, p(\ell_{j,1},\dotsc,\ell_{j,s}) \, \ell_{j,1}^{a_1} \dotsm \ell_{j,s}^{a_s} .
\]
Thus for $p \in \cL$, $D_p(M) = 0$.
This shows $\cL \subseteq \ker C^{\bd-\ba}_M$,
hence
\[
  r \geq \codim \cL \geq \codim \ker C^{\bd-\ba}_M = \rank C^{\bd-\ba}_M .
\]

Now we use the additional hypothesis that $C^{\rad{\bd-\ba}}_M$ is injective.
With this assumption, we claim $\PP \cL$ is disjoint from
$\PP \ker C^{\bd-\ba}_M \cap \Seg(\prod_{i=1}^{s} \nu_{d_i - a_i}(\PP V_i^*))$,
and that this intersection is isomorphic to $\Sigma_\ba(M)$.
\begin{remark}
The recent paper \cite{Abo:2011lr}
shows the defectivity of certain secant varieties
by considering the intersection of a Segre-Veronese variety with the image (rather than kernel)
of a catalecticant map.
\end{remark}
If the disjointness fails, then $\cL$ contains an element $p = h_1^{d_1-a_1} \dotsm h_k^{d_k-a_k}$
for some $h_i \in V_i^*$, $h_i \neq 0$, $1 \leq i \leq k$.
For each $1 \leq j \leq r$,
\[
  p(P_j) = p(\ell_{j,1},\dotsc,\ell_{j,k}) = h_1(\ell_{j,1})^{d_1-a_1} \dotsm h_k(\ell_{j,k})^{d_k-a_k} = 0.
\]
We must have $h_i(\ell_{j,i})=0$ for some $i$ such that $d_i - a_i > 0$.
Let $\be = (e_1,\dotsc,e_k) = \rad{\bd - \ba}$.
Then $h_i(\ell_{j,i})^{e_i}=0$, so also $h_1(\ell_{j,1})^{e_1} \dotsm h_k(\ell_{j,k})^{e_k} = 0$.
This holds for all $j$, so $h_1^{e_1} \dotsm h_k^{e_k} \in \ker C^{\be}_M$,
contradicting that $C^{\be}_M$ is injective.
Hence if $C^{\rad{\bd-\ba}}_M$ is injective
then $\cL$ is disjoint from the intersection, as claimed.

Next, we claim an element $h_1^{d_1-a_1} \dotsm h_k^{d_k-a_k}$ lies in $\ker C^{\bd - \ba}_M$
if and only if $([h_1],\dotsc,[h_k]) \in \Sigma_{\ba}(M)$.
We have $h_1^{d_1-a_1} \dotsm h_k^{d_k-a_k} \in \ker C^{\bd-\ba}_M$
if and only if $(h_1^{d_1-a_1} \dotsm h_k^{d_k-a_k})(M) = 0$,
if and only if $D(h_1^{d_1-a_1} \dotsm h_k^{d_k-a_k})(M) = 0$
for all type $\ba$ differential operators $D \in S^{\ba}(V^*)$,
if and only if $(h_1^{d_1-a_1} \dotsm h_k^{d_k-a_k})(D(M)) = 0$ for all such $D$.
Since $D(M)$ is a multihomogeneous polynomial of multidegree $\bd-\ba$,
we have that this latter is equal to the evaluation at the point $([h_1],\dotsc,[h_k]) \in \prod \PP V_i^*$,
up to a scalar:
\[
  (h_1^{d_1-a_1} \dotsm h_k^{d_k-a_k})(D(M))
  =
  (\bd - \ba) !  \cdot  D(M) \big|_{ (h_1,\dotsc,h_k) } .
\]
That this vanishes for all $D$ is exactly the condition that $([h_1], \dotsc, [h_k]) \in \Sigma_{\ba}(M)$,
as claimed.
\end{proof}

\begin{example}\label{example: product of variables bihomogeneous singularity improvement}
Let $F = x_1 \dotsm x_a y_1 \dotsm y_b$, a bihomogeneous form of bidegree $(a,b)$.
One may check easily that $C^{(1,0)}_F$, $C^{(0,1)}_F$, and $C^{(1,1)}_F$ are injective.
For $\ba = (p,q)$ with $0 \leq p \leq a$, $0 \leq q \leq b$,
the image of $C^{\ba}_F$ is spanned by subproducts of $a-p$ of the $x$'s and $b-q$ of the $y$'s.
The common vanishing locus $\Sigma_{\ba}(F)$ is points with at least $p+1$ of the $x$'s vanishing
or at least $q+1$ of the $y$'s vanishing, equivalently,
at most $a-p-1$ of the $x$'s nonvanishing or at most $b-q-1$ of the $y$'s nonvanishing.
This is a finite union of products $\PP^{a-p-2} \times \PP^{b-1}$ and $\PP^{a-1} \times \PP^{b-q-2}$,
with dimension $\max\{a+b-p-3,a+b-q-3\} = a+b-3-\min\{p,q\}$.
This shows that $\r_{MH}(x_1 \dotsm x_a y_1 \dotsm y_b) \geq \binom{a}{p}\binom{b}{q} + a+b-2-\min\{p,q\}$.
For instance, $\r_{MH}(x_1 x_2 x_3 y_1 y_2) \geq \binom{3}{1}\binom{2}{1} + 3-1 = 8$.
Since $\r_{MH}(x_1 x_2 x_3 y_1 y_2) \leq 2^{3+2-2} = 8$, this determines the multihomogeneous rank.
\end{example}

\begin{example}\label{example: determinant bihomogeneous singularity improvement}
The generic determinant $\det_n$ is bihomogeneous of bidegree $(a,n-a)$
in the variables from the first $a$ and last $n-a$ rows.
But $C^{(1,1)}_{\det_n}$ is not injective:
the kernel is spanned by elements of the following two types.
\begin{enumerate}
\item $\d_{i_1,j} \d_{i_2,j}$, a product of two differentials in the same column, with $i_1 \leq a < i_2$.
\item $\d_{i,j} \d_{k,\ell} + \d_{i,\ell} \d_{k,j}$, the permanent of a $2 \times 2$ submatrix with $i \leq a < k$.
\end{enumerate}
(It is easy to see that these elements are in the kernel, and the row conditions ensure that they have bidegree $(1,1)$.
Shafiei's theorem \cite{Shafiei:ud} describes $\det_n^\perp$ and implies that the above elements span
the kernel of $C^{(1,1)}_{\det_n}$.)
So Theorem~\ref{thm: multihomogeneous polynomial improvement} is not applicable to $\r_{MH}(\det_n)$.
\end{example}

%
%

\subsection{Generalized ranks}

Finally, we give an improved lower bound for generalized ranks.

\begin{thm}\label{thm: generalized rank improvement}
Let $X$ be a smooth irreducible variety.
Let $L$ be a very ample line bundle on $X$ and let $V = H^0(X,L)^*$ (so $X \hookrightarrow \PP V$).
Let $v \in V$.
Let $G$ be a vector bundle on $X$, $b > 0$, and $E = G^b$;
more generally assume $E$ is a vector bundle on $X$
and there is a bundle map $G^b \to E$ whose kernel has no global sections.
Let $e = \rank E$.
Assume $C^G_v$ is injective.
Then $e \, r_X(v) > \rank C^E_v + \dim \Sigma$,
where $\Sigma \subset \PP H^0(X,G) \overset{\nu_b}{\hookrightarrow} \PP H^0(X,E)$
is the subvariety of $\PP H^0(X,G)$ defined by the image of the transpose
$(C^E_v)^t = C^{L \otimes E^*}_v : H^0(X,L \otimes E^*) \to H^0(X,E)^*$
(and $\dim \Sigma = -1$ if $\Sigma = \varnothing$).
\end{thm}
(Here linear equations on $\PP H^0(X,E)$ induce degree $b$ equations on the Veronese image $\nu_b(\PP H^0(X,G))$.
$\Sigma$ is the locus defined by the equations arising from $\img (C^E_v)^t$.)

\begin{proof}
Let $v = x_1 + \dotsb + x_r$, each $[x_i] \in X$.
Let $\cL = \{ h \in H^0(E) \mid h([x_1]) = \dotsb = h([x_r]) = 0 \}$.
Clearly $C^E_v(h) = 0$ for each $h \in \cL$, so $\cL \subseteq \ker C^E_v$.
Each $h([x_i])=0$ imposes $e$ conditions on the global section $h$, so $\cL$ is defined by a system of $er$ equations.
This shows $er \geq \codim \cL \geq \codim \ker C^E_v = \rank C^E_v$.

If $\cL$ contains $h^b$ for any $h \in H^0(X,G)$ then $h \in \ker C^G_v$, contradicting the hypothesis.
Thus $\PP \cL$ is disjoint from $\nu_b(\PP H^0(X,G)) \cap \PP \ker C^E_v$.
Hence
\[
  er \geq \codim \PP\cL > \codim \PP\ker C^E_v + \dim(\nu_b(\PP H^0(X,G)) \cap \PP \ker C^E_v) .
\]
Finally, for $h \in H^0(X,G)$, $h^b \in \PP \ker C^E_v$ if and only if $h^b$ is annihilated by each element of the image
of the transpose $(C^E_v)^t$.
This shows $\nu_b(\PP H^0(X,G)) \cap \PP \ker C^E_v \cong \Sigma$.
\end{proof}

\section{Apolarity Lemmas}\label{section: apolarity lemmas}

In this section we go beyond considering just rank to actually considering the terms that arise in a Waring decomposition.
These are related to certain containments of ideals, corresponding to schemes called \emph{apolar schemes}.

\subsection{Classical Waring rank}

\begin{defn}
Let $S = \Bbbk[x_1,\dotsc,x_n]$ and let $T = \Bbbk[\d_1,\dotsc,\d_n]$ be the dual ring.
Let $F \in S_d$.
Then $F^\perp = \{ \Theta \in T : \Theta F = 0 \}$ is a homogeneous ideal, called the \defining{apolar ideal}
or \defining{annihilating ideal} of $F$.

The quotient ring $A^F = T / F^\perp$ is called the \defining{apolar algebra} of $F$.
\end{defn}

Clearly $F^\perp$ contains every form of degree $d+1$ or greater.
Correspondingly, $A^F$ is an Artinian algebra.
In fact $A^F$ is an Artinian Gorenstein algebra, and every Artinian Gorenstein algebra (finitely generated, standard graded)
is isomorphic to an apolar algebra $A^F$ for some $F$.

Each graded piece $F^\perp_{d-a}$ is the kernel of a catalecticant, $F^\perp_{d-a} = \ker C^{d-a}_F$.
The catalecticants are just the graded pieces of the quotient map $T \to A^F$.
In particular the rank of $C^{d-a}_F$ is equal to the value of the Hilbert function $h_{A^F}(a) = \dim ((A^F)_a)$.
The lower bound $\r(F) \geq \max_{0 \leq a \leq d} \rank C^F_{d-a}$ of Section~\ref{section: catalecticants - classical Waring rank}
becomes $\r(F) \geq \max_{0 \leq a \leq d} h_{A^F}(a)$.

\begin{thm}[Apolarity Lemma]
Let $F \in S_d$. Let $\ell_1,\dotsc,\ell_r \in S_1$ and $I = I([\ell_1],\dotsc,[\ell_r])$.
Then there are scalars $c_1,\dotsc,c_r$ such that $F = c_1 \ell_1^d + \dotsb + c_r \ell_r^d$
if and only if $I \subset F^\perp$.
\end{thm}
This is the ``classical'' Apolarity Lemma.
See for example \cite[Theorem 5.3]{MR1735271}, \cite[\textsection1.3]{MR1780430}.
\begin{proof}
We have seen for each graded piece $\cL = I_k$ and $\Theta \in \cL$ that $\Theta \ell_i^d = 0$ for each $i$,
so $\Theta F = 0$ and $\Theta \in \ker C^F_k$.
This shows $I \subset F^\perp$.

Conversely if $I \subset F^\perp$, in particular $I_d \subset F^\perp_d$.
Note $I_d = \bigcap_{i=1}^r I(\ell_i)_d = \bigcap_{i=1}^r (\ell_i^d)^\perp$.
Then $\Span\{F\} = (F^\perp)^\perp \subseteq \sum_{i=1}^r \Span\{\ell_i^d\} = \Span\{\ell_1^d,\dotsc,\ell_r^d\}$, as desired.
\end{proof}

Here is a version for schemes which is well-known to experts.

\begin{thm}
Let $F \in S_d$.
Let $Z \subset \PP V$ be an arbitrary scheme with saturated homogeneous defining ideal $I = I(Z)$.
Let $\nu_d : \PP V \to \PP S^d V$ be the degree $d$ Veronese map.
Then $[F]$ lies in the linear span of the scheme $\nu_d(Z)$ if and only if $I \subset F^\perp$.
\end{thm}
\begin{proof}
$[F]$ lies in the linear span of $\nu_d(Z)$ if and only if every linear form vanishing on $\nu_d(Z)$ also vanishes on $[F]$,
that is, $F$ is annihilated by the space of linear forms on $\PP S^d V$ that vanish on $\nu_d(Z)$;
these are precisely the degree $d$ forms on $\PP V$ that vanish on $Z$, i.e., the degree $d$ piece of $I$.
So $[F]$ lies in the linear span of $\nu_d(Z)$ if and only if $F \in (I_d)^\perp$, equivalently $I_d \subset (F^\perp)_d$.
For $d = \deg F$, $I_d \subset (F^\perp)_d$ if and only if $I \subset F^\perp$
by for example \cite[Proposition~3.4(iii)]{MR3121848} or \cite[Lemma~2.15]{MR1735271}.
\end{proof}

A scheme $Z \subset \PP V$ is called \defining{apolar to $F$} if its defining ideal is contained in $F^\perp$,
equivalently $[F]$ is in the linear span of $\nu_d(Z)$.
Thus the Waring rank of $F$ is equal to the least length of a reduced zero-dimensional apolar scheme to $F$.
This leads to obvious generalizations:
The \defining{smoothable rank} $s\r(F)$ of $F$ is equal to the least length of a smoothable zero-dimensional apolar scheme to $F$.
The \defining{cactus rank} $c\r(F)$ of $F$ is equal to the least length
of any zero-dimensional apolar scheme to $F$.
Evidently $c\r(F) \leq s\r(F) \leq \r(F)$.
For many more notions of rank, see \cite{Bernardi:2012fk}.

\begin{remark}
Let $X \subset \PP^n$.
A \defining{secant $(r-1)$-plane to $X$} is an $(r-1)$-plane $\langle x_1,\dotsc,x_r \rangle$
spanned by distinct, reduced points $x_1,\dotsc,x_r$ in $X$.
The \defining{$r$th secant variety of $X$}, denoted $\sigma_r(X)$,
is the Zariski closure of the union of the secant $(r-1)$-planes:
\[
  \sigma_r(X) = \overline{ \bigcup \{ \langle R \rangle \mid R \subset X, \text{$R$ consists of at most $r$ distinct points in $X$} \} } .
\]
Equivalently it is the Zariski closure of the set of points of rank at most $r$.
For $q \neq 0$, the \defining{border rank} of $q$ with respect to $X$,
denoted $b\r_X(q)$, is the least $r$ such that $[q] \in \sigma_r(X)$.
By definition, $b\r_X(q) \leq \r_X(q)$.

For a $d$-form $F$, the border rank with respect to the Veronese variety is denoted $b\r(F)$.
It is the least $r$ such that $F$ is a limit of forms of Waring rank $r$ or less.
For example $\r(x y^{d-1}) = d$, but
\[
  x y^{d-1} = \lim_{t \to 0} \frac{1}{dt} \Big( (y+tx)^d - y^d \Big),
\]
so $x y^{d-1}$ is a limit of forms of rank $2$.
Thus $b\r(x y^{d-1}) \leq 2$.
In fact $b\r(x y^{d-1}) = 2$ by the catalecticant lower bound which we describe now.

It turns out that $b\r(F) \leq s\r(F) \leq \r(F)$.
The inequality $b\r(F) \leq s\r(F)$ holds because of the closure operation in the definition of secant variety.
Strict inequality may hold, see \cite{Buczynska:2013zh}.

The catalecticant lower bound \eqref{eq: catalecticant bound single homogeneous form}
is in fact a lower bound for border rank: $\rank C^a_F \leq b\r(F)$.
Indeed, for all forms $F$ of Waring rank $r$ or less, the catalecticant $C^a_F$ has rank at most $r$,
so all the $(r+1)$-minors of a matrix representing $C^a_F$ vanish.
Then these minors vanish also on the Zariski closure of the locus of forms of Waring rank $r$ or less,
which is the $r$th secant variety, i.e., exactly the locus of forms of border rank $r$ or less.
This shows that if $b\r(F) \leq r$ then $\rank C^a_F \leq r$.
Hence $\rank C^a_F \leq b\r(F)$.
\end{remark}

And the catalecticant lower bound is also a lower bound for cactus rank: $\rank C^a_F \leq c\r(F)$.
See \cite[Thm.~5.3D]{MR1735271}.
(Briefly: if $Z$ is a zero-dimensional apolar scheme of length $r$ then the coordinate ring $\Bbbk[Z] = T/I(Z)$
is one-dimensional with Hilbert polynomial $r$, and has non-decreasing Hilbert function,
so $\rank C^a_F = \dim T/F^\perp \leq \dim T/I(Z) \leq r$.)

\begin{remark}
The following inequalities hold:
\begin{gather*}
  \rank C^a_F \leq c\r(F) \leq s\r(F) \leq \r(F), \\
  \rank C^a_F \leq b\r(F) \leq s\r(F) \leq \r(F).
\end{gather*}
For examples with $cr(F) < br(F)$, see for instance \cite{MR2996880},
where it is shown that a general cubic form in $n+1$ variables has cactus rank at most $2n+2$,
while by the Alexander--Hirschowitz theorem it has border rank
\[
  \left\lceil \frac{1}{n+1} \binom{n+3}{3} \right\rceil .
\]
The reverse inequality is also possible.
For an example, see \cite{Buczynska:2013zh},
where for the particular (not general) cubic in five variables $F = x_0^2 y_0 - (x_0 + x_1)^2 y_1 + x_1^2 y_2$
it is shown that $br(F) = 5$, $cr(F) = 6$.
\end{remark}

The colorful name ``cactus rank'' was coined in \cite{MR2842085}, following \cite{MR3121848}.
However earlier terminology in \cite[Definition~5.1, Definition~5.66]{MR1735271} is as follows:
\defining{annihilating scheme} for apolar scheme to $F$,
\defining{scheme length} for cactus rank,
\defining{smoothable scheme length} for smoothable rank,
and \defining{length} for border rank.

In the next section we will review and generalize a lower bound for cactus rank
discovered by Ranestad and Schreyer.

\subsection{Simultaneous Waring rank}

\begin{defn}
Let $W \subseteq S_d$.
Then
\[
  W^\perp = \{ \Theta \in T : \Theta F = 0 \text{ for all $F \in W$} \}
    = \bigcap_{F \in W} F^\perp.
\]
The apolar algebra is $A^W = T / W^\perp$.
\end{defn}

Note that each graded piece $W^\perp_k$ is the kernel of the catalecticant $C_W^{(0,k)}$.
The apolar algebra $A^W$ is a level Artin algebra; these have been studied in for example \cite{MR2292384}.

Say a scheme $Z \subset \PP V$ is \defining{apolar to $W \subseteq S^d V$} if $I = I(Z) \subseteq W^\perp$.
Then $\r(W)$ is the least length of a reduced zero-dimensional apolar scheme.
As before we define the smoothable rank $s\r(W)$ and cactus rank $c\r(W)$:
$s\r(W)$ is the least length of a smoothable zero-dimensional apolar scheme
and $c\r(W)$ is the least length of a zero-dimensional apolar scheme.

Here is an apolarity lemma for linear series:
\begin{thm}
Let $W \subseteq S_d$.
Let $\ell_1,\dotsc,\ell_r \in V$ and $I = I([\ell_1],\dotsc,[\ell_r])$.
Then $W \subseteq \Span \{ \ell_1^d, \dotsc, \ell_r^d \}$ if and only if $I \subset W^\perp$.

More generally let $Z \subset \PP V$ be a scheme with saturated homogeneous defining ideal $I = I(Z)$
and let $\nu_d : \PP V \to \PP S^d V$ be the degree $d$ Veronese map.
Then $\PP W \subseteq \Span(\nu_d(Z))$ if and only if $I \subset W^\perp$.
\end{thm}
\begin{proof}
$W \subseteq \Span \{ \ell_1^d, \dotsc, \ell_r^d \}$ if and only if for each $F \in W$,
there are constants $c_1,\dotsc,c_r$ such that $F = c_1 \ell_1^d + \dotsb + c_r \ell_r^d$,
if and only if for each $F \in W$, $I \subset F^\perp$ (by the usual Apolarity Lemma!),
if and only $I \subseteq \bigcap_{F \in W} F^\perp = W^\perp$.

The proof of the scheme version is the same as before.
\end{proof}

\begin{example}\label{example: matrix linear series apolar ideals}
Consider the linear series $R_k$, $D_k$, $P_k$ as in Example~\ref{example: matrix linear series},
in variables $x_{i,j}$ for $1 \leq i \leq m$, $1 \leq j \leq n$.
Let the dual variables be $\d_{i,j}$; let $X^*$ be the $m \times n$ matrix with entries $\d_{i,j}$.
It is easy to see that
\[
  R_k^\perp = \langle \d_{i,j_1}\d_{i,j_2}, \d_{i_1,j}\d_{i_2,j} \mid 1 \leq i,i_1,i_2 \leq m, 1 \leq j,j_1,j_2 \leq n \rangle,
\]
that is, the ideal generated by products of two $\d_{i,j}$ from either the same row or the same column of $X^*$.
It is also easy to see that $D_k$ is annihilated by these products together with permanents of $2 \times 2$ submatrices of $X^*$,
while $P_k$ is annihilated by the same products together with $2 \times 2$ minors of $X^*$.
By Shafiei's results \cite{Shafiei:ud}, these actually generate the apolar ideals:
\[
  D_k^\perp = R_k^\perp + P_2^*,
  \qquad
  P_k^\perp = R_k^\perp + D_2^*,
\]
where $P_2^*$ is the ideal generated by permanents of $2 \times 2$ submatrices of $X^*$,
\[
  P_2^* = \langle \d_{i,j} \d_{k,\ell} + \d_{i,\ell} \d_{k,j} \rangle,
\]
and $D_2^*$ is the ideal generated by $2 \times 2$ minors of $X^*$,
\[
  D_2^* = \langle \d_{i,j} \d_{k,\ell} - \d_{i,\ell} \d_{k,j} \rangle.
\]
\end{example}

\subsection{Multihomogeneous polynomial}

\begin{defn}
Let $S = \Bbbk[V_1 \oplus \dotsb \oplus V_s]$
and $T = \Bbbk[V_1^* \oplus \dotsb \oplus V_s^*]$, multigraded rings.
For $M \in S_{\bd} = S^{d_1} V_1 \otimes \dotsm \otimes S^{d_s} V_s$,
\[
  M^\perp = \{ \Theta \in T : \Theta M = 0 \} .
\]
The apolar algebra is $A^M = T / M^\perp$.
\end{defn}
The apolar algebra $A^M$ is a multigraded Artin algebra.
Each multigraded piece $(M^\perp)_{\bk}$ is the kernel of the catalecticant $C^{\bk}_M$.

A scheme $Z \subset \PP V = \prod \PP V_i$ is \defining{apolar to $M$} if $I = I(Z) \subseteq M^\perp$.
The smoothable rank and cactus rank are defined as before.

We will state an apolarity lemma for multihomogeneous polynomials
using the following notation.
For a point $P = (\ell_1,\dotsc,\ell_s) \in V_1 \times \dotsm \times V_s$, with each $\ell_j \neq 0$,
we denote $[P] = ([\ell_1],\dotsc,[\ell_s]) \in \PP V_1 \times \dotsm \times \PP V_s$.
For each $\bd = (d_1,\dotsc,d_s)$ let $\nu_{\bd} : \prod \PP V_i \to \prod \PP S^{d_i} V_i$ be the Segre-Veronese map,
$\nu_{\bd}([\ell_1],\dotsc,[\ell_s]) = [\ell_1^{d_1} \dotsm \ell_s^{d_s}]$.
For convenience let us denote $P^{\bd} = \ell_1^{d_1} \dotsm \ell_s^{d_s}$.
Note that the coordinate ring of $\PP V = \PP V_1 \times \dotsm \times \PP V_s$ is $\Bbbk[\PP V] = \Bbbk[V^*] = T$.

\begin{thm}
Let $M$ be a multihomogeneous polynomial of multidegree $\bd = (d_1,\dotsc,d_s)$.
For $i = 1,\dotsc,r$, let $P_i = (\ell_{i,1},\dotsc,\ell_{i,s})$ where each $\ell_{i,j} \in V_j$, $\ell_{i,j} \neq 0$.
Let $I = I([P_1],\dotsc,[P_r]) \subset T$.
Then $M \in \Span(P_1^{\bd},\dotsc,P_r^{\bd})$ if and only if $I \subset M^\perp$.

More generally let $Z \subset \PP V = \prod \PP V_i$ be a scheme with defining ideal $I = I(Z)$.
Then $[M]$ is in the linear span of $\nu_{\bd}(Z)$ if and only if $I \subset M^\perp$.
\end{thm}
\begin{proof}
If $\Theta \in I_{\bk}$ then $\Theta([P_1]) = \dotsb = \Theta([P_r]) = 0$ so, as a differential operator,
$\Theta(P_1^{\bd}) = \dotsb = \Theta(P_r^{\bd}) = 0$, hence $\Theta(M) = 0$ and $\Theta \in M^\perp$.
Conversely, if $I \subset M^\perp$ then in particular $I_{\bd} \subseteq M^\perp_{\bd}$.
For $\Theta \in T_{\bd}$, $\Theta \in I_{\bd}$ if and only if as a polynomial $\Theta(P_i) = 0$ for each $i$,
if and only if as a differential operator $\Theta(P_i^{\bd}) = 0$ for each $i$,
if and only if, in the pairing between the dual spaces $S_{\bd}$ and $T_{\bd}$, $\Theta \in (P_i^{\bd})^\perp$ for each $i$.
So $\bigcap_{i=1}^r (P_i^{\bd})^\perp = I_{\bd} \subseteq M^\perp_{\bd}$.
Transposing yields $\Span\{M\} \subseteq (I_{\bd})^\perp = (\bigcap (P_i^{\bd})^\perp)^\perp = \Span\{P_1^{\bd},\dotsc,P_s^{\bd}\}$.

The proof of the scheme version is the same as before.
\end{proof}

\begin{remark}\label{remark: maciej galazka}
Work in progress by Maciej Ga{\l}{\c a}zka
shows that similar statements hold on more general toric varieties.
\end{remark}

\subsection{Generalized rank}

I do not know a statement in the full generality of Section \ref{section: generalized rank}.
Here is a statement essentially due to Carlini.

\begin{thm}[{\cite{MR2202247}, \cite{MR2184818}}]\label{thm: apolarity lemma carlini}
Let $F \in \Bbbk[V]$ be a homogeneous form of degree $d$ in $n = \dim V$ variables.
Let $W_1,\dotsc,W_s \subset V$ be linear subspaces and let $I = I(W_1 \cup \dotsb \cup W_s)$
be the homogeneous defining ideal of the reduced union of the $W_i$.
Then the following are equivalent:
\begin{enumerate}
\item There exist $G_i \in \Bbbk[W_i]$, $i=1,\dotsc,s$, such that $F \in \Span\{G_1,\dotsc,G_s\}$.
\item $I \subset F^\perp$.
\end{enumerate}
\end{thm}
In particular the least number of terms in a ``codimension one'' decomposition of $F$,
i.e., a decomposition as a sum of forms in $n-1$ variables
(see Example~\ref{example: carlini codimension one decompositions}),
is equal to the least number of hyperplanes whose union is apolar to $F$,
equivalently, the least degree of a form in $F^\perp$ that factors as a product of distinct linear factors.
Similarly, the least number of terms in a decomposition of $F$ as a sum of binary forms
(see Example~\ref{example: carlini binary decompositions})
is equal to the least number of projective lines whose union is apolar to $F$.
These are the results stated by Carlini, but his proofs actually give the full theorem above.
For the convenience of the reader we give here the proof, following Carlini's ideas.
\begin{proof}
First suppose $F \in \Span\{G_1,\dotsc,G_s\}$.
For each $i$ fix a power sum decomposition of $G_i$, with terms corresponding to points in $W_i$.
Let $J$ be the homogeneous defining ideal of all the (projective) points, for all the terms that arise for all of the $G_i$.
Since $F$ is a linear combination of the $G_i$, these terms also give a power sum decomposition of $F$.
By the classical Apolarity Lemma, $J \subset F^\perp$.
And since each of the points defined by $J$ lies in one of the $W_i$, $I \subset J$.
So $I \subset F^\perp$.

Conversely suppose $I \subset F^\perp$.
For each $i$ choose enough points $\ell_{i,1},\dotsc,\ell_{i,k_i} \in W_i$ so that the $\ell_{i,j}^d$ span
the space of $d$-forms on $\PP W_i^*$, that is, $\Bbbk[W_i]_d$
(we may take $k_i = \binom{\dim W_i + d - 1}{d}$ and choose the $\ell_{i,j}$ generally in $W_i$).
Let $J$ be the defining ideal of all the points $\ell_{i,j}$.
We claim that $J_d \subset I_d$.
If $\Theta \in T_d$ is any dual $d$-form vanishing at each point $\ell_{i,j}$ then, as a differential operator,
$\Theta$ annihilates each $\ell_{i,j}^d$.
Since for each $i$ these span all the $d$-forms on $\PP W_i$, then in fact $\Theta$ annihilates all of $\Bbbk[W_i]_d$, for each $i$.
In particular $\Theta$ annihilates $\ell^d$ for every $\ell \in W_i$, for each $i$.
Hence, returning to considering $\Theta$ as a polynomial, $\Theta$ vanishes at each point $[\ell] \in \PP W_i$, for each $i$.
This means $\Theta \in I_d$.

Now from $J_d \subset I_d \subset F^\perp$ we claim $J \subset F^\perp$.
For degrees $k > d$, $F^\perp_k = T_k$, so $J_k \subset F^\perp_k$.
For degrees $k < d$, $\Theta \in F^\perp_k$ if and only if $\Theta F = 0$,
if and only if for every $\Psi \in T_{d-k}$, $\Psi \Theta F = 0$,
if and only if for every $\Psi \in T_{d-k}$, $\Psi \Theta \in F^\perp_d$.
Meanwhile $\Theta \in J_k$ if and only if $\Theta([\ell_{i,j}]) = 0$ for all $i,j$,
if and only if for every $\Psi \in T_{d-k}$, $\Psi \Theta$ vanishes at each $[\ell_{i,j}]$,
if and only if for every $\Psi \in T_{d-k}$, $\Psi \Theta \in J_d$.
So if $\Theta \in J_k$ then for every $\Psi \in T_{d-k}$ we have $\Psi \Theta \in J_d \subset F^\perp_d$,
which implies that $\Theta \in F^\perp_k$.
Therefore $J_k \subset F^\perp_k$, as desired.
This proves the claim that $J \subset F^\perp$.
Then by the classical Apolarity Lemma, $F$ is in the span of the $\ell_{i,j}^d$, say $F = \sum_i \sum_j c_{i,j} \ell_{i,j}^d$.
For each $i$, set $G_i = \sum_j c_{i,j} \ell_{i,j}^d$.
Then $G_i \in \Bbbk[W_i]_d$ and $F \in \Span\{G_1,\dotsc,G_s\}$.
\end{proof}

The classical Apolarity Lemma is precisely the case $\dim W_1 = \dotsb = \dim W_s = 1$.

Note that in contrast to the previous statements, a subideal $I \subset F^\perp$ does not necessarily
determine the decomposition; rather it only determines the subspaces $W_1,\dotsc,W_s$ over which the decomposition occurs.

It would be interesting to find out if the above theorem,
or at least the numerical corollaries discussed before the proof,
are related to the catalecticants arising from
the sheaves mentioned in Example~\ref{example: subspace variety sheaf}.

\begin{remark}
It would be very interesting to have similar statements more generally,
or at least for natural situations such as
split rank (decomposition as a sum of products of linear forms, see Example~\ref{example: split rank})
and, if $d=kt$, then decomposition as a sum of $k$th powers of $t$-forms (see Example~\ref{example: kth powers of t-forms}).
\end{remark}

\section{Ranestad--Schreyer bounds}\label{section: ranestad-schreyer bounds}

Ranestad and Schreyer gave an extremely elegant lower bound for cactus and Waring rank in \cite{MR2842085}.
We recall their result, then generalize it.

\subsection{Classical Waring rank}

\begin{theorem}[{\cite{MR2842085}}]\label{thm: ranestad-schreyer}
Let $F \in S^d V$.
Let $\delta$ be a positive integer such that the homogeneous ideal $F^\perp$
is generated in degrees less than or equal to $\delta$.
Then the Waring rank $\r(F)$, smoothable rank $s\r(F)$, and cactus rank $c\r(F)$ satisfy
\[
  \r(F) \geq s\r(F) \geq c\r(F) \geq \frac{\ell(A^F)}{\delta} ,
\]
where $\ell(A^F)$ is the length of the apolar algebra $A^F$.
\end{theorem}
For the reader's convenience we include the proof given by Ranestad and Schreyer.
\begin{proof}
$\r(F) \geq s\r(F) \geq c\r(F)$ is obvious.
Let $Z \subset \PP V$ be a zero-dimensional apolar scheme of length $r$.
Let $I = I(Z)$ and let $\widehat{Z}$ be the affine variety defined by $I$.
The definition of $\delta$ means that the common zero locus in $\PP V$ of the linear series $F^\perp_{\delta}$
is exactly equal to the scheme defined by the ideal $F^\perp$; namely, the empty scheme.
So in affine space $V$, the common zero locus of $F^\perp_{\delta}$ is (a scheme supported at) the origin.
Thus a general $\delta$-form $G \in F^\perp_\delta$ does not vanish
at any of the points in the support of $Z$.
Then the affine hypersurface $V(G)$ has proper intersection with $\widehat{Z}$.
Since $G \in F^\perp$, $\Spec A^F \subset V(G)$; and since $Z$ is apolar to $F$, $\Spec A^F \subset \widehat{Z}$.
So $\Spec A^F \subseteq V(G) \cap \widehat{Z}$.
By B\'ezout's theorem, then, $\ell(A^F) = \ell(\Spec A^F) \leq \deg(V(G)) \deg(\widehat{Z}) = \delta r$.
\end{proof}
See \cite{MR2842085} for an application of this to monomials.
(And see \cite{Carlini20125}, \cite{MR3017012} for more about apolarity of monomials.)
See \cite{Shafiei:ud}, \cite{Shafiei:2013fk} for applications of this to determinants, permanents, Pfaffians, etc.
See \cite{Teitler:2013uq} for an application of this to reflection arrangements.
See \cite{Buczynska:2013kx} for a closer examination of this bound, particularly the quantity $\delta$.

In fact this proof shows:
\begin{corollary}\label{cor: ranestad-schreyer slight improvement}
Let $F \in S^d V$.
Let $\epsilon$ be a positive integer such that the homogeneous ideal generated by $F^\perp_{\leq \epsilon}$
defines a zero-dimensional affine variety, i.e., its common zero locus consists only of the origin in affine space.
Then $c\r(F) \geq \ell(A^F) / \epsilon$.
\end{corollary}
Of course if $F^\perp$ is generated in degrees less than or equal to $\delta$ then $F^\perp_{\leq \delta} = F^\perp$
whose common zero locus is just the origin because it is a homogeneous Artinian ideal.

\begin{corollary}
Let $F \in S^d V$.
Suppose $F^\perp$ is a complete intersection generated in degrees $0 < d_1 \leq \dotsb \leq d_n$.
Then $s\r(F) = c\r(F) = d_1 \dotsm d_{n-1}$.
\end{corollary}
\begin{proof}
$\ell(A^F) = d_1 \dotsm d_n$ so $s\r(F) \geq c\r(F) \geq d_1 \dotsm d_{n-1}$.
Let $F^\perp = \langle G_1,\dotsc,G_n \rangle$, $\deg G_i = d_i$.
Then $\langle G_1, \dotsc, G_{n-1} \rangle$ is a one-dimensional complete intersection,
so it is smoothable (as every complete intersection is).
Hence $s\r(F) \leq \deg(T/\langle G_1,\dotsc,G_{n-1} \rangle) = d_1 \dotsm d_{n-1}$.
\end{proof}
This occurs for monomials \cite{MR2842085} and reflection arrangements \cite{Teitler:2013uq}.

\begin{remark}\label{remark: bounds for general forms}
For a general $d$-form $F$, the apolar ideal is generated entirely in degree $\delta = (d+2)/2$ if $d$ is even;
if $d$ is odd, the apolar ideal is either generated in the two degrees $\lfloor (d+2)/2 \rfloor$, $\delta = \lceil (d+2)/2 \rceil$
or entirely in the degree $\delta = \lfloor (d+2)/2 \rfloor$.
(It is easy to see that $F^\perp$ must have a generator of degree at most $(d+2)/2$.
One can show that $F^\perp$ has no generators in degrees strictly less than $\lfloor (d+2)/2 \rfloor$
using results on compressed algebras, see Definition 3.11 and Proposition 3.12 of \cite{MR1735271}.
What takes more work is to show that $F^\perp$ is generated in degrees less than or equal to $\lceil (d+2)/2 \rceil$.
If $d$ is even this follows from Proposition 4.1B (pg. 362) and Example 4.7 of \cite{MR748843}.
For $d$ odd it follows from unpublished notes of Uwe Nagel [personal communication].)

One can then see that the catalecticant lower bound is better than the Ranestad--Schreyer lower bound,
and Theorem~\ref{thm: single homogeneous polynomial} provides no improvement over the catalecticant lower bound.
However the actual rank is known by the Alexander--Hirschowitz theorem and it is strictly greater than the catalecticant lower bound.
\end{remark}

Theorem~\ref{thm: ranestad-schreyer} gives a lower bound for smoothable rank and cactus rank,
which may be strictly smaller than Waring rank.
In contrast, the lower bound in Theorem~\ref{thm: single homogeneous polynomial}
is in fact a lower bound for Waring rank $\r(F)$, as opposed to cactus rank or smoothable rank.
Here is an example to illustrate the difference.

\begin{example}
Consider $F = xy^{d-1}$ with $d \geq 3$.
This has cactus rank and smoothable rank equal to $c\r(F) = s\r(F) = 2$ \cite{MR2842085}.
Note $F$ is concise (since it is not a pure power), $\rank C^{d-2}_F = 2$, and $\Sigma_2(F) = V(y) \neq \varnothing$.
Then Theorem~\ref{thm: single homogeneous polynomial} gives the lower bound
\[
  \r(xy^{d-1}) \geq \rank C^{d-2}_F + \dim \widehat{\Sigma}_2(F) = 2 + 1 = 3 .
\]
Of course this bound is far from sharp: we know that $\r(xy^{d-1}) = d$.
But this simple example shows that Theorem~\ref{thm: single homogeneous polynomial} is not generally a lower bound
for cactus rank or smoothable rank.
\end{example}

In this sense, the lower bound of Theorem~\ref{thm: single homogeneous polynomial} has the potential to be better:
it is better than the bound of Theorem~\ref{thm: ranestad-schreyer} for any form $F$ such that
\[
  \frac{\ell(A^F)}{\delta} \leq c\r(F) < \rank C^a_F + \dim \widehat{\Sigma}_a(F) \leq \r(F) .
\]

Nevertheless in practice Theorem~\ref{thm: ranestad-schreyer} seems to be quite a bit better than
Theorem~\ref{thm: single homogeneous polynomial} for many examples of interest.

\begin{example}\label{example: ranestad-schreyer-shafiei determinant bound}
We have $\ell(A^{\det_n}) = \sum \binom{n}{k}^2 = \binom{2n}{n}$
and $\det_n^\perp$ is generated by quadrics \cite{Shafiei:ud}.
Hence $\r(\det_n) \geq \frac{1}{2} \binom{2n}{n}$.
For the sake of discussion we call this the Ranestad--Schreyer--Shafiei bound for $\r(\det_n)$.

For $F = \det_3$ and $F = \det_4$ the bound of Theorem~\ref{thm: single homogeneous polynomial}
is better than the Ranestad--Schreyer--Shafiei bound,
but for $n \geq 5$, the Ranestad--Schreyer--Shafiei bound for $\r(\det_n)$ is better than
the bound of Theorem~\ref{thm: single homogeneous polynomial}.
(See Example~\ref{example: determinant}.)
Incidentally, I do not know the cactus or smoothable ranks of $\det_n$ for any $n>2$.
\end{example}

\begin{remark}\label{remark: reducedness used in proof}
This discussion raises the question, why does Theorem~\ref{thm: single homogeneous polynomial}
give a lower bound for Waring rank rather than smoothable or cactus rank?
In other words, in giving a lower bound for the length of a reduced apolar scheme,
where is the reducedness used in the proof of the theorem?

It is used in the step that concludes,
if $h^{d-a}$ vanished at $[\ell_1], \dotsc, [\ell_r]$, then $h$ would also vanish at each $[\ell_i]$.
This is precisely the statement that $I = I(\{[\ell_1],\dotsc,[\ell_r]\})$ is radical, i.e., the scheme is reduced.
\end{remark}

\begin{remark}
Until recently, Theorem~\ref{thm: single homogeneous polynomial} was the only general result
known to me that gives a lower bound for the Waring rank of an arbitrary form,
as opposed to a lower bound for the cactus rank or border rank.

Another approach to lower bounds for Waring rank itself
comes from the technique used to find Waring ranks of monomials and sums of pairwise coprime monomials in \cite{Carlini20125}.
Very recently this method has been developed further in \cite{Carlini:2014rr} and \cite{Woo:2014ix}.
So far it has been used to find Waring ranks of increasingly general examples,
namely, sums of polynomials in small numbers of linearly independent variables.
The remarkable momentum of recent developments makes me optimistic that this approach will yield
more results in the future.
\end{remark}

\subsection{Simultaneous Waring rank}

\begin{thm}
Let $W \subseteq S^d V$ be a linear series.
Suppose $W^\perp$ is generated in degrees less than or equal to $\delta$.
Then $\r(W) \geq s\r(W) \geq c\r(W) \geq \ell(A^W) / \delta$.
\end{thm}
The proof is the same.

\begin{example}
Let $X = (x_{i,j})$, $1 \leq i \leq m$, $1 \leq j \leq n$,
and consider the linear series $D_k$, $P_k$, $R_k$.
First we compute the lengths of the apolar algebras.
We have
\[
  \ell(A^{D_k}) = \ell(A^{P_k}) = \sum_{0 \leq a \leq k} \binom{m}{a} \binom{n}{a} .
\]
When $k = m \leq n$ this simplifies to
\[
  \sum_{0 \leq a \leq m} \binom{m}{m-a} \binom{n}{a} = \binom{m+n}{m} .
\]
And
\[
  \ell(A^{R_k}) = \sum_{0 \leq a \leq k} \binom{m}{a} \binom{n}{a} a! .
\]

Shafiei showed that $(\det_k)^\perp$ and $(\per_k)^\perp$ are each generated by quadrics \cite{Shafiei:ud}.
Using her result, it is easy to show that $D_k^\perp$ and $P_k^\perp$ are generated by quadrics,
see Example~\ref{example: matrix linear series apolar ideals}.
Therefore
\[
  \r(D_k), \r(P_k) \geq \frac{1}{2} \sum_{0 \leq a \leq k} \binom{m}{a} \binom{n}{a},
\]
and in particular when $k = m \leq n$,
\[
  \r(D_m), \r(P_m) \geq \frac{1}{2} \binom{m+n}{m} .
\]
The same lower bound holds for cactus and smoothable ranks.

It is also easy to see that $R_k^\perp$ is generated by quadrics,
again see Example~\ref{example: matrix linear series apolar ideals}.
So
\[
  \r(R_k) \geq \frac{1}{2} \sum_{0 \leq a \leq k} \binom{m}{a} \binom{n}{a} a! .
\]
\end{example}

\subsection{Multihomogeneous polynomials}\label{section: multihomogeneous ranestad-schreyer}

In order to give a generalization of the Ranestad--Schreyer Theorem (\ref{thm: ranestad-schreyer})
for multihomogeneous polynomials, we will use some elementary facts about finite schemes in multiprojective spaces.
While these facts are surely well-known to experts, it was surprisingly difficult to find a comprehensive reference.
I am grateful to Adam Van Tuyl for suggesting some (nearly-comprehensive) references, including
his dissertation \cite{Tuyl:2001cr},
the dissertation of Lavila-Vidal \cite{Lavila-Vidal:1999sp},
and \cite{MR1686450}.

Here is a brief list of the particular facts we need.
\begin{lemma}
If $Z \subset \prod_{i=1}^s \PP V_i$ is a zero-dimensional scheme with $I = I(Z)$,
then $V(I) \subset \prod_{i=1}^s V_i$ is $s$-dimensional
and $\deg V(I) = \ell(Z)$,
where degree is computed with respect to the grading by total degree.
\end{lemma}
\begin{proof}
The dimension statement is precisely \cite[Prop.~2.2.9]{Tuyl:2001cr} or \cite[Lemma~1.4.1]{Lavila-Vidal:1999sp}.

Say $\ell(Z) = r$.
This means that for every multidegree $\bdelta$,
$I(Z)_{\bdelta}$ has codimension at most $r$ in $S_{\bdelta}$, where $S = \Bbbk[V_1 \oplus \dotsb \oplus V_s]$;
and there is a multidegree $\bdelta_0$ such that for every $\bdelta \geq \bdelta_0$,
$I(Z)_{\bdelta}$ has codimension equal to $r$ in $S_{\bdelta}$.
Now in the grading by total degree, for $t \gg 0$,
\[
\begin{split}
  \dim (S/I)_t &= \sum_{|\bdelta|=t} \dim(S_{\bdelta}/I(Z)_{\bdelta}) \\
    &= \sum_{\substack{|\bdelta|=t \\ \bdelta \geq \bdelta_0}} r
      + \sum_{\substack{|\bdelta|=t \\ \bdelta \not\geq \bdelta_0}} \dim(S_{\bdelta}/I(Z)_{\bdelta}) \\
    &= r \cdot \#\{\bdelta' \geq 0 : |\bdelta'| = t - |\bdelta_0| \}
      + \sum_{\substack{|\bdelta|=t \\ \bdelta \not\geq \bdelta_0}} \dim(S_{\bdelta}/I(Z)_{\bdelta}) \\
    &= r \cdot \binom{t - |\bdelta_0| + s-1}{s-1}
      + \sum_{\substack{|\bdelta|=t \\ \bdelta \not\geq \bdelta_0}} \dim(S_{\bdelta}/I(Z)_{\bdelta}) .
\end{split}
\]
Note
\[
\begin{split}
  \#\{\bdelta : |\bdelta|=t, \bdelta \not\geq \bdelta_0\} &= \binom{t+s-1}{s-1} - \binom{t-|\bdelta_0|+s-1}{s-1} \\
    &= \binom{t-|\bdelta_0|+s-1}{s-2} + \dotsb + \binom{t-1+s-1}{s-2},
\end{split}
\]
a polynomial $q(t)$ of degree $s-2$.
Therefore
\[
  0 \leq \sum_{\substack{|\bdelta|=t \\ \bdelta \not\geq \bdelta_0}} \dim(S_{\bdelta}/I(Z)_{\bdelta})
    \leq r \cdot q(t),
\]
so this sum is bounded by a polynomial of degree $s-2$; for sufficiently large $t$ it is equal to a polynomial
of degree at most $s-2$.
Hence for $t \gg 0$, $\dim (S/I)_t = P(t)$ is a polynomial in $t$ of degree $s-1$ with leading coefficient $r/(s-1)!$.

This $P(t)$ is precisely the Hilbert polynomial of $\PP V(I) \subset \PP(V_1 \oplus \dotsb \oplus V_s)$.
Hence $\deg V(I) = r = \ell(Z)$ (and we have re-proved the dimension statement).
\end{proof}

\begin{lemma}
Let $Z \subset \prod_{i=1}^s \PP V_i$ be a zero-dimensional scheme with $I = I(Z)$.
For each $i$ let $\hat{V}_i = V_1 \times \dotsb \times \{0\} \times \dotsb \times V_s$,
the product of all the $V_j$ with $j \neq i$.
Also for each $i$ let $\pr_i : \prod \PP V_j \to \prod_{j \neq i} \PP V_j$
be the projection onto all the factors except $V_i$, and let $Z_i = \pr_i(Z)$.
Let $I_i = I(Z_i)$.

Then in $\prod V_i$ we have $V(I) \cap \hat{V}_i = V(I_i)$,
an $(s-1)$-dimensional affine variety.
\end{lemma}
The reader may easily verify this.

\begin{theorem}\label{thm: multihomogeneous polynomial ranestad-schreyer}
Let $M \in S^{\bd} V$.
Let $\bdelta = (\delta_1,\dotsc,\delta_s)$ be such that $M^\perp$ is generated in multidegrees less than or equal to $\bdelta$.
Then $\r_{MH}(M) \geq s\r_{MH}(M) \geq c\r_{MH}(M) \geq \ell(A^M)/(\prod \delta_i)$.
\end{theorem}
\begin{proof}
Let $Z \subset \prod \PP V_i$ be a closed zero-dimensional scheme of length $r$
and suppose that $I = I(Z) \subset M^\perp$.
The scheme $\Spec A^M$ naturally lies as a closed subscheme in $\Spec T \cong V = \prod V_i$.
Let $\widehat{Z} = V(I) \subset V$.
Note that $\Spec A^M \subset \widehat{Z}$.

For each $i=1,\dotsc,s$ let $S_i = \Bbbk[\PP V_i] \subset \Bbbk[\PP V] = S$.
The definition of $\bdelta$ means that $M^\perp \cap S_i$ is an Artinian ideal generated in degrees less than or equal to $\delta_i$.
For each $i$ let $G_i$ be a general element of $M^\perp_{(0,\dotsc,\delta_i,\dotsc,0)}$.
Since $M^\perp_{(0,\dotsc,\delta_i,\dotsc,0)}$ has no basepoints, $G_i$ does not vanish at any of the points of $Z$.
Let $B = V(G_1,\dotsc,G_s) \subset V$, a complete intersection of codimension $s$ and degree $\prod \delta_i$.

Let $\hat{V}_i = V_1 \times \dotsm \times \{0\} \times \dotsm \times V_s$, for $1 \leq i \leq s$.
Outside of the union $\bigcup \hat{V}_i$, $B$ is disjoint from $\widehat{Z}$, since none of the $G_i$ vanish at any point of $Z$.
On each $\hat{V}_i$, $G_i$ vanishes identically, but $G_j$ is not identically vanishing for $j \neq i$.
So $B \cap \hat{V}_i$ is a complete intersection of codimension $s-1$ in $\hat{V}_i$.
By the generality of the $G$'s, $B \cap \hat{V}_i$ does not meet $\widehat{Z} \cap \hat{V}_i$ outside the origin.

Therefore $B \cap \widehat{Z}$ is supported only at the origin.
So
\[
  \ell(A^M) = \ell(\Spec A^M) \leq \ell(B \cap \widehat{Z})
    = \deg(B) \deg(\widehat{Z}) = \Big(\prod \delta_i \Big) r,
\]
which completes the proof.
\end{proof}

\begin{example}\label{example: bihomogeneous product ranestad schreyer}
The bihomogeneous form $F = x_1 \dotsm x_a y_1 \dotsm y_b$
has apolar algebra of length $2^{a+b}$.
Let $\d_1,\dotsm,\d_a$ be the dual variables to the $x_i$ and let $\epsilon_1,\dotsc,\epsilon_b$ be the dual variables to the $y_j$.
Then $F^\perp = \langle \d_1^2,\dotsc,\d_a^2,\epsilon_1^2,\dotsc,\epsilon_b^2 \rangle$.
So $F^\perp$ is generated in bidegrees $(2,0)$ and $(0,2)$.
Set $\bdelta = (2,2)$.
Then
\[
  \r_{MH}(x_1 \dotsm x_a y_1 \dotsm y_b) \geq \frac{2^{a+b}}{4} = 2^{a+b-2} .
\]
Therefore the rank is in fact equal to $2^{a+b-2}$.
This answers the question raised in Example~\ref{example: bihomogeneous product}.
\end{example}

\begin{example}
More generally let
\[
  F = (\prod_{j=1}^{n_1} x_{1,j}^{d_1}) \dotsm (\prod_{j=1}^{n_s} x_{s,j}^{d_s}) ,
\]
a multihomogeneous form of multidegree $(n_1 d_1,\dotsc,n_s d_s)$.
Then
\[
  \r_{MH}(F) \leq \prod_{i=1}^s \r((x_{i,1} \dotsm x_{i,n_i})^{d_i}) = \prod_{i=1}^s (d_i+1)^{n_i-1}.
\]
On the other hand $F^\perp = \langle \d_{1,1}^{d_i+1}, \dotsc, \d_{s,n_s}^{d_s+1} \rangle$.
Then $\ell(A^F) = \prod_{i=1}^s (d_i+1)^{n_i}$
and, with $\bdelta = (d_1+1,\dotsc,d_s+1)$, we get $\r_{MH}(F) \geq \prod_{i=1}^s (d_i+1)^{n_i-1}$,
which exactly determines $\r_{MH}(F)$.
\end{example}

\begin{example}\label{example: matrix multiplication tensor}
In this example we consider the tensor corresponding to matrix multiplication.
This has been intensively studied, see for example \cite[Chap.~11]{MR2865915}, \cite{Landsberg:2014yq}.
Our results are not able to improve or even match
the bounds obtained by specialized methods.
Nevertheless, we consider the example in order to illustrate our results
and because of its intrinsic importance.
Also, we work out the apolar ideal of the matrix multiplication tensor.

Fix $n$ and let $A = (a_{i,j})$, $B = (b_{i,j})$, and $C = (c_{i,j})$ be generic $n \times n$ matrices in separate variables.
Let the dual variables be $\alpha_{i,j}$, $\beta_{i,j}$, and $\gamma_{i,j}$, for $1 \leq i, j \leq n$.
Let $\cA \cong \cB \cong \cC \cong \Bbbk^{n^2}$ be vector spaces of dimension $n^2$,
where $\cA$ has basis $\{a_{i,j}\}$, and so on.
Let $S = \Bbbk[A,B,C] = \Bbbk[\alpha_{1,1},\dotsc,\gamma_{n,n}]$,
the multigraded polynomial ring with each $a_{i,j}$ having multidegree $(1,0,0)$,
each $b_{i,j}$ having multidegree $(0,1,0)$, and each $c_{i,j}$ having multidegree $(0,0,1)$;
and let $T = \Bbbk[\alpha_{1,1},\dotsc,\gamma_{n,n}]$, with similar multigrading.
Let
\[
  \mult_n = \Tr(ABC),
\]
a multihomogeneous form of multidegree $(1,1,1)$ --- that is, a tensor.
This is the tensor representing matrix multiplication,
\[
  \mult_n \in \cA \otimes \cB \otimes \cC \cong \Hom(\cA \otimes \cB, \cC^*).
\]
Indeed,
\[
  \mult_n = \sum_{i=1}^n (ABC)_{i,i} = \sum_{i=1}^n \sum_{j=1}^n \sum_{k=1}^n a_{i,j} b_{j,k} c_{k,i},
\]
so that $\gamma_{k,i}M = (AB)_{i,k}$.
That is, as an element of $\Hom(\cA \otimes \cB, \cC^*)$,
$\mult_n$ takes pairs of matrices $(A,B)$ to their transposed product, $(AB)^t$.
The transposition simply corresponds to the dualization of the space $\cC^*$.
(More generally, one may consider rectangular matrix multiplication,
see for example \cite{Landsberg:2014yq}.)

Let $I_a = \langle \alpha_{1,1},\dotsc,\alpha_{n,n} \rangle$
and similarly let $I_b$ and $I_c$ be the ideals generated by the $\beta_{i,j}$ and $\gamma_{i,j}$, respectively;
we set $I = I_a^2 + I_b^2 + I_c^2$.
Let
\[
  I_{ab} = \langle \alpha_{i,j_1} \beta_{j_2,k} \mid 1 \leq i,j_1,j_2,k \leq n, j_1 \neq j_2 \rangle
\]
and similarly $I_{bc}, I_{ca}$; set $I' = I_{ab} + I_{bc} + I_{ca}$.
Finally let
\[
  J_{ab} = \langle \alpha_{i,j} \beta_{j,k} - \alpha_{i,1} \beta_{1,k} \mid 1 \leq i,j,k \leq n \rangle
\]
and similarly $J_{bc}$, $J_{ca}$; set $J = J_{ab} + J_{bc} + J_{ca}$.
It is easy to see that $\mult_n$ is annihilated by $K = I + I' + J$.
We claim that this is equal to $\mult_n^\perp$.

Let $\Theta \in \mult_n^\perp$ be a homogeneous element.
Since $\mult_n$ is concise, we must have $\deg \Theta > 1$.
Using the generators of $I$ we can eliminate, modulo $K$,
every term of $\Theta$ that has more than one $\alpha$, $\beta$, or $\gamma$ factor.
If $\Theta$ is not already zero modulo $K$ then we must have $\deg \Theta \leq 3$.
First, consider the case $\deg \Theta = 2$.
Every term appearing in $\Theta$ must have the form
$\alpha_{i,j} \beta_{j,k}$, $\beta_{j,k} \gamma_{k,i}$, or $\gamma_{k,i} \alpha_{i,j}$;
all other terms can be eliminated modulo $K$ using the generators of $I'$.
Using the generators of $J$, we can write, modulo $K$,
\[
  \Theta = \sum_{i,k} p_{i,k} \alpha_{i,1}\beta_{1,k} +
    \sum_{j,i} q_{j,i} \beta_{j,1}\gamma_{1,i} + \sum_{k,j} r_{k,j} \gamma_{k,1}\alpha_{1,j}.
\]
Then
\[
  \Theta \mult_n = \sum_{i,k} p_{i,k} c_{k,i} + \sum_{j,i} q_{j,i} a_{i,j} + \sum_{k,j} r_{k,j} b_{j,k} .
\]
Since $\Theta \mult_n = 0$ we must have every coefficient $p_{i,k} = q_{j,i} = r_{k,j} = 0$ for all $i,j,k$.
So $\Theta$ is zero modulo $K$.
This shows that $K$ and $\mult_n^\perp$ have the same degree $2$ elements.

Next consider the case $\deg \Theta = 3$.
Every term appearing in $\Theta$ must have the form $\alpha_{i,j} \beta_{j,k} \gamma_{k,i}$,
as all other terms can be eliminated using the generators of $I+I'$.
But $\alpha_{i,j} \beta_{j,k} \gamma_{k,i}$ is equal to $\alpha_{1,1} \beta_{1,1} \gamma_{1,1}$ modulo $J$.
So modulo $K$, we can write $\Theta = p \alpha_{1,1} \beta_{1,1} \gamma_{1,1}$ for some $p$.
Then
\[
  0 = \Theta \mult_n = (p \alpha_{1,1} \beta_{1,1} \gamma_{1,1})(a_{1,1} b_{1,1} c_{1,1}) = p.
\]
Therefore $\Theta \in K$, which shows that $K$ and $\mult_n^\perp$ have the same degree $3$ elements.

Finally since $\deg \mult_n = 3$, $\mult_n^\perp$ contains all elements of degree $\geq 4$, and so does $I$.
Hence $\mult_n^\perp$ and $K$ both have the same elements of degree $\geq 4$ (namely, all of them).
This finishes the proof that $\mult_n^\perp = K$, as claimed.

Next we claim that the generators we have listed for $I$, $I'$, and $J$ form a minimal set of generators for $\mult_n^\perp$.
But this is clear since they are all of total degree $2$ and linearly independent.
Therefore $\mult_n^\perp$ is generated in multidegrees $(2,0,0)$, $(1,1,0)$, and cyclic rotations of these.

Now we compute the ranks of catalecticants of $\mult_n$.
These are the same as the flattenings usually considered in the tensor literature.
Let $(0,0,0) \leq \ba = (a_1,a_2,a_3) \leq (1,1,1)$.
If $\ba = (0,0,0)$ or $\ba = (1,1,1)$ then $\rank C^{\ba}_{\mult_n} = 1$.
If $\ba = (1,0,0)$ then $\rank C^{\ba}_{\mult_n} = n^2$ since $C^{(1,0,0)}_{\mult_n}(\beta_{i,k}\gamma_{k,j}) = a_{i,j}$,
showing that $C^{(1,0,0)}_{\mult_n}$ is surjective.
Similarly the $(0,1,0)$ and $(0,0,1)$ catalecticants have rank $n^2$.
Then their transposes, the $(0,1,1)$, $(1,0,1)$, and $(1,1,0)$ catalecticants, also have rank $n^2$.
Therefore $\ell(A^{\mult_n}) = 2 + 6n^2$.

By Theorem~\ref{thm: multihomogeneous polynomial ranestad-schreyer} we have
\[
  \r_{MH}(\mult_n) \geq \frac{\ell(A^{\mult_n})}{2^3} = \frac{1}{4} + \frac{3}{4} n^2 .
\]
This is actually worse than the simple catalecticant bound \eqref{eq: multihomogeneous polynomial catalecticant bound},
$\r_{MH}(\mult_n) \geq n^2$.
This illustrates that Theorem~\ref{thm: multihomogeneous polynomial ranestad-schreyer},
like Theorem~\ref{thm: ranestad-schreyer}, gives better bounds when $F^\perp$ is generated in small degrees (or multidegrees)
relative to the degree (or multidegree) of $F$ itself; see \cite{Buczynska:2013kx}
for more on comparisons between the degrees of generators of $F^\perp$ and the degree of $F$.
Regarding matrix multiplication, the best known lower bound at this time is $\r_{MH}(\mult_n) \geq 3n^2 - o(n^2)$,
due to Landsberg \cite{Landsberg:2014yq}.

We may also simply regard $\mult_n$ as a homogeneous polynomial of (total) degree $3$, in $3n^2$ variables,
and ask for its classical Waring rank.
Theorem~\ref{thm: ranestad-schreyer} gives
\[
  \r(\mult_n) \geq \frac{\ell(A^{\mult_n})}{2} = 1 + 3n^2,
\]
since $\mult_n^\perp$ is generated by quadrics.
\end{example}

\subsection{Generalized rank}

%

\begin{theorem}\label{thm: ranestad schreyer for carlini k-ary rank}
Let $F \in S^d V$. Fix $k < n = \dim V$.
Let $\r_k(F)$ be the least number of terms in a decomposition of $F$ as a sum of forms each depending on $k$ or fewer variables.
Let $F^\perp = (H_1,\dotsc,H_t)$, $\deg H_i = d_i$, $d_1 \leq \dotsb \leq d_t$,
and suppose $j$ is such that $(H_1,\dotsc,H_j)$ defines a zero-dimensional affine variety.
Necessarily $j \geq n$.
Then $\r_k(F) \geq \ell(A^F) / (d_{j-k+1} \dotsm d_j)$.
\end{theorem}

\begin{proof}
Fix a decomposition $F = F_1 + \dotsb + F_r$, each $F_i$ depending on the variables in $W_i$, a $k$-dimensional subspace.
Let $C$ be the reduced union of the $\PP W_i$.
Every component of $C$ is $(k-1)$-dimensional.

Let $I = (H_1,\dotsc,H_j)$.
For each $i=1,\dotsc,j$, let $G_i \in I_{d_i}$ be general.
Then $I = (G_1,\dotsc,G_j)$.
For each $i=1,\dotsc,j$, let $C_i = C \cap V(G_i,\dotsc,G_j)$
and let $I(i) = (G_1,\dotsc,G_{i-1})$.

Since the linear series $I_{d_j}$ has no basepoints, by generality of $G_j$ and Bertini's theorem,
$V(G_j)$ does not contain any component of $C$.
So every component of $C_j$ is $(k-2)$-dimensional.
Suppose inductively that every component of $C_i$ is $(k-j+i-2)$-dimensional.
Note that $V(G_1,\dotsc,G_{i-1})$ is disjoint from $V(G_i,\dotsc,G_j)$,
so the linear series $I(i)_{d_{i-1}}$ (the degree $d_{i-1}$ elements of the ideal $(G_1,\dotsc,G_{i-1})$)
has no basepoints on $C_i$.
By generality of $G_{i-1}$ and Bertini, then, $V(G_{i-1})$ does not include any component of $C_i$,
so every component of $C_{i-1}$ has dimension exactly one less than the dimension of $C_i$.

Therefore $C_{j-k+1}$ is empty.
Let $J = (G_{j-k+1},\dotsc,G_j)$ and let $\widehat{Z} = V(J)$ be the affine variety defined by $J$.
Let $\widehat{C} = \bigcup W_i$ be the reduced union of the $W_i$.
Note that every component of $\widehat{C}$ has dimension $k$,
while since $J$ has $k$ generators, every component of $\widehat{Z}$ has codimension at most $k$.
We have just seen that $\widehat{Z}$ intersects $\widehat{C}$ only at the origin;
thus $\widehat{Z}$ is a complete intersection,
in particular $\widehat{C}$ and $\widehat{Z}$ intersect properly.
Since $J \subset F^\perp$, $\Spec A^F \subset \widehat{Z}$.
And by the Apolarity Lemma, $I(C) \subset F^\perp$, so $\Spec A^F \subset \widehat{C}$.
Therefore
\[
  \ell(A^F) = \ell(\Spec A^F) \leq \deg(\widehat{C} \cap \widehat{Z}) = \deg(\widehat{C}) \deg(\widehat{Z})
    = r (d_{j-k+1} \dotsm d_j),
\]
as claimed.
\end{proof}


\begin{conjecture}
Let $F = x_1^{d_1} \dotsm x_n^{d_n}$ with $0 < d_1 \leq \dotsb \leq d_n$.
Theorem~\ref{thm: ranestad schreyer for carlini k-ary rank} gives
\[
  \r_k(x_1^{d_1} \dotsm x_n^{d_n}) \geq \frac{(d_1+1) \dotsm (d_n+1)}{(d_{n-k+1}+1) \dotsm (d_n+1)} = (d_1+1) \dotsm (d_{n-k}+1) .
\]
Conversely, let
\[
  x_1^{d_1} \dotsm x_{n-k+1}^{d_{n-k+1}} = \sum_{i=1}^m \ell_i^{d_1+\dotsb+d_{n-k+1}}
\]
be a Waring decomposition, $m = \r(x_1^{d_1} \dotsm x_{n-k+1}^{d_{n-k+1}}) = (d_2+1)\dotsm(d_{n-k+1}+1)$.
Then
\[
  x_1^{d_1} \dotsm x_n^{d_n} = \sum_{i=1}^m \ell_i^{d_1+\dotsb+d_{n-k+1}} x_{n-k+2}^{d_{n-k+2}} \dotsm x_n^{d_n} ,
\]
where each term depends essentially on $k$ variables.
Therefore
\[
  \r_k(x_1^{d_1} \dotsm x_n^{d_n})
    \leq \r(x_1^{d_1} \dotsm x_{n-k+1}^{d_{n-k+1}})
    = (d_2+1) \dotsm (d_{n-k+1}+1).
\]
We conjecture that this is an equality.

It is trivially true when $k=1$
(ordinary Waring rank, reducing to the theorem of Carlini--Catalisano--Geramita \cite{Carlini20125})
or $k=n$.
It is true whenever $d_1 = \dotsb = d_{n-k+1} \leq d_{n-k+2} \leq \dotsb \leq d_n$
because the upper and lower bounds are equal:
\[
  \r_k\big( (x_1 \dotsm x_{n-k+1})^d x_{n-k+2}^{d_{n-k+2}} \dotsm x_n^{d_n} \big)
    = \r\big( (x_1 \dotsm x_{n-k+1})^d \big) = (d+1)^{n-k} .
\]
\end{conjecture}

\bigskip

\section*{Acknowledgments}

I am grateful to a number of people for helpful conversations, suggestions, and pointing out errors,
including
Jarek Buczy\'nski, Harm Derksen, Chris Francisco, Anthony Iarrobino,
Nathan Ilten, Jeremy Martin, Uwe Nagel, Luke Oeding, and Adam Van Tuyl.
The software \textit{Macaulay2} \cite{M2} was invaluable in computing examples.

\bibliography{../../biblio}

\providecommand{\bysame}{\leavevmode\hbox to3em{\hrulefill}\thinspace}
\providecommand{\MR}{\relax\ifhmode\unskip\space\fi MR }
\providecommand{\MRhref}[2]{%
  \href{http://www.ams.org/mathscinet-getitem?mr=#1}{#2}
}
\providecommand{\href}[2]{#2}
\begin{thebibliography}{BCMT10}

\bibitem[AB11]{MR2729216}
Enrique Arrondo and Alessandra Bernardi, \emph{On the variety parameterizing
  completely decomposable polynomials}, J. Pure Appl. Algebra \textbf{215}
  (2011), no.~3, 201--220. \MR{2729216 (2011h:14066)}

\bibitem[AB12]{Abo:2011lr}
Hirotachi Abo and Maria~Chiara Brambilla, \emph{New examples of defective
  secant varieties of {S}egre--{V}eronese varieties}, Collectanea Mathematica
  \textbf{63} (2012), no.~3, 287--297, DOI: 10.1007/s13348-011-0039-7.

\bibitem[Abo14]{MR3166068}
Hirotachi Abo, \emph{Varieties of completely decomposable forms and their
  secants}, J. Algebra \textbf{403} (2014), 135--153. \MR{3166068}

\bibitem[AH95]{MR1311347}
J.~Alexander and A.~Hirschowitz, \emph{Polynomial interpolation in several
  variables}, J. Algebraic Geom. \textbf{4} (1995), no.~2, 201--222.
  \MR{1311347 (96f:14065)}

\bibitem[AS13]{MR3003310}
Jeaman Ahn and Yong~Su Shin, \emph{On {G}orenstein sequences of socle degrees 4
  and 5}, J. Pure Appl. Algebra \textbf{217} (2013), no.~5, 854--862.
  \MR{3003310}

\bibitem[BB12]{MR2945601}
Edoardo Ballico and Alessandra Bernardi, \emph{Decomposition of homogeneous
  polynomials with low rank}, Math. Z. \textbf{271} (2012), no.~3-4,
  1141--1149. \MR{2945601}

\bibitem[BB13a]{MR3056286}
\bysame, \emph{Unique decomposition for a polynomial of low rank}, Ann. Polon.
  Math. \textbf{108} (2013), no.~3, 219--224. \MR{3056286}

\bibitem[BB13b]{Buczynska:2013zh}
Weronika Buczy{\'n}ska and Jaros{\l}aw Buczy{\'n}ski, \emph{On differences
  between the border rank and the smoothable rank of a polynomial},
  \href{http://arxiv.org/abs/1305.1726}{\nolinkurl{arXiv:1305.1726}} [math.AG],
  May 2013.

\bibitem[BB14]{MR3121848}
Weronika Buczy{\'n}ska and Jaros{\l}aw Buczy{\'n}ski, \emph{Secant varieties to
  high degree {V}eronese reembeddings, catalecticant matrices and smoothable
  {G}orenstein schemes}, J. Algebraic Geom. \textbf{23} (2014), no.~1, 63--90.
  \MR{3121848}

\bibitem[BBCM11]{MR2895192}
Alessandra Bernardi, J{\'e}r{\^o}me Brachat, Pierre Comon, and Bernard
  Mourrain, \emph{Multihomogeneous polynomial decomposition using moment
  matrices}, I{SSAC} 2011---{P}roceedings of the 36th {I}nternational
  {S}ymposium on {S}ymbolic and {A}lgebraic {C}omputation, ACM, New York, 2011,
  pp.~35--42. \MR{2895192}

\bibitem[BBKT13]{Buczynska:2013kx}
Weronika Buczy{\'n}ska, Jaros{\l}aw Buczy{\'n}ski, Johannes Kleppe, and Zach
  Teitler, \emph{Apolarity and direct sum decomposability of polynomials},
  \href{http://arxiv.org/abs/1307.3314}{\nolinkurl{arXiv:1307.3314}} [math.AG],
  Jul 2013.

\bibitem[BBM12]{Bernardi:2012fk}
Alessandra Bernardi, J{{\'e}}r{{\^o}}me Brachat, and Bernard Mourrain, \emph{A
  comparison of different notions of ranks of symmetric tensors},
  \href{http://arxiv.org/abs/1210.8169v2}{\nolinkurl{arXiv:1210.8169}}
  [math.AG], Oct 2012.

\bibitem[BBS08]{MR2383331}
A.~Bia{\l}ynicki-Birula and A.~Schinzel, \emph{Representations of multivariate
  polynomials by sums of univariate polynomials in linear forms}, Colloq. Math.
  \textbf{112} (2008), no.~2, 201--233. \MR{2383331 (2009b:12006)}

\bibitem[BBT13]{MR3017012}
Weronika Buczy{\'n}ska, Jaros{\l}aw Buczy{\'n}ski, and Zach Teitler,
  \emph{Waring decompositions of monomials}, J. Algebra \textbf{378} (2013),
  45--57. \MR{3017012}

\bibitem[BCMT10]{MR2736103}
Jerome Brachat, Pierre Comon, Bernard Mourrain, and Elias Tsigaridas,
  \emph{Symmetric tensor decomposition}, Linear Algebra Appl. \textbf{433}
  (2010), no.~11-12, 1851--1872. \MR{2736103 (2011k:15047)}

\bibitem[BGI11]{Bernardi201134}
Alessandra Bernardi, Alessandro Gimigliano, and Monica Id{\`a}, \emph{Computing
  symmetric rank for symmetric tensors}, Journal of Symbolic Computation
  \textbf{46} (2011), no.~1, 34--53.

\bibitem[BI92]{MR1172667}
David Bernstein and Anthony Iarrobino, \emph{A nonunimodal graded {G}orenstein
  {A}rtin algebra in codimension five}, Comm. Algebra \textbf{20} (1992),
  no.~8, 2323--2336. \MR{1172667 (93i:13012)}

\bibitem[BL94]{MR1227512}
Mats Boij and Dan Laksov, \emph{Nonunimodality of graded {G}orenstein {A}rtin
  algebras}, Proc. Amer. Math. Soc. \textbf{120} (1994), no.~4, 1083--1092.
  \MR{1227512 (94g:13008)}

\bibitem[BL13]{MR2996361}
Jaros{\l}aw Buczy{\'n}ski and J.~M. Landsberg, \emph{Ranks of tensors and a
  generalization of secant varieties}, Linear Algebra Appl. \textbf{438}
  (2013), no.~2, 668--689. \MR{2996361}

\bibitem[BO08]{MR2387598}
Maria~Chiara Brambilla and Giorgio Ottaviani, \emph{On the
  {A}lexander-{H}irschowitz theorem}, J. Pure Appl. Algebra \textbf{212}
  (2008), no.~5, 1229--1251. \MR{MR2387598 (2008m:14104)}

\bibitem[Boi95]{MR1311776}
Mats Boij, \emph{Graded {G}orenstein {A}rtin algebras whose {H}ilbert functions
  have a large number of valleys}, Comm. Algebra \textbf{23} (1995), no.~1,
  97--103. \MR{1311776 (96h:13040)}

\bibitem[BP13]{Ballico:2013sf}
Edoardo Ballico and Alessandro~De Paris, \emph{Generic power sum decompositions
  and bounds for the {W}aring rank},
  \href{http://arxiv.org/abs/1312.3494}{\nolinkurl{arXiv:1312.3494}} [math.AG],
  Dec 2013.

\bibitem[BR13]{MR2996880}
Alessandra Bernardi and Kristian Ranestad, \emph{On the cactus rank of cubic
  forms}, J. Symbolic Comput. \textbf{50} (2013), 291--297. \MR{2996880}

\bibitem[Bro33]{Bronowski:1933fk}
J.~Bronowski, \emph{The sums of powers as simultaneous canonical expressions},
  Math. Proc. Cambridge Philos. Soc. \textbf{29} (1933), no.~2, 245--256.

\bibitem[BT14]{Blekherman:2014eq}
Greg Blekherman and Zach Teitler, \emph{On maximum, typical, and generic
  ranks}, \href{http://arxiv.org/abs/1402.2371}{arXiv:1402.2371} [math.AG], Feb
  2014.

\bibitem[Car05]{MR2202247}
Enrico Carlini, \emph{Codimension one decompositions and {C}how varieties},
  Projective varieties with unexpected properties, Walter de Gruyter GmbH \&
  Co. KG, Berlin, 2005, pp.~67--79. \MR{2202247 (2007f:14056)}

\bibitem[Car06a]{MR2184818}
\bysame, \emph{Binary decompositions and varieties of sums of binaries}, J.
  Pure Appl. Algebra \textbf{204} (2006), no.~2, 380--388. \MR{2184818
  (2006j:14070)}

\bibitem[Car06b]{MR2279854}
\bysame, \emph{Reducing the number of variables of a polynomial}, Algebraic
  geometry and geometric modeling, Math. Vis., Springer, Berlin, 2006,
  pp.~237--247. \MR{2279854 (2007i:12010)}

\bibitem[CC03]{MR1998391}
Enrico Carlini and Jaydeep Chipalkatti, \emph{On {W}aring's problem for several
  algebraic forms}, Comment. Math. Helv. \textbf{78} (2003), no.~3, 494--517.
  \MR{1998391 (2005b:14097)}

\bibitem[CCC14]{Carlini:2014rr}
Enrico Carlini, Maria~Virginia Catalisano, and Luca Chiantini, \emph{Progress
  on the symmetric {S}trassen conjecture},
  \href{http://arxiv.org/abs/1405.3721}{\nolinkurl{arXiv:1405.3721}} [math.AG],
  May 2014.

\bibitem[CCG12]{Carlini20125}
Enrico Carlini, Maria~Virginia Catalisano, and Anthony~V. Geramita, \emph{The
  solution to the {W}aring problem for monomials and the sum of coprime
  monomials}, J. Algebra \textbf{370} (2012), 5--14.

\bibitem[CGLM08]{MR2447451}
Pierre Comon, Gene Golub, Lek-Heng Lim, and Bernard Mourrain, \emph{Symmetric
  tensors and symmetric tensor rank}, SIAM J. Matrix Anal. Appl. \textbf{30}
  (2008), no.~3, 1254--1279. \MR{2447451 (2009i:15039)}

\bibitem[Cha01]{MR1813598}
Karen~A. Chandler, \emph{A brief proof of a maximal rank theorem for generic
  double points in projective space}, Trans. Amer. Math. Soc. \textbf{353}
  (2001), no.~5, 1907--1920 (electronic). \MR{1813598 (2002i:14046)}

\bibitem[Cha02]{MR1943904}
\bysame, \emph{Linear systems of cubics singular at general points of
  projective space}, Compositio Math. \textbf{134} (2002), no.~3, 269--282.
  \MR{1943904 (2003i:13023)}

\bibitem[CM96]{comonmour96}
P.~COMON and B.~MOURRAIN, \emph{Decomposition of quantics in sums of powers of
  linear forms}, Signal Processing, Elsevier 53(2), 1996.

\bibitem[CO13]{Carlini:2013vl}
Enrico Carlini and Alessandro Oneto, \emph{Monomials as sums of {$k$}-th powers
  of forms}, \href{http://arxiv.org/abs/1305.4553}{\nolinkurl{arXiv:1305.4553}}
  [math.AC], May 2013.

\bibitem[CS11]{MR2754189}
Gonzalo Comas and Malena Seiguer, \emph{On the rank of a binary form}, Found.
  Comput. Math. \textbf{11} (2011), no.~1, 65--78. \MR{2754189}

\bibitem[Der13]{Derksen:2013sf}
Harm Derksen, \emph{On the nuclear norm and the singular value decomposition of
  tensors}, \href{http://arxiv.org/abs/1308.3860}{\nolinkurl{arXiv:1308.3860}}
  [math.OC], Aug 2013.

\bibitem[Die10]{MR2744811}
Reinhard Diestel, \emph{Graph theory}, fourth ed., Graduate Texts in
  Mathematics, vol. 173, Springer, Heidelberg, 2010. \MR{2744811 (2011m:05002)}

\bibitem[DSS09]{MR2723140}
Mathias Drton, Bernd Sturmfels, and Seth Sullivant, \emph{Lectures on algebraic
  statistics}, Oberwolfach Seminars, vol.~39, Birkh\"auser Verlag, Basel, 2009.
  \MR{2723140 (2012d:62004)}

\bibitem[ER93]{MR1215329}
Richard Ehrenborg and Gian-Carlo Rota, \emph{Apolarity and canonical forms for
  homogeneous polynomials}, European J. Combin. \textbf{14} (1993), no.~3,
  157--181. \MR{1215329 (94e:15062)}

\bibitem[Fon02]{MR1929406}
Claudio Fontanari, \emph{On {W}aring's problem for many forms and {G}rassmann
  defective varieties}, J. Pure Appl. Algebra \textbf{174} (2002), no.~3,
  243--247. \MR{1929406 (2003h:14079)}

\bibitem[Fon06]{MR2246903}
\bysame, \emph{On {W}aring's problem for partially symmetric tensors.
  {V}ariations on a theme of {M}ella}, Ann. Univ. Ferrara Sez. VII Sci. Mat.
  \textbf{52} (2006), no.~1, 37--43. \MR{2246903 (2007g:14064)}

\bibitem[FOS12]{MR2935563}
Ralf Fr{\"o}berg, Giorgio Ottaviani, and Boris Shapiro, \emph{On the {W}aring
  problem for polynomial rings}, Proc. Natl. Acad. Sci. USA \textbf{109}
  (2012), no.~15, 5600--5602. \MR{2935563}

\bibitem[Ger96]{Geramita}
Anthony~V. Geramita, \emph{Inverse systems of fat points: {W}aring's problem,
  secant varieties of {V}eronese varieties and parameter spaces for
  {G}orenstein ideals}, The {C}urves {S}eminar at {Q}ueen's, {V}ol.\ {X}
  ({K}ingston, {ON}, 1995), Queen's Papers in Pure and Appl. Math., vol. 102,
  Queen's Univ., Kingston, ON, 1996, pp.~2--114.

\bibitem[GHMS07]{MR2292384}
Anthony~V. Geramita, Tadahito Harima, Juan~C. Migliore, and Yong~Su Shin,
  \emph{The {H}ilbert function of a level algebra}, Mem. Amer. Math. Soc.
  \textbf{186} (2007), no.~872, vi+139. \MR{2292384 (2007k:13033)}

\bibitem[Gly10]{MR2673027}
David~G. Glynn, \emph{The permanent of a square matrix}, European J. Combin.
  \textbf{31} (2010), no.~7, 1887--1891. \MR{2673027 (2011h:15010)}

\bibitem[GS]{M2}
Daniel~R. Grayson and Michael~E. Stillman, \emph{Macaulay2, a software system
  for research in algebraic geometry}, Available at
  \href{http://www.math.uiuc.edu/Macaulay2/}%
  {http://www.math.uiuc.edu/Macaulay2/}.

\bibitem[Gun86]{Gundelfinger:1886fk}
S.~Gundelfinger, \emph{Zur theorie der bin{\"a}ren formen}, J. Reine Angew.
  Math. \textbf{100} (1886), 413--424.

\bibitem[GY10]{MR2850282}
John~Hilton Grace and Alfred Young, \emph{The algebra of invariants}, Cambridge
  Library Collection, Cambridge University Press, Cambridge, 2010, Reprint of
  the 1903 original. \MR{2850282 (2012g:01022)}

\bibitem[Iar84]{MR748843}
Anthony Iarrobino, \emph{Compressed algebras: {A}rtin algebras having given
  socle degrees and maximal length}, Trans. Amer. Math. Soc. \textbf{285}
  (1984), no.~1, 337--378. \MR{748843 (85j:13030)}

\bibitem[IK99]{MR1735271}
Anthony Iarrobino and Vassil Kanev, \emph{Power sums, {G}orenstein algebras,
  and determinantal loci}, Lecture Notes in Mathematics, vol. 1721,
  Springer-Verlag, Berlin, 1999, Appendix C by Iarrobino and Steven L. Kleiman.
  \MR{1735271 (2001d:14056)}

\bibitem[Jel13]{Jelisiejew:2013fk}
Joachim Jelisiejew, \emph{An upper bound for the {W}aring rank of a form},
  \href{http://arxiv.org/abs/1305.6957}{\nolinkurl{arXiv:1305.6957}} [math.AC],
  May 2013.

\bibitem[KB09]{MR2535056}
Tamara~G. Kolda and Brett~W. Bader, \emph{Tensor decompositions and
  applications}, SIAM Rev. \textbf{51} (2009), no.~3, 455--500. \MR{2535056
  (2010j:15027)}

\bibitem[Kir08]{MR2386244}
George~A. Kirkup, \emph{Minimal primes over permanental ideals}, Trans. Amer.
  Math. Soc. \textbf{360} (2008), no.~7, 3751--3770. \MR{2386244 (2009a:13017)}

\bibitem[Kle99]{Kleppe:1999fk}
Johannes Kleppe, \emph{Representing a homogenous polynomial as a sum of powers
  of linear forms}, Master's thesis, University of Oslo, 1999,
  \url{http://folk.uio.no/johannkl/kleppe-master.pdf}.

\bibitem[Kun86]{MR859177}
Joseph P.~S. Kung, \emph{Gundelfinger's theorem on binary forms}, Stud. Appl.
  Math. \textbf{75} (1986), no.~2, 163--169. \MR{859177 (87m:11020)}

\bibitem[Lan12]{MR2865915}
J.~M. Landsberg, \emph{Tensors: geometry and applications}, Graduate Studies in
  Mathematics, vol. 128, American Mathematical Society, Providence, RI, 2012.
  \MR{2865915}

\bibitem[Lan13]{Landsberg:2013ys}
J.M. Landsberg, \emph{Geometric complexity theory: an introduction for
  geometers},
  \href{http://arxiv.org/abs/1305.7387}{\nolinkurl{arXiv:1305.7387}} [math.AG],
  May 2013.

\bibitem[Lan14]{Landsberg:2014yq}
J.~M. Landsberg, \emph{New lower bounds for the rank of matrix multiplication},
  SIAM Journal on Computing \textbf{43} (2014), no.~1, 144--149,
  DOI:10.1137/120880276.

\bibitem[LO12]{LO-apolar-catalecticant}
J.~Landsberg and Giorgio Ottaviani, \emph{Equations for secant varieties of
  {V}eronese and other varieties}, Annali di Matematica Pura ed Applicata
  (2012), 1--38.

\bibitem[LS00]{MR1777172}
Reinhard~C. Laubenbacher and Irena Swanson, \emph{Permanental ideals}, J.
  Symbolic Comput. \textbf{30} (2000), no.~2, 195--205. \MR{1777172
  (2001i:13039)}

\bibitem[LT10]{Landsberg:2009yq}
J.M. Landsberg and Zach Teitler, \emph{On the ranks and border ranks of
  symmetric tensors}, Found.\ Comp.\ Math. \textbf{10} (2010), no.~3, 339--366.

\bibitem[LV99]{Lavila-Vidal:1999sp}
Olga Lavila-Vidal, \emph{On the diagonals of a {R}ees algebra}, Ph.D. thesis,
  Universitat de Barcelona, 1999.

\bibitem[Mel06]{MR2238925}
Massimiliano Mella, \emph{Singularities of linear systems and the {W}aring
  problem}, Trans. Amer. Math. Soc. \textbf{358} (2006), no.~12, 5523--5538
  (electronic). \MR{2238925 (2007h:14059)}

\bibitem[Mel09]{MR2439429}
\bysame, \emph{Base loci of linear systems and the {W}aring problem}, Proc.
  Amer. Math. Soc. \textbf{137} (2009), no.~1, 91--98. \MR{2439429
  (2009g:14072)}

\bibitem[Mil13]{Miller:2013lo}
Jeff Miller, \emph{Earliest known uses of some of the words of mathematics
  (c)}, Dec 2013, \url{http://jeff560.tripod.com/c.html}.

\bibitem[MS05]{MR2110098}
Ezra Miller and Bernd Sturmfels, \emph{Combinatorial commutative algebra},
  Graduate Texts in Mathematics, vol. 227, Springer-Verlag, New York, 2005.
  \MR{2110098 (2006d:13001)}

\bibitem[OO13]{Oeding2013}
Luke Oeding and Giorgio Ottaviani, \emph{Eigenvectors of tensors and algorithms
  for {W}aring decomposition}, Journal of Symbolic Computation \textbf{54}
  (2013), 9--35.

\bibitem[Par13]{Paris:2013fk}
Alessandro~De Paris, \emph{A proof that the maximal rank for plane quartics is
  seven},
  \href{http://arxiv.org/abs/1309.6475}{\nolinkurl{http://arxiv.org/abs/1309.6475}},
  Sep 2013.

\bibitem[Pos12]{MR2886162}
Elisa Postinghel, \emph{A new proof of the {A}lexander-{H}irschowitz
  interpolation theorem}, Ann. Mat. Pura Appl. (4) \textbf{191} (2012), no.~1,
  77--94. \MR{2886162}

\bibitem[Rez92]{MR1096187}
Bruce Reznick, \emph{Sums of even powers of real linear forms}, Mem. Amer.
  Math. Soc. \textbf{96} (1992), no.~463, viii+155. \MR{1096187 (93h:11043)}

\bibitem[Rez13a]{Reznick:2013yq}
Bruce Reznick, \emph{Forms as sums of powers of lower degree forms}, Slides
  from a talk at the SIAM Conference on Applied Algebraic Geometry, Fort
  Collins, Colorado, 2013.
  \url{http://www.math.uiuc.edu/~reznick/8213f-pm.pdf}, Aug 2013.

\bibitem[Rez13b]{Reznick:2013uq}
\bysame, \emph{On the length of binary forms}, Quadratic and Higher Degree
  Forms (New York) (K.~Alladi, M.~Bhargava, D.~Savitt, and P.~Tiep, eds.),
  Developments in Math., vol.~31, Springer, 2013, pp.~207--232.

\bibitem[Rob98]{MR1686450}
Paul~C. Roberts, \emph{Multiplicities and {C}hern classes in local algebra},
  Cambridge Tracts in Mathematics, vol. 133, Cambridge University Press,
  Cambridge, 1998. \MR{1686450 (2001a:13029)}

\bibitem[RS00]{MR1780430}
Kristian Ranestad and Frank-Olaf Schreyer, \emph{Varieties of sums of powers},
  J. Reine Angew. Math. \textbf{525} (2000), 147--181. \MR{1780430
  (2001m:14009)}

\bibitem[RS11]{MR2842085}
\bysame, \emph{On the rank of a symmetric form}, J. Algebra \textbf{346}
  (2011), 340--342. \MR{2842085}

\bibitem[Rys63]{MR0150048}
Herbert~John Ryser, \emph{Combinatorial mathematics}, The Carus Mathematical
  Monographs, No. 14, Published by The Mathematical Association of America,
  1963. \MR{0150048 (27 \#51)}

\bibitem[Sha13]{Shafiei:2013fk}
Masoumeh~Sepideh Shafiei, \emph{Apolarity for determinants and permanents of
  generic symmetric matrices},
  \href{http://arxiv.org/abs/1303.1860}{\nolinkurl{arXiv:1303.1860}} [math.AC],
  Mar 2013.

\bibitem[Sha14]{Shafiei:ud}
\bysame, \emph{Apolarity for determinants and permanents of generic matrices},
  To appear in Journal of Commutative Algebra, 2014.

\bibitem[Shi12]{MR2997451}
Yong~Su Shin, \emph{Secants to the variety of completely reducible forms and
  the {H}ilbert function of the union of star-configurations}, J. Algebra Appl.
  \textbf{11} (2012), no.~6, 1250109, 27. \MR{2997451}

\bibitem[SS12]{MR2968919}
Sumi Seo and Hema Srinivasan, \emph{On unimodality of {H}ilbert functions of
  {G}orenstein {A}rtin algebras of embedding dimension four}, Comm. Algebra
  \textbf{40} (2012), no.~8, 2893--2905. \MR{2968919}

\bibitem[Sta78]{MR0485835}
Richard~P. Stanley, \emph{Hilbert functions of graded algebras}, Advances in
  Math. \textbf{28} (1978), no.~1, 57--83.

\bibitem[Stu02]{MR1925796}
Bernd Sturmfels, \emph{Solving systems of polynomial equations}, CBMS Regional
  Conference Series in Mathematics, vol.~97, Published for the Conference Board
  of the Mathematical Sciences, Washington, DC; by the American Mathematical
  Society, Providence, RI, 2002. \MR{1925796 (2003i:13037)}

\bibitem[Syl51a]{Sylvester:1851kx}
J.J. Sylvester, \emph{An essay on canonical forms, supplement to a sketch of a
  memoir on elimination, transformation and canonical forms}, originally
  published by George Bell, Fleet Street, London, 1851. Paper 34 in {\it
  Mathematical Papers}, Vol.~1, Chelsea, New York, 1973, originally published
  by Cambridge University Press in 1904., 1851.

\bibitem[Syl51b]{Sylvester:1851wd}
\bysame, \emph{On a remarkable discovery in the theory of canonical forms and
  of hyperdeterminants}, originally published in Philosophical Magazine,
  vol.~2, 1851, pp. 391--410. Paper 41 in {\it Mathematical Papers}, Vol.~1,
  Chelsea, New York, 1973, originally published by Cambridge University Press
  in 1904., 1851.

\bibitem[Syl86]{Sylvester:1886uq}
\bysame, \emph{Sur une extension d'un th{\'e}or{\`e}me de {C}lebsch relatif aux
  courbes du quatri{\`e}me degr{\'e}}, C.R. Acad. Sci. \textbf{102} (1886),
  1532--1534.

\bibitem[Tor13]{Torrance:2013qf}
Douglas~A. Torrance, \emph{Nondefective secant varieties of varieties of
  completely decomposable forms}, Ph.D. thesis, U. Idaho, 2013,
  \href{http://arxiv.org/abs/1306.1293}{arXiv:1306.1293} [math.AG].

\bibitem[Tuy01]{Tuyl:2001cr}
Adam~Van Tuyl, \emph{Sets of {P}oints in {M}ulti-{P}rojective {S}paces and
  their {H}ilbert {F}unction}, Ph.D. thesis, Queen's University, 2001.

\bibitem[TW13]{Teitler:2013uq}
Zach Teitler and Alex Woo, \emph{Apolarity and reflection groups},
  \href{http://arxiv.org/abs/1304.7202}{\nolinkurl{arXiv:1304.7202}} [math.AG],
  Apr 2013.

\bibitem[Ven13]{Ventura:2013bh}
Emanuele Ventura, \emph{A note on the {W}aring ranks of reducible cubic forms},
  \href{http://arxiv.org/abs/1305.5394}{\nolinkurl{arXiv:1305.5394}} [math.AG],
  May 2013.

\bibitem[Woo14]{Woo:2014ix}
Youngho Woo, \emph{Some cases on strassen additive conjecture},
  \href{http://arxiv.org/abs/1406.2213}{\nolinkurl{arXiv:1406.2213}} [math.AG],
  Jun 2014.

\bibitem[Yu99]{MR1673948}
Yang Yu, \emph{The permanent rank of a matrix}, J. Combin. Theory Ser. A
  \textbf{85} (1999), no.~2, 237--242. \MR{1673948 (99j:15013)}

\end{thebibliography}

\bigskip

\end{document}